\documentclass[11pt,letterpaper,reqno]{amsart}

\usepackage[margin={1.25in, 1.5in}]{geometry}
\usepackage[bottom]{footmisc}

\usepackage[dvipsnames]{xcolor}
\def\TEXTCOLOR {black}
\def\PAGECOLOR {white}
\def\NOTECOLOR {Purple}
\definecolor{linkBlue}{rgb}{.05,.25,.6}

\color{\TEXTCOLOR}
\pagecolor{\PAGECOLOR}

\usepackage{amsmath}
\usepackage{amssymb}
\usepackage{bbm}
\usepackage{mathtools}
\usepackage{amsfonts}
\usepackage{comment}
\usepackage{enumitem}
\usepackage[mathscr]{euscript}
\usepackage{tikz-cd}
\usepackage[margin=1in]{caption}
\usepackage{mathrsfs}
\usepackage{appendix}
\usepackage{bibentry}
\usepackage{thmtools}
\usepackage{graphicx}
\graphicspath{{Images/}}
\usepackage{pgfplots}
\pgfplotsset{compat=1.18}

\usepackage[pdftex, psdextra]{hyperref}
\hypersetup{
	colorlinks=true,
	linkcolor=linkBlue,
	filecolor=magenta,
	urlcolor=linkBlue,
	citecolor=linkBlue}
\usepackage[nameinlink,noabbrev]{cleveref}

\theoremstyle{definition}
\newtheorem{definition}{Definition}[subsection]
\newtheorem*{definition*}{Definition}
\newtheorem{theorem}[definition]{Theorem}
\newtheorem*{theorem*}{Theorem}
\newtheorem{corollary}[definition]{Corollary}
\newtheorem{lemma}[definition]{Lemma}
\newtheorem*{lemma*}{Lemma}
\newtheorem{proposition}[definition]{Proposition}
\newtheorem{claim}[definition]{Claim}
\newtheorem{remark}[definition]{Remark}
\newtheorem{example}{Example}[subsection]
\newtheorem*{example*}{Example}
\newtheoremstyle{named}{}{}{\itshape}{}{\bfseries}{.}{.5em}{\thmnote{#3}}
\theoremstyle{named}
\newtheorem*{namedtheorem}{Theorem}

\let\emptyset\varnothing
\let\tilde\widetilde
\renewcommand{\bar}[1]{\overline{\hspace{-1pt}{#1}\hspace{-.2pt}}}
\let\hat\widehat
\newcommand{\de}{\text{d}}

\newcommand{\R}{\mathbb{R}}

\newcommand{\C}{\mathbb{C}}

\newcommand{\Z}{\mathbb{Z}}

\renewcommand{\P}{\mathbb{P}}

\newcommand{\from}{\,:\,}
\newcommand{\into}{\hookrightarrow}

\newcommand{\acts}{\mathbin{{\text{\raisebox{.75em}{\rotatebox{-90}{$\circlearrowright$}}}}}}

\newcommand{\im}{\text{im}\,}
\newcommand{\sm}{\smallsetminus}
\newcommand{\eps}{\varepsilon}

\newcommand{\e}{{\mathrm e}}

\newcommand{\norm}[1]{\lVert#1\rVert}

\newcommand{\CP}{\ensuremath{\mathbb {CP}}}

\DeclareMathOperator{\delbar}{\bar\partial}
\DeclareMathOperator{\del}{\partial}
\let\Re\relax
\let\Im\relax
\DeclareMathOperator{\Re}{Re}
\DeclareMathOperator{\Im}{Im}
\DeclareMathOperator{\tr}{tr}
\DeclareMathOperator{\rk}{rk}
\DeclareMathOperator{\diag}{diag}
\DeclareMathOperator{\GL}{GL}
\DeclareMathOperator{\SL}{SL}
\DeclareMathOperator{\U}{U}
\DeclareMathOperator{\SU}{SU}
\DeclareMathOperator{\SO}{SO}
\DeclareMathOperator{\Hom}{Hom}
\DeclareMathOperator{\End}{End}
\DeclareMathOperator{\Rep}{Rep}

\DeclareMathOperator*{\twoslash}{\hspace{-.2em}/\hspace{-.2em}/\hspace{-.2em}}
\DeclareMathOperator*{\fourslash}{\twoslash\twoslash}
\newcommand{\bx}{\mathbf{x}}
\newcommand{\by}{\mathbf{y}}

\usetikzlibrary{decorations.pathreplacing}

\title[Hyperk\"ahler Degenerations from Higgs to Hyperpolygon Spaces]{
    Hyperk\"ahler Degenerations from 
   Parabolic $\text{SL}(2,\C)$-Higgs Bundles Moduli Spaces on the Punctured Sphere to Hyperpolygon Spaces
}
\author{Laura Fredrickson}
\address{Department of Mathematics\\
University of Oregon, Eugene, OR 97403 USA}
\email{lfredric@uoregon.edu}

\author{Arya Yae}
\address{Department of Mathematics\\
University of Oregon, Eugene, OR 97403 USA}
\email{ayae@uoregon.edu}

\date{} 
\begin{document}

\begin{abstract}
    Complete hyperk\"ahler 4-manifolds of finite energy are grouped into ALE, ALF, ALG$^{(*)}$, ALH$^{(*)}$, each of these being further classified according to the Dynkin type of their noncompact end. A family of ALG-$D_4$ spaces are modeled by certain moduli spaces of strongly parabolic $\mathrm{SL}(2,\mathbb{C})$-Higgs bundles on the Riemann sphere with $n=4$ punctures. Meanwhile, a family of ALE-$D_4$ spaces are modeled by certain Nakajima quiver varieties known as $n=4$ hyperpolygon spaces. There is a map from hyperpolygon space to the moduli space of strong parabolic $\mathrm{SL}(2,\mathbb{C})$-Higgs bundles that is a diffeomorphism onto its open and dense image. We show that under a fine-tuned degenerate limit, the pullback of a family of ALG-$D_4$ metrics parameterized by $R$ converges pointwise to the ALE-$D_4$ metric as $R \to 0$.  While the connection to gravitational instantons occurs in the $n=4$ case, we prove our result for any finite $n$.
\end{abstract}

\maketitle

\tableofcontents

\newcommand{\app}{\text{app}}

\section{Introduction}
A gravitational instanton is a hyperkähler 4-manifold $(X,g,J_1,J_2,J_3,\omega_1,\omega_2,\omega_3)$ with finite energy
    $$E(g)=\int_X|\text{Rm}_g|^2\text{dvol}_g<\infty.$$
Let $r$ be a coordinate given by geodesic distance from a fixed point $p_0$ in $X$.  If we impose the slightly stronger condition $|\text{Rm}_g|^2\in O(r^{-2-\eps})$ as $r\to\infty$, G. Chen and X. Chen prove that $X$ must have a single noncompact end of type ALE, ALF, ALG, ALH according to the volume growth \cite{CC21}.  Namely, we have $\text{Vol}(B_r(p_0))\sim r^m$ for some $m\in\{1,2,3,4\}$, and ALE, ALF, ALG, ALH respectively correspond to $m=4, 3, 2, 1$.  When we relax the stronger curvature decay condition, the classification expands to include two additional possibilities ALG$^*$ and ALH$^*$.

The so-called Modularity Conjecture, attributed to Boalch \cite{aim},
posits that all of these gravitational instantons can be realized as gauge-theoretic moduli spaces. In particular, the ALG$^{(*)}$ gravitational instantons should be realized as Hitchin moduli spaces on certain punctured Riemann surfaces $C$ with divisor $D$ with special fixed data at $D$. The ALE gravitational instantons can be realized as certain quiver varieties.  There are also certain maps between gauge theoretic spaces.  A follow-up question to the Modularity Conjecture is:
Can these gauge-theoretic maps be used to understand degenerations of gravitational instantons from one type to another? 

We consider this question in the particular case of well-tuned families of ALG-$D_4$ gravitational instantons\footnote{For any fixed modular parameter $\tau \in \mathbb{H}/\mathrm{PSL}(2,\Z)$, the family of relevant ALG gravitational instantons are all asymptotic to $(\C \times T^2_\tau)/\Z_2$, where $\Z_2$ acts by $(z,w) \mapsto (-z, -w)$. } degenerating to ALE-$D_4$ gravitational instantons\footnote{
    The relevant ALE gravitational instantons are asymptotic to $\C^2/Q8$. Under the McKay correspondence, the affine Dynkin diagram arises as the root system of a finite subgroup $\Gamma$ of $SU(2)$. The relevant subgroup $\Gamma$ here is the quaternion group $Q8\simeq \{\pm 1, \pm i, \pm j, \pm k\}$ which is generated inside $SU(2)$ by the elements $\begin{pmatrix} i & 0 \\ 0 & i \end{pmatrix}$ and $\begin{pmatrix} 0 & -1 \\ 1 & 0 \end{pmatrix}$.
}. The ALG-$D_4$ moduli spaces are strongly parabolic $\SU(2)$-Hitchin moduli spaces on the four-punctured sphere \cite{FMSW21} and the ALE-$D_4$ are  certain Nakajima quiver varieties known as $n=4$ hyperpolygon spaces \cite{KN90}.  Note that we are only considering the subfamily of gravitational instantons that admit a triholomorphic $\U(1)$-action.  The relevant gauge theoretic map from $n=4$ hyperpolygon spaces to strongly parabolic $\SL(2,\C)$-Higgs bundle moduli spaces (also called  $\SU(2)$-Hitchin moduli spaces) on the four-punctured sphere is described in \cite{GM11}.  It preserves the holomorphic symplectic structures \cite{BFGM15}.

Moreover, a similar result holds for arbitrary $n$: the hyperk\"ahler metrics on well-tuned families of parabolic $\SU(2)$-Hitchin moduli spaces on the $n$-punctured sphere converge to the $n$-hyperpolygon space. In this case, the hyperk\"ahler metric on hyperpolygon space is known to be quasi-asymptotically conical \cite{DimakisRochon}, a higher-dimensional generalization of ALE. We note that at this time there is no parallel theorem about $\SU(2)$-Hitchin moduli spaces on the $n$-punctured sphere having a hyperk\"ahler type that generalizes ALG.

\bigskip

These hyperk\"ahler metrics are difficult to access. In particular, the hyperk\"ahler metric on the Hitchin moduli space is written in terms of solutions of gauge-theoretic systems of coupled nonlinear elliptic (modulo gauge) partial differential equations.  (In \Cref{sec:preliminaries}, we will introduce many of the relevant notions including parabolic Higgs bundles and describe the hyperk\"ahler metric. But for the sake of brevity, we will here assume that the reader has some familiarity with at least ordinary Higgs bundles, in order to give a succinct description of this work and its relation to other works.) The hyperk\"ahler metric on parabolic $\SU(2)$-Hitchin moduli space is understood as one approaches the noncompact end of the Hitchin moduli space in \cite{MSWW16, Mochizuki, FredricksonSLn, Fre19, FMSW21,mochizuki2024asymptoticbehaviourhitchinmetric,mochizuki2024comparisonhitchinmetricsemiflat} (``$R \to \infty$''), as the PDE decouples in the limit.  However, in this paper we are taking a very different limit ($``R \to 0$'', while simultaneously degenerating the boundary conditions near the punctures).  

An ``$R \to 0$'' type limit appears in \cite{DFKMMN16, CW18, CFW24}, where the conformal limit of a parabolic Higgs bundle is computed in \cite{CFW24}. To review, fix a stable parabolic Higgs bundle $(\mathcal{E}, \varphi)$ and consider the family of harmonic metrics $h_R$ solving the $R$-rescaled Hitchin equation
\begin{equation}\label{eq:Hitchinintro}
    R^{-1}F_{\nabla(\mathcal{E}, h_R)}^\perp + R [\varphi, \varphi^{\dagger_{h_R}}]=0,
\end{equation}
where $\nabla(\mathcal{E}, h_R)$ is the Chern connection.
Let $\zeta \in \C^\times$ be the twistor parameter where $\zeta=0$ corresponds to the Higgs bundle moduli space, and fix $\frac{\zeta}{R}=\hbar $.
Collier--Fredrickson--Wentworth prove that the conformal limit
\[ \lim_{R,\zeta \to 0} \frac{R}{\zeta} \varphi + \nabla(\mathcal{E}, h_R) + \zeta R \varphi^{\dagger_{h_R}} \]
exists in cases that include (1) strongly parabolic Higgs bundles or (2) weakly parabolic Higgs bundles with full flags. Then, they discuss how the conformal limit interacts with the natural stratification of the space of parabolic logarithmic $\lambda$-connections by $\C^\times$-limits. One common feature shared by conformal limit and our limit is that both are joint limits with special tuning. In the conformal limit, $R \to 0$ and $\zeta \to 0$; in our limit, $R \to 0$  while the parabolic weights $\alpha_i \to \frac{1}{2}$. 

The degeneration of the boundary conditions appears in \cite{Jud98}, and for parabolic Higgs bundles in \cite{KW18}.  For parabolic Higgs bundles, the harmonic metric $h_R$ is singular at the divisor $D \subset C$, and the boundary condition for the harmonic metric $h_R$ is determined by weighted flags at those fixed points. In the case of $\SL(2,\C)$ Higgs bundles, the weighted flag at $p_i \in D$ is
$$\begin{array}[column sep = -5pt]{ccccccc}
\hspace{2.9cm}   &0   &\subset    &F_i        &\subset    &E_{p_i}    &\\
\hspace{2.9cm}   &    &           &1-\alpha_i &<          &\alpha_i,  &\qquad\alpha_i \in (0, \frac{1}{2}).
\end{array}$$
We consider a degeneration of the boundary condition $\alpha_i \to \frac{1}{2}$ from the full flag case to the not-full-flag case. In the special cases where a point of the Hitchin moduli space can be interpreted as a uniformization metric on the punctured surface $\hat C=C\sm D$, this is a degeneration from conical singularities to cuspidal singularities at $D$.

The joint limit has the effect of sending the volume of fibers of the Hitchin fibration to $\infty$ (``$R \to 0$'') while preserving the holomorphic symplectic form. We tune the $R \to 0$ and $\alpha_i \to \frac{1}{2}$ limits so that 
\[ \alpha_i(R) = \frac{1}{2} - R \beta_i\]
for some fixed $\beta_i$.
This tuning of the degenerating boundary condition $\alpha_i(R)$ is carefully chosen; when $n=4$, one can interpret this choice as the choice to hold part of the cohomology data of the real symplectic structure constant as $R$ varies (see \Cref{fig: metric degeneration}).

\bigskip

We prove the following result:
\begin{namedtheorem}[(c.f. Theorem \ref{thm: HK degeneration})]
    Fix generic $\vec\beta\in(0,\infty)^n$. Let $\alpha_i(R)=\frac12-R\beta_i$, 
let $\mathcal{X}(\vec \beta)$ be the $n$-hyperpolygon space, let $\mathcal M_R(\vec\alpha(R))$ be the moduli space of solutions to the $R$-rescaled Hitchin's equations on the $n$-punctured sphere with parabolic weights $\vec{\alpha}(R)$, and let 
\[ \mathcal{T}_R: \mathcal{X}(\vec \beta) \to \mathcal M_R(\vec\alpha(R))\]
 be the natural embedding in \cite{GM11}.

As $R \to 0$,
   the pullback of the family of metrics $\mathcal{T}_R^*(g_{R,\vec\alpha(R)})$ on the Hitchin moduli moduli space $\mathcal M_R(\vec\alpha(R))$ converges pointwise to the metric $2\pi\cdot g_{\mathcal X(\vec\beta)}$ on hyperpolygon space $\mathcal X(\vec\beta)$.
\end{namedtheorem}

We briefly explain the proof.
Hyperpolygon space can be constructed as a hyperk\"ahler manifold via a hyperk\"ahler quotient construction or as a holomorphic symplectic manifold via a holomorphic symplectic quotient construction. In the first, one must choose representatives which solve both real and complex moment map equations---a ``unitary hyperpolygon'' in our terminology---while in the second, one must only choose a representative which solves the complex moment map equation---a ``hyperpolygon'' in our terminology.  The Hitchin moduli space $\mathcal{M}_R(\vec \alpha(R))$ and the Higgs bundle moduli space $\mathcal{M}^{\mathrm{Higgs}}(\vec \alpha(R))$ are similarly related.  Godinho--Mandini \cite{GM11} construct a map $\mathcal{T}: \mathcal{X}(\vec \beta) \to \mathcal{M}^{\mathrm{Higgs}}(\vec \alpha(R))$ preserving the holomorphic symplectic structures coming from the complex moment maps \cite{BFGM15}. 

The map $\mathcal{T}_R$ above factors through $\mathcal{T}$.
Namely, a point in $\mathcal{M}_R(\vec \alpha(R))$ is a Higgs bundle in $\mathcal{M}^{\mathrm{Higgs}}(\vec \alpha(R))$ together with the hermitian metric $h_R$ on the underlying complex bundle that solves $R$-rescaled Hitchin's equations (the real moment map). 
The key idea in our proof of \Cref{thm: HK degeneration} involves a delicate construction of approximate solutions to Hitchin's equations (the real moment map), which uses all the hyperpolygon moment map data to ensure the metric is adapted to the parabolic structure and globally defined.  

This result was independently proved by different methods by Lynn Heller, Sebastian Heller, and Claudio Meneses in \cite{HHM}. Their work builds on \cite{HHT2025}. A similar observation is made in \cite[Appendix A]{BTX}, and we would like to better understand the connection with Coulomb branches and Higgs branches, at the level of hyperk\"ahler (rather than holomorphic symplectic) geometry. 
\bigskip

This more complicated theorem is of a similar type to the simple degeneration of ALF-$A_n$ gravitational to ALE-$A_n$ gravitational instantons. All ALF-$A_n$ and ALE-$A_n$ gravitational instantons can be obtained using the Gibbons-Hawking ansatz, as 
the total space of a principal $U(1)$ bundle over $\R^3$. In the case of ALF-$A_n$, the hyperk\"ahler  multi-Taub-NUT metrics are specified by a harmonic function
    $$V_R(x)=R+\sum_{i=1}^n\frac1{4\pi|x-p_i|}$$
for $R>0$ and distinct points $p_1,\dots,p_{n+1}\in\R^3$, up to translation.  
The hyperk\"ahler metric is 
\begin{equation}
    g_R=V_R \norm{\de \vec{x}}^2 + V_R^{-1} \left(\frac{1}{2 \pi} \Theta\right)^2,
\end{equation}
where $\vec{x} \in \mathbb{R}^3$ and $\Theta$ is a connection on the $U(1)$  bundle whose curvature $F_R$ is related to the potential $V_R$ by $F_R = -2 \pi \star \de V_R$. Note that $\Theta$ doesn't depend on the constant $R$, so the potential determines the entire hyperk\"ahler structure.  As $R\to0$, the volume of the $U(1)$-fibers tends toward infinity and the metric degenerates to an ALE-$A_n$ instanton called the multi-Eguchi-Hanson metric.

\bigskip

\subsection*{Organization}
The paper is organized as follows:

In \Cref{sec:preliminaries}, we introduce the $n$-sided hyperpolygon spaces and the strongly parabolic Hitchin moduli spaces on $\mathbb{CP}^1$ with $n$ punctures, we describe their hyperkähler metrics, and the map from hyperpolygons to strongly parabolic Higgs bundles. 

The $U(1)$-action on each hyperk\"ahler space gives rise to a Morse--Bott function.  In \Cref{sec: Torelli Numbers}, we compute the value of this function at the corresponding fixed points and show that they match.   Then, we restrict to $n=4$ for concreteness (the real dimension four case), and explore the topology and geometry of the two families of moduli spaces.

In \Cref{sec: Local Model}, we introduce local model solutions of the $\SU(2)$-Hitchin equations on the disk. We conduct a detailed local analysis of these model solutions and their first variation. 

In \Cref{sec: The Hyperkahler Metrics}, we use the local models to construct approximate harmonic metrics for Higgs bundles on $\CP^1$ with $n$-punctures.  We then use an implicit function theorem argument perturb our approximate metrics to actual solutions, and finally prove our main theorem.

\subsection*{Notational Conventions} We take hermitian forms $h$ to be conjugate linear in the first slot and $\C$-linear in the second slot.  We identify such forms with their Gram matrices, e.g.\ $h(u,v)=u^\dagger hv$.  If $h$ is a hermitian form on $V$ and $B\in\End(V)$, then the $h$-adjoint is $B^{\dagger_h}=h^{-1}B^\dagger h$, where $\dagger$ (without the $h$) denotes the standard complex conjugation.  For the standard 1-forms $\de\bar z$ and $\de z$ in a coordinate chart on a Riemann surface $C$, the Hodge star is given by $\star\de\bar z=i\de z$ and $\star\de z=-i\de\bar z$, extended $\mathcal C_C^\infty$-antilinearly.

Given a holomorphic structure $\delbar_E$ and a metric $h$ on a vector bundle $E$, we denote the Chern connection by $\nabla(\delbar_E,h)$.

For a matrix $A\in\mathfrak{gl}(n,\C)$ we use $A^\perp$ to denote the orthogonal projection onto $\mathfrak{sl}(n,\C)$.  In other words, $A^\perp:=A-\frac1n\tr(A)$ is the trace-free part of $A$.

\subsection*{Acknowledgements}

The authors would like to thank Nick Addington, Sergey Cherkis, Sze Hong Kwong, Rafe Mazzeo, Claudio Meneses, Nick Proudfoot, Steve Rayan, Laura Schaposnik, Hartmut Weiss, Richard Wentworth, and Graeme Wilkin for useful discussions. AY particularly thanks Nick Proudfoot for his help with hyperpolygons at the start of this project.  LF is partially supported by NSF grant DMS-2005258.  AY is supported by NSF grant RTG-2039316.

\section{Preliminaries}\label[section]{sec:preliminaries}
In this section, we summarize relevant background on hyperpolygon spaces, Hitchin moduli spaces, and their hyperk\"ahler metrics.  Finally, we describe gauge-theoretic embeddings from the hyperpolygon spaces to the Higgs bundle moduli spaces constructed in \cite{GM11}, and provide a corrected statement and proof of their theorem showing that the map is well-defined.  Both moduli spaces have distinguished complex structures $J_1$ and holomorphic symplectic forms $\Omega_{J_1}$, and the family of embeddings preserves these complex structures and preserves the holomorphic symplectic forms up to a factor of $2\pi$.

\subsection{Hyperpolygon Spaces}
Nakajima quiver varieties were introduced by Kronheimer and Nakajima \cite{KN90, Nak94}.  In his foundational paper \cite{Nak94}, Nakajima constructs quiver varieties as hyperkähler reductions associated with a general finite graph (quiver).  The term ``hyperpolygon space'' and its formal construction were introduced by Konno in \cite{Kon02}. These are Nakajima quiver varieties for a particularly simple quiver (see Figure \ref{fig: quiver}).

In this section, we review the construction of hyperpolygon space via finite-dimensional hyperk\"ahler quotient.

\subsubsection{Construction}
Let $Q_n$ be a quiver consisting one central node labelled 2 and $n$ exterior nodes labelled 1, and arrows from each exterior node to the central node.  Let $\tilde Q_n$ be the associated Nakajima quiver (see \Cref{fig: quiver}), and let $X=\Rep\tilde Q_n$ be the space of representations.  Under the identification
    $$X\cong T^*\left(\bigoplus_{i=1}^n\Hom(\C,\C^2)\right)\cong\Hom(\C^n,\C^2)\oplus\Hom(\C^2,\C^n),$$
 $(\bx,\by) \in \Hom(\C^n,\C^2)\oplus\Hom(\C^2,\C^n),$ denotes an element of $X$; we denote the columns of $\bx$ by $x_1,\dots,x_n$ and the rows of $\by$ by $y_1,\dots,y_n$.  The group $G_\C=\left(\SL(2,\C)\times(\C^\times)^n\right)/(\Z/2)$ acts by change of coordinates on the nodes, i.e.\ if $g=(A,t)\in G_\C$ and $(\bx,\by)\in X$, then $g\cdot(\bx,\by)=(A\bx t^{-1},t\by A^{-1})$ (here we identify $t\in(\C^\times)^n$ with a diagonal matrix in $\GL(n,\C)$).

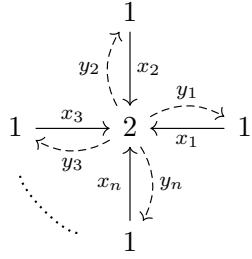
\begin{figure}[t]
$$
\begin{tikzcd}[row sep=1cm, column sep=1cm]
                &1\arrow{d}{x_2}    &\\
    1\arrow{r}{x_3}  &2\arrow["y_1", r, dashed, bend left]\arrow["y_2", u, dashed, bend left]\arrow["y_3", l, dashed, bend left]\arrow["y_n", d, dashed, bend left]
                        &1\arrow{l}{x_1} \\
                &1\arrow{u}{x_n}
                    \arrow[lu, no head, thick, dotted, shorten <=3ex, shorten >=3ex, bend left=45, xshift=.4ex, yshift=.8ex]
\end{tikzcd}
$$
\caption{The Nakajima quiver $\tilde Q_n$}
\label{fig: quiver}
\end{figure}

The space $X\cong T^*\C^{2n}$ has a standard hyperkähler structure $(g,J_1,\omega_{J_1},\Omega_{J_1})$ where $\omega_{J_1}$ and $\Omega_{J_1}$ are respectively real and holomorphic symplectic forms.  
The action $G_\C\acts X$ preserves $\Omega_{J_1}$ and admits an algebraic moment map
    $$\mu_\C=(\mu_{\SL(2,\C)},\mu_{(\C^\times)^n})\from X\to\mathfrak g_\C^*=\mathfrak{sl}(2,\C)^*\times(\C^n)^*$$
given by
\begin{equation}\label{eqn: HP complex moment maps}
    \mu_{\SL(2,\C)}(\bx,\by)=\sum_{i=1}^n (x_iy_i)^\perp,\qquad
        \mu_{(\C^\times)^n}(\bx,\by)=(y_1x_1,\dots,y_nx_n).
\end{equation}
Here, we write $A^\perp$ to denote the trace-free part $A-\frac12\tr(A)\text{Id}$.  
Similarly, the action $G\acts X$ preserves $\omega_{J_1}$ and admits the real moment map
    $$\mu_\R=(\mu_{\SU(2)},\mu_{U(1)^n})\from X\to\mathfrak g^*=\mathfrak{su}(2)^*\times(\mathfrak u(1)^n)^*$$
given by
\begin{equation}\label{eqn: HP real moment map}
    \mu_{\SU(2)}(\bx,\by)=\frac i2\sum_i (x_ix_i^\dagger)^\perp-(y_i^\dagger y_i)^\perp,\qquad
        \mu_{U(1)^n}=\frac i2(|x_i|^2-|y_i|^2)_{i=1}^n.
\end{equation}
We identify these Lie algebras with their duals using the trace pairing, and we also identify $(\mathfrak u(1)^n)^*=(i\R)^n$ with $\R^n$.

The moment maps $\mu_\C,\mu_\R$ are respectively $G_\C$- and $G$-equivariant.  In order to perform hyperkähler reduction, we must select GIT parameters in the centers $Z(\mathfrak g_\C^*)$ and $Z(\mathfrak g^*)$ so that their preimages under the moment maps are preserved by the group actions.  For the complex parameter, we always choose $0\in\mathfrak g_\C^*$. 
\begin{definition}
    A representation $(\mathbf x,\mathbf y)$ of $\tilde Q_n$ is a \textit{hyperpolygon} if $\mu_\C(\mathbf x,\mathbf y)=0$. 
\end{definition}
The real parameter is an element of $Z(\mathfrak g^*)=Z(\mathfrak{su}(2)^*\oplus(i\R)^n)\cong0\oplus(i\R)^n$ which we identify with $\R^n$.
\begin{definition}
    Given a choice of $\vec \beta \in \R^n$, a representation $(\mathbf x,\mathbf y)$ of $\tilde Q_n$ is a \textit{unitary hyperpolygon} if $\mu_\C(\mathbf x,\mathbf y)=0$ and $\mu_{\R}(\mathbf x, \mathbf y)=\vec \beta$.
\end{definition}

\begin{definition}\label{def:straight}
Given a hyperpolygon $(\bx,\by)$, we say that a subset $I\subset\{1,\dots,n\}$ is \textit{straight}\cite{Kon02} if the vectors $\{x_i\}_{i\in I}$ are pairwise proportional.  For $I\subseteq[n]:=\{1,\dots,n\}$, let
\begin{equation}\label{eq:W}
    W_I(\vec\beta)=\sum_{i\in I}\beta_i-\sum_{i\notin I}\beta_i.
\end{equation}
\end{definition}

\begin{definition}\label{def: stable hyperpolygon}
A hyperpolygon $(\mathbf x,\mathbf y)$ is $\vec\beta$-\emph{stable} \cite{Kon02} if and only if :
\begin{itemize}
    \item $x_i\ne0$ for all $i$, and 
    \item there is no straight subset $I$ with $y_i=0$ for all $i\notin I$ and $W_I(\vec \beta)>0$.
\end{itemize}
\end{definition}

\begin{definition}
The \textit{hyperpolygon space} $\mathcal X(\vec\beta)$ is the moduli space \[X\twoslash_{(0,\vec\beta)}G_\C=\mu^{-1}_{\C}(0)^{\vec\beta\text{-st}}/G_\C\] of $\vec\beta$-stable hyperpolygons up to the action of $G_\C$. 
\end{definition}

From its construction via a finite-dimensional holomorphic symplectic quotient, $\mathcal X(\vec\beta)$ inherits a holomorphic symplectic structure $(J_1^{HP}, \Omega_{J_1}^{HP})$.

As a consequence of the Kempf--Ness theorem, a hyperpolygon $(\bx,\by)$ is $\vec\beta$-stable if and only if there exists some $g\in G_\C$ such that $g\cdot(\bx,\by)\in\mu_{\R}^{-1}(\vec\beta)$.  Therefore, $\mathcal X(\vec\beta)$ can also be constructed as the finite-dimensional hyperkähler quotient \[\mathcal X(\vec \beta)=X\fourslash_{(0, \beta)} G= (\mu_{\C}^{-1}(0)\cap\mu_{\R}^{-1}(\vec\beta))/G.\]
Given $\widetilde{Q}_n$, 
the dimension of hyperpolygon space from is $\dim_{\R}\mathcal X(\vec \beta)= 4(n-3)$. 
\bigskip

\begin{figure}[h]
\centering
\resizebox{4cm}{!}{
\begin{tikzpicture}
\tikzstyle{every node}=[font=\LARGE]
\draw [ fill=black , line width=1pt ] (16,2) circle (.2cm);
\draw [line width=4pt] (16,2) -- (10,5);
    \draw [decorate, decoration = {brace,amplitude=12pt,raise=12pt}] (16,2) -- (10,5)
        node[midway,xshift=-6ex,yshift=-4em,font=\Huge]{\scalebox{1.2}{$\sqrt2\beta_1$}};
\draw [line width=4pt, ->, >=Stealth, dashed] (10,5) -- (5,7.5);
    \node[font=\Huge, xshift=0em, yshift=3.5em] at (10.5,4.75) {\scalebox{1.2}{$v_1$}};
\draw [ color=green!40!black, line width=4pt, ->, >=Stealth, dashed] (0,10) -- (5,7.5);
    \node[font=\Huge, xshift=2em, yshift=2em, color=green!40!black] at (2.5,8.75) {\scalebox{1.2}{$w_1$}};

\draw [ fill=black , line width=1pt ] (0,10) circle (.2cm);
\draw [line width=4pt] (0,10) -- (3,16);
    \draw [decorate, decoration = {brace,amplitude=12pt,raise=12pt}] (0,10) -- (3,16)
        node[midway,xshift=-15ex,yshift=2.5em,font=\Huge]{\scalebox{1.2}{$\sqrt2\beta_2$}};
\draw [line width=4pt, ->, >=Stealth, dashed] (3,16) -- (5,20);
    \node[font=\Huge, xshift=3.5em, yshift=0em] at (2.5,15) {\scalebox{1.2}{$v_2$}};
\draw [ color=green!40!black, line width=4pt, ->, >=Stealth, dashed] (7,24) -- (5,20);
    \node[font=\Huge, xshift=2em, yshift=-3em, color=green!40!black] at (6,22) {\scalebox{1.2}{$w_2$}};

\draw [ fill=black , line width=1pt ] (7,24) circle (.2cm);
\draw [line width=4pt] (7,24) -- (10.5,24);
    \draw [decorate, decoration = {brace,amplitude=12pt,raise=12pt}] (7,24) -- (10.5,24)
        node[midway,xshift=0ex,yshift=5em,font=\Huge]{\scalebox{1.2}{$\sqrt2\beta_3$}};
\draw [line width=4pt, ->, >=Stealth, dashed] (10.5,24) -- (12.5,24);
    \node[font=\Huge, xshift=3em, yshift=-2.5em] at (9.25,24) {\scalebox{1.2}{$v_3$}};
\draw [ color=green!40!black, line width=4pt, ->, >=Stealth, dashed] (14.5,24) -- (12.5,24);
    \node[font=\Huge, xshift=0em, yshift=-2.5em, color=green!40!black] at (13.5,24) {\scalebox{1.2}{$w_3$}};

\draw [ fill=black , line width=1pt ] (14.5,24) circle (.2cm);
\draw [line width=4pt] (14.5,24) -- (15, 24-22/3);
    \draw [decorate, decoration = {brace,amplitude=12pt,raise=12pt}] (14.5,24) -- (15, 24-22/3)
        node[midway,xshift=14ex,yshift=.2em,font=\Huge]{\scalebox{1.2}{$\sqrt2\beta_4$}};
\draw [line width=4pt, ->, >=Stealth, dashed] (15, 24-22/3) -- (15.5, 24-2*22/3);
    \node[font=\Huge, xshift=-3em, yshift=-1em] at (15, 24-22/3) {\scalebox{1.2}{$v_4$}};
\draw [ color=green!40!black, line width=4pt, ->, >=Stealth, dashed] (16,2) -- (15.5, 24-2*22/3);
    \node[font=\Huge, xshift=-3em, yshift=4em, color=green!40!black] at (15.7, 6) {\scalebox{1.2}{$w_4$}};

\draw [line width=4pt, loosely dotted] (16,2) -- (7,24);x
\draw [line width=4pt, blue] (11.5,15.5) arc (65:160:1.5);
\end{tikzpicture}
}
\caption{A hyperpolygon represented as vectors in $\mathfrak{su}(2)\cong\R^3$}
\label{fig: hyperpolygon}
\end{figure}
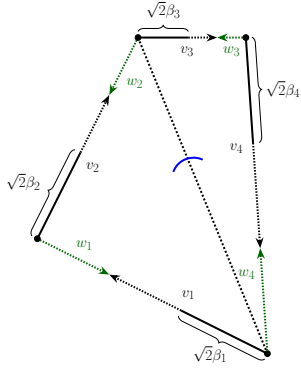

\begin{remark}[Geometric interpretation of unitary hyperpolygons]
Let us briefly interpret the geometric meaning of a unitary hyperpolygon $(\bx,\by)\in\mu_\C^{-1}(0)\cap\mu_\R^{-1}(\vec\beta)$, building on \cite[Remark 2.6]{HP03}. Consider the $\mathfrak{su}(2)$ matrices $v_i=(x_ix_i^\dagger)^\perp$ and $w_i=(y_i^\dagger y_i)^\perp$.  We calculate the norms
\begin{align}
    |v_i|^2=\tr\left(\left(x_ix_i^\dagger-\frac12\tr(x_ix_i^\dagger)\text{Id}\right)^\dagger\left(x_ix_i^\dagger-\frac12\tr(x_ix_i^\dagger)\text{Id}\right)\right)=\frac{|x_i|^4}2
        \label{eqn: v norm}\\
    |w_i|^2=\tr\left(\left(y_i^\dagger y_i-\frac12\tr(y_i^\dagger y_i)\text{Id}\right)^\dagger\left(y_i^\dagger y_i-\frac12\tr(y_i^\dagger y_i)\text{Id}\right)\right)=\frac{|y_i|^4}2
        \label{eqn: w norm}
\end{align}
so $|v_i|=|x_i|^2/\sqrt2$ and $|w_i|=|y_i|^2/\sqrt2$.  The moment maps \eqref{eqn: HP real moment map} thus constrain $\mu_{\SU(2)}(\bx,\by)=\sum v_i-w_i$ and $|v_i|-|w_i|=\sqrt2\beta_i$.

The assignment $x_i\mapsto v_i$ is a 2-to-1 map $\C^2/U(1)\to\mathfrak{su}(2)\cong\R^3$ ramified over 0, and it intertwines the standard left action $\SU(2)\acts\C^2$ with the adjoint action $\SU(2)\acts\mathfrak{su}(2)$, which in turn coincides with the action $\SO(3)\acts\R^3$.  The analogous statement holds for $(\C^2)^*/U(1)\to\mathfrak{su}(2)$ given by $y_i\mapsto w_i$.  One can check that the complex moment map condition $0=\mu_{\C^\times,i}=y_ix_i$ implies $v_i$ and $w_i$ are proportional as vectors in $\R^3$ and point in opposite directions\footnote{In \cite[Figure 5]{HP03}, the hyperpolygon image doesn't show $v_i, w_i$ as proportional.}, so $\sqrt{2\beta_i}=|v_i|-|w_i|=|v_i+w_i|$.  Thus, given a hyperpolygon $(\bx,\by)\in\mu_\C^{-1}(0)\cap\mu_\R^{-1}(\vec\beta)$ we can form an $n$-sided polygon in $\R^3$ with each side having the form $v_i-w_i$ and satisfying the relation $|v_i+w_i|=\sqrt2\beta_i$.  The action of the $\U(1)^n$ factor of $G$ has no effect on our polygon, and the action of the $\SU(2)$ factor corresponds to transforming the polygon by $\SO(3)\acts\R^3$.

Lastly, note that map from unitary hyperpolygons (up to action of $G$) to polygons in $\R^3$ satisfying these constraints (up to the action of $\SO(3)$) is not surjective.  The geometric construction of the polygon above never uses the condition $0=\mu_{\SL(2,\C)}(\bx,\by)=\sum x_iy_i$.  One example of these ``hidden'' constraints is that we can never have exactly one nonzero $w_i$, since that means exactly one $y_i$ is nonzero, and since all $x_i\ne0$ the equation $\sum x_iy_i=0$ is impossible.  Similarly, it is also impossible to have exactly three pairwise nonproportional $w_i$ with the rest all being 0, though this is harder to see (see \Cref{lemma: nilpotent cone components}).  As a consequence, a 3-sided hyperpolygon space with $n=3$ exterior nodes is either empty or a single point of the form $(\bx,\mathbf0)$, depending on the chamber of $\vec\beta$.
\end{remark}

\subsubsection{Tangent Space and Hyperk\"ahler Metric}
In this section, we write an expression for the hyperk\"ahler metric on $\mathcal{X}(\vec \beta)$ in terms of $(\mathbf{x}, \mathbf{y})$ and 
$(\dot\bx, \dot\by) \in T_{(\mathbf{x}, \mathbf{y})}\mathcal{X}(\vec \beta)$.  We begin with some general theory.

\bigskip

Suppose a Lie group (possibly a Banach Lie group) $G$ acts properly on a Hilbert space $(V,h)$, say $h=g+i\omega$.  Suppose further that the action of $G$ admits a moment map $\mu$ with respect to $\omega$, and consider the quotient $X=V\twoslash_\alpha G=\mu^{-1}(\alpha)/G$ for $\alpha\in Z(\mathfrak g)$.
\begin{proposition}\label[proposition]{prop: symplectic quotient tangent space}
    If $X$ is smooth at a point $[p]$ with $p\in\mu^{-1}(\alpha)$, then the tangent space of $X$ at $[p]$ is isometric to the orthogonal complement in $T_pV$ of the subspace generated by the infinitesimal action of $\mathfrak g_\C$ at $p$.
\end{proposition}
\begin{proof}
    For simplicity, assume the stabilizer of $p$ in $G_\C$ is discrete (else, we can replace $\mathfrak g_\C$ with $\mathfrak g_\C/\mathfrak g_\C^p$ where $\mathfrak g_\C^p$ denotes the Lie algebra of the stabilizer of $p$).  Let $\rho\from G\times V\to V$ be the action map.  The construction of the metric on $T_{[p]}X$ is defined through the exact sequence
    \begin{equation}\label{eqn: tangent space exact sequence}
        0\to\de\rho\big|_{(1,p)}(\mathfrak g)\to T_p\mu^{-1}(\alpha)\to T_{[p]}X\to0
    \end{equation}
    by identifying $T_{[p]}X$ with the orthogonal complement in $T_p\mu^{-1}(\alpha)$ of $\de\rho\big|_{(1,p)}(\mathfrak g)$.  The vectors generated by the infinitesimal action of $i\mathfrak g$ on $p$ span the orthogonal complement of $\ker\de\mu|_p$ (for these are both Lagrangian subspaces for $\omega$ by definition of moment maps), so a vector $v\in T_pV$ is orthogonal to the action of the complexified Lie algebra $\mathfrak g_\C$ if and only if it solves $\de\mu(v)=0$.  Hence, \eqref{eqn: tangent space exact sequence} can be extended to the split exact sequence of Hilbert spaces
    \begin{equation}\label{eqn: tangent space exact sequence complexified}
        0\to\de\rho\big|_{(1,p)}(\mathfrak g_\C)\to T_p V\to T_{[p]}X\to0
    \end{equation}
    where $T_{[p]}X$ is isometrically identified with $\de\rho\big|_{(1,p)}(\mathfrak g_\C)^\perp$.  This proves the result.
\end{proof}

Now suppose $G$ acts properly on a complex vector space $W$ with a hyperkähler structure $(g,J_1,\omega_{J_1},\Omega_{J_1})$ such that $G$ preserves $\omega_{J_1}$ and $G_\C$ preserves $\Omega_{J_1}$, and suppose these actions admit real and complex moment maps $\mu_\R$, $\mu_\C$.  Let $\alpha_\C\in Z(\mathfrak g_\C), \alpha_\R\in Z(\mathfrak g_\R)$, and consider the hyperkähler quotient
    $$Y=W\hspace{-.5em}\fourslash_{(\alpha_\C,\alpha_\R)}\hspace{-.5em}G=\left(\mu_\C^{-1}(\alpha_\C)\cap\mu_\R^{-1}(\alpha_\R)\right)/G.$$
\Cref{prop: symplectic quotient tangent space} gives the following:
\begin{corollary}\label[corollary]{cor: hyperkahler quotient tangent space}
    If $Y$ is smooth at $[p]$ for $p\in\mu_\C^{-1}(\alpha_\C)\cap\mu_\R^{-1}(\alpha_\R)$, then the tangent space of $Y$ at $[p]$ is isometric to the orthogonal complement in $\ker\de\mu_\C\big|_p$ of the subspace generated by $\mathfrak g_\C$ at $p$.  In other words, we have an isometric splitting of the exact sequence:
    \begin{equation}\label{eqn: HK tangent space exact sequence}
    \begin{tikzcd}[column sep=1.5em]
        0\arrow{r}  &\de\rho\big|_p(\mathfrak g_\C)\arrow{r}    &\ker\de\mu_\C\big|_p\arrow{r}  &T_{[p]}Y\arrow{r}\arrow[l, bend right=33, dashed] &0
    \end{tikzcd}
    \end{equation}
\end{corollary}

\bigskip

Returning to the case of hyperpolygon space, we have $G=\left(\SU(2)\times(\U(1))^n\right)$ and $G_\C=\left(\SL(2,\C)\times(\C^*)^n\right)$ acting on $X=\Rep\tilde Q$, and the moment maps $\mu_\C$ and $\mu_\R$ decompose into
    $$\mu_{\SL(2,\C)}(\mathbf x,\mathbf y)=\sum_jx_jy_j,\qquad \mu_{\C^\times,j}(\mathbf x,\mathbf y)=y_jx_j,$$
    $$\mu_{\SU(2)}(\mathbf x,\mathbf y)=\frac i2\sum_j(x_jx_j^\dagger-y_j^\dagger y_j)^\perp,\qquad \mu_{\U(1),j}(\mathbf x,\mathbf y)=\frac i2(x_j^\dagger x_j-y_jy_j^\dagger).$$
Suppose $(\mathbf x,\mathbf y)\in\mu_\C^{-1}(0)\cap\mu_\R^{-1}(\vec\beta)$.  The tangent space to $[({\mathbf x},{\mathbf y})]\in\mathcal X(\vec\beta)$ can be identified with the subspace of $\ker\de\mu_\C\big|_{(\bx,\by)}$ consisting of vectors $(\dot\bx,\dot\by)$ which are orthogonal to the vector fields on $\Rep(\tilde Q)$ generated by $\mathfrak g_\C=\mathfrak{sl}(2,\C)\oplus\C^n$.  An element $A\in\mathfrak{sl}(2,\C)$ generates the vector $(A\mathbf x,-\mathbf yA)\in T_{({\mathbf x},{\mathbf y})}\Rep(\tilde Q)$, so the orthogonality condition on a $(\dot {\mathbf x},\dot {\mathbf y})$ is
\begin{align*}
    0=\langle(A\mathbf x,-\mathbf yA),(\dot {\mathbf x},\dot {\mathbf y})\rangle
        &=\tr((A\mathbf x)^\dagger\dot {\mathbf x})-\tr((\mathbf yA)^\dagger\dot {\mathbf y})\\
        &=\tr(A^\dagger(\dot{\mathbf x}\mathbf x^\dagger-{\mathbf y}^\dagger\dot {\mathbf y}))\\
        &=\langle A,\dot{\mathbf x}\mathbf x^\dagger-{\mathbf y}^\dagger\dot {\mathbf y}\rangle_{\mathfrak{sl}(2,\C)}
\end{align*}
where $\langle-,-\rangle_{\mathfrak{sl}(2,\C)}$ is the trace pairing.  Similarly, an element $\lambda_j$ of the $j$th factor of $\C^\times$ in $G$ generates the vector $(-x_j\lambda_j,\lambda_j y_j)$, so the orthogonality condition requires
    $$0=\langle(-x_j\lambda_j,\lambda_j y_j),(\dot x_j,\dot y_j)\rangle=-\tr((x_j\lambda_j)^\dagger\dot x_j)+\tr((\lambda_j y_j)^\dagger\dot y_j)=\bar\lambda_j(-x_j^\dagger\dot x_j+\dot y_j y_j^\dagger).$$
These orthogonality equations must hold for all $A\in\mathfrak{sl}(2,\C)$ and $\lambda_j$, so we get:
\begin{proposition}\label[proposition]{prop: hyperpolygon space tangent space}
    The equations
    \begin{equation}\label{eqn: linearized moment map}
        0=\de\mu_{\SL(2,\C)}(\dot {\mathbf x},\dot {\mathbf y})=\frac i2\sum(\dot x_jy_j+x_j\dot y_j),\qquad 0=\de\mu_{\C^*,j}(\dot {\mathbf x},\dot {\mathbf y})=\frac i2(\dot y_jx_j+y_j\dot x_j)
    \end{equation}
    \begin{equation}\label{eqn: gauge orbit orthogonality}
        0=\sum(\dot x_j x_j^\dagger-y_j^\dagger\dot y_j)^\perp\qquad 0=x_j^\dagger\dot x_j-\dot y_j y_j^\dagger
    \end{equation}
    cut out an isometric copy of $T_{({\mathbf x},{\mathbf y})}\mathcal X(\vec\beta)$ in $T_{(\bx,\by)}X$.
\end{proposition}

\begin{definition}
We shall call solutions $(\dot\bx,\dot\by)$ to \eqref{eqn: linearized moment map}  and \eqref{eqn: gauge orbit orthogonality} \emph{unitary lifts} of tangent vectors in $\mathcal X(\vec\beta)$, or \emph{unitary deformations} of $(\bx,\by)$.
\end{definition}
In summary, to evaluate the metric on $\mathcal X(\vec\beta)$ on a pair of tangent vectors $v_1,v_2\in T_p(\mathcal X(\vec\beta))$, one must choose a unitary hyperpolygon $(\bx,\by)$ representing $p$ and unitary lifts $(\dot\bx_1,\dot\by_1),(\dot\bx_2,\dot\by_2)$ of $v_1,v_2$, and pair these lifts using the natural hyperkähler metric on $X=\Rep(\tilde Q)$.  In particular, we have the following:
\begin{proposition}\label[proposition]{prop: metric on hyperpolygon space}
    The norm of a unitary deformation $(\dot\bx,\dot\by)\in T_{(\bx,\by)}\mathcal X(\vec\beta)$ is
        $$|(\dot\bx,\dot\by)|_{g_\text{HP}}^2=\sum_i|\dot x_i|^2+|\dot y_i|^2.$$
\end{proposition}
When $n=4$, and $\vec \beta \in (0, \infty)^4$ generic, the hyperk\"ahler manifold $\mathcal X(\vec\beta)$ is known to be  an ALE-$D_4$ gravitational instanton.
For general $n$ and $\vec \beta$ generic $\mathcal X(\vec \beta)$ is a hyperk\"ahler manifold of quasi-asymptotically conical type \cite{DimakisRochon}.

\subsection{\texorpdfstring{Moduli Space of Strongly Parabolic $\SL(2,\C)$-Higgs Bundles}{Moduli Space of Strongly Parabolic SL(2,C)-Higgs Bundles}}

Hitchin's equations for a Higgs bundle on a Riemann surface arise as the dimensional reduction of the self-dual Yang-Mills equations from four dimensions to two dimensions \cite{Hit87}.  For surfaces of genus $g>1$, the moduli space of solutions up to unitary gauge equivalence form a hyperkähler manifold.  The famous non-abelian Hodge correspondence (NAHC) provides existence of solutions which are gauge-equivalent to any given stable Higgs bundle.  These results were extended by Simpson to the non-compact case of a punctured Riemann surface (of any genus) with prescribed parabolic local system structures at each puncture \cite{Sim90}.  For an $\SL(2,\C)$-bundle $E$ over the $n$-punctured sphere $\C\P^1\sm\{p_1,\dots,p_n\}$, a choice a parabolic structure amounts to a flag $\mathcal F_i$ in $E_{p_i}$ and corresponding parabolic weights $\alpha_1,\dots,\alpha_n$ at each puncture. 

In this subsection, we will review strongly parabolic $\SL(2,\C)$-Higgs bundles, the construction of strongly parabolic $\SU(2)$-Hitchin moduli spaces as an infinite-dimensional hyperk\"ahler quotient, its tangent spaces, and its hyperk\"ahler metric.

\subsubsection{Parabolic Bundles}  We briefly review parabolic bundles, focusing on our context.
\begin{definition}[Parabolic Bundle with Full Flags]
Let $C$ be a Riemann surface and let $D=\{p_1,\dots,p_n\}$ a divisor on $C$ with distinct $p_i$.   A \emph{parabolic bundle with full flags} on $(C,D)$ is a holomorphic vector bundle $(E,\delbar_\mathcal E)$ of rank $k$ and the data of a weighted full flag at each $p_i$. I.e.\ at each $p_i$, we choose a full flag $\mathcal F_i=F_i^\bullet$ in $E_{p_i}$ and parabolic weights $\alpha_i^{(1)},\dots,\alpha_i^{(k)} \in (0, 1)$:
$$
\begin{array}[column sep = -5pt]{ccccccccc}
    0   &\subset   &F_i^{(1)}       &\subset    &\cdots &\subset    &F_i^{(k)}      &=          &E_{p_i}\\[4pt]
      &        &\alpha_i^{(1)}  &>          &\cdots &>          &\alpha_i^{(k)} &          &
\end{array}
$$
We use $\mathcal E$ to denote the data $(E, \delbar_{\mathcal E}, \vec\alpha, \mathcal F)$.
\end{definition}

\begin{definition}
    The \textit{parabolic degree} of a parabolic bundle $\mathcal E$ is $\text{pardeg}\,\mathcal E=\deg E+\sum_{i,j}\alpha_i^{(j)}$.  The \textit{slope} of $\mathcal E$ is $\text{slope}\,\mathcal E=\frac{\text{pardeg}\,\mathcal E}{\rk\mathcal E}$.
\end{definition}

\begin{definition}
    A parabolic bundle $\mathcal E$ is \emph{(semi)stable} if for every proper parabolic subbundle $\mathcal S\subset\mathcal E$ we have $\textrm{slope}\,\mathcal S<\textrm{slope}\,\mathcal E$ (respectively $\text{slope}\,\mathcal S\leq\text{slope}\,\mathcal E$).
\end{definition}

The parabolic structure on $E$ induces one on $\det E$ with weights $\sum_k\alpha_i^{(k)}$ at $p_i$.
\begin{definition}
An \emph{$\SL(k, \C)$-parabolic bundle with full flags}  is a rank $k$ parabolic bundle with full flags $\mathcal E=(E,\delbar_{\mathcal E},\vec\alpha,\mathcal F)$ together with an isomorphism between $\det \mathcal E$ and the trivial bundle with a fixed parabolic structure. 
\end{definition}

Note that the reduction of structure group to $\SL(k,\C)$ requires the induced parabolic weights on $\det E$ be integral.

\begin{remark}[Simplified notation for $\SL(2,\C)$-Parabolic Bundles]\label[remark]{rem:simplified}
As we will primarily discuss the $\rk E=2$ case, it is useful to develop simplified notation.  The flags $\mathcal F_i$ are each determined by a 1-dimensional subspace $F_i\subset E_{p_i}$, and the parabolic weights $\alpha_i^{(1)}+\alpha_i^{(2)}=1$ are determined by $\alpha_i=\alpha_i^{(2)}$. 
\end{remark}

\begin{example}\label[example]{ex:holrank2}
Moreover, we will take $C=\CP^1$. Then the Birkhoff--Grothendieck theorem classifies holomorphic vector bundles on $\CP^1$. Going forward, we will fix $\det \mathcal{E} \simeq \mathcal{O}$, with parabolic weight $1$ at $p_i \in D$, i.e. near $p_i$ $h_{\det}\simeq |z-p_i|^{2 \cdot 1}$. The adapted Hermitian-Einstein metric is (up to a scale)
\begin{equation}\label{eq:hdet}
h_{\det}(e, e)= (1+|z|^2)^{-n}\prod_{i=1}^n |z-p_i|^{2},
\end{equation} 
in the standard trivialization $e$ of $\mathcal{O}$.
Consequently, in the rank $2$ case, $(E, \delbar_E) \simeq \mathcal{O}(d) \oplus \mathcal{O}(-d)$. The additional weighted flag data is described above in \Cref{rem:simplified}.
\end{example}

\subsubsection{\texorpdfstring{Strongly Parabolic $\SL(2,\C)$-Higgs Bundles}{Strongly Parabolic SL(2,C)-Higgs Bundles}}

\begin{definition}
    A \emph{strongly parabolic $\SL(k,\C)$-Higgs field} on $\mathcal E$ is a global holomorphic section $\varphi$ of $\Omega^{1,0}(\End_0\mathcal E)\otimes\mathcal O(D)\cong\End_0\mathcal E\otimes K_C(D)$, where the subscript `$0$' denotes trace-free endomorphisms.  At each point $p_i \in D$, the residue is strictly upper triangular, i.e. 
    \[\mathrm{res}_{p_i} \varphi: F_i^{(k)} \to F_i^{(k-1)}.\] 
\end{definition}
For rank $2$, $\mathrm{res}_{p_i} \varphi: E_{p_i} \to F_i$ and $\mathrm{res}_{p_i} \varphi:F_i \to 0$.

\begin{remark}
A weakly parabolic Higgs field is  one where the residue is upper triangular, but need not be strictly upper triangular. We do not consider this generalization in this paper.
\end{remark}

\begin{definition}
    A \textit{strongly parabolic $
    SL(k,\C)$-Higgs bundle with full flags} is a pair $(\mathcal E, \varphi)$ such that $\mathcal E$ is a $\SL(k,\C)$-parabolic bundle with full flags, and $\varphi$ is a strongly parabolic $\SL(k,\C)$-Higgs field on $\mathcal E$.
\end{definition}

\begin{definition}
A Higgs bundles $(\mathcal{E}, \varphi)$ is \emph{(semi)stable} if for every proper parabolic subbundle $\mathcal{S} \subset \mathcal{E}$ preserved by $\varphi$, we have $\mathrm{slope} \; \mathcal{S} <  \mathrm{slope} \; \mathcal{E}$ (respectively, $\mathrm{slope} \; \mathcal{S} \leq \mathrm{slope} \;  \mathcal{E}$).
\end{definition}
For generic weights, semistability implies stability.

\begin{example} Continuing \Cref{ex:holrank2}, on $\C\P^1$ we can assume $D \subset \C \subset \C\P^1$ by making a conformal transformation.  Every holomorphic bundle of rank $2$ of (holomorphic) degree $0$ is isomorphic to $\mathcal{O}(d) \oplus \mathcal{O}(-d)$.  Given flags $(F_1, \cdots, F_n)$,  a compatible Higgs field has the form \[\varphi=\sum\frac{\varphi_i}{z-p_i}\de z\] where each $\varphi_i=\mathrm{res}_{p_i} \varphi$ is strictly upper triangular with respect to the flag and $\sum\varphi_i=0$ (this ensures holomorphicity at $\infty$).
\end{example}

Fix $R>0$.  The $R$-nonabelian Hodge correspondence defines a one-to-one correspondence between polystable Higgs bundles $(\mathcal{E}, \varphi)$ and harmonic bundles $(\mathcal{E}, \varphi, h_R)$.  In the parabolic setting this was established by Simpson in \cite{Sim90}.
The weighted flag can be thought of boundary conditions for the hermitian metric $h_R$. There is a class of hermitian metrics on the holomorphic vector bundle $\mathcal{E}|_{C-D}$ called \emph{acceptable}. This is a condition on the norm of the curvature of the Chern connection $\nabla(\delbar_E, h)$
\cite[p.736]{Sim90}. The key property of acceptable hermitian metrics is that the growth rates of local meromorphic sections induces a filtration of $(E, \delbar_E)|_{C-D}$ by coherent subsheaves. This gives a weighted flag at each point $p \in D$. We say an acceptable hermitian metric on parabolic bundle $\mathcal{E}$ is \emph{adapted} if the weighted flag from the hermitian metric agrees with the weighted flag on $\mathcal{E}$. 
In this rank 2 case, the last condition is equivalent to the asymptotic description
    $$h\sim\begin{pmatrix}|z-p_i|^{2(1-\alpha_i)}&h_{12}\\h_{21}&|z-p_i|^{2\alpha_i}\end{pmatrix}$$
in any basis with $\langle e_1\rangle=F_i$, where $h_{12},h_{21}\in O(|z-p_i|)^{2\alpha_i}$.

\begin{theorem}[Nonabelian Hodge Correspondence]
    Fix $R>0$.  A parabolic Higgs bundle $(\mathcal{E},\varphi)$ is $\vec\alpha$-stable if and only if there exists a hermitian metric $h_R$ on the underlying complex vector bundle $E$ that (1) is adapted to the parabolic structure and (2)  solves \begin{equation} \label{eq:Hitchin}
 R^{-1} F_{\nabla(\delbar_E, h_R)}^\perp + R[\varphi, \varphi^{\dagger_{h_R}}]=0,
\end{equation}
where $F_{\nabla(\delbar_E, h_R)}^\perp$ is the trace-free part of the curvature of the Chern connection $\nabla(\delbar_E, h_R)$.
\end{theorem}

\begin{definition}
The equation in \eqref{eq:Hitchin} is called the \emph{$R$-rescaled Hitchin's equation}.
The hermitian metric $h_R$ solving \eqref{eq:Hitchin} 
and is known as the \emph{$R$-harmonic metric}. 
A triple $(\mathcal{E}, \varphi, h_R)$ consisting of a Higgs bundle $(\mathcal{E}, \varphi)$ and harmonic metric $h_R$ solving \eqref{eq:Hitchin} is called an \emph{$R$-harmonic bundle}.
\end{definition}

\bigskip

\subsubsection{\texorpdfstring{Hitchin Moduli Spaces and Strongly Parabolic $\SL(2,\C)$-Higgs Bundle Moduli Spaces}{Hitchin Moduli Spaces and Strongly  Parabolic SL(2,C)-Higgs Bundle Moduli Space}}
The Hitchin moduli space $\mathcal{M}_R(\vec \alpha)$ consists of $R$-harmonic bundles $(\mathcal{E}, \varphi, h_R)$, up to equivalence. The Higgs bundle moduli space $\mathcal{M}^{\text{Higgs}}(\vec \alpha)$ consists of Higgs bundles $(\mathcal{E}, \varphi)$, up to equivalence.
The Hitchin moduli space and Higgs bundle moduli space when $D=\emptyset$ were introduced and studied in \cite{Hit87}, and analogues of important properties (e.g. hyperk\"ahler structure, algebraic completely integrable system structure, spectral data, etc.) continue to hold in the parabolic setting. 
In \cite{Yokogawa}, Yokogawa constructed the coarse moduli space of semistable strongly parabolic Higgs bundles, fixing a Riemann surface $(C,D)$ and a parabolic line bundle $\det \mathcal E$.  In \cite{Kon02}, Konno constructed the Higgs bundle moduli spaces and Hitchin moduli spaces analytically, proving that for generic parabolic weights, the Hitchin moduli space is a smooth manifold with a hyperk\"ahler structure. 

\bigskip
We briefly explain the analytic constructions, because we are interested in the hyperk\"ahler structure; particularly, we are interested in the sense in which $\mathcal{M}^{\text{Higgs}}(\vec \alpha)$ arises as an infinite-dimensional holomorphic symplectic quotient, and the sense in which $\mathcal{M}_R(\vec \alpha)$ arises as an infinite-dimensional hyperk\"ahler quotient. 

We begin with the smooth category. The moduli space that we are describing here might be called the \emph{algebraic Dolbeault moduli space}, mirroring the terminology in \cite{BiquardBoalch}. Fix a Riemann surface---here $C=\mathbb{CP}^1$ and divisor $D=\{p_1, \cdots, p_n\}$. By making a conformal transformation, we can assume, without loss of generality, that $D \subset \C \subset \mathbb{CP}^1$. Fix the trivial complex vector bundle $E \to \CP^1$ with weighted flags at $p_i \in D$ determined by $(F_i, \alpha_i)$ (see  \Cref{rem:simplified}). On $\det E$, fix a holomorphic structure $\delbar_{\det E}$ (here $(\det E, \delbar_{\det E})\simeq \mathcal{O}$ since $\deg E =0$), induced parabolic weights $1$ at $p_i \in D$ and a hermitian metric $h_{\mathrm{det}}$ adapted to the parabolic weights such that $F^{\perp}_{\nabla(\delbar_{\det E}, h_{\det E})}=0$, as described in \eqref{eq:hdet}.  Moreover, fix $h_0$ on $E$ adapted to the parabolic structure, as in \cite[p. 255]{Kon02}, and call it the \emph{background metric} on the underlying complex vector bundle $E$. 

Let $\mathfrak G=\mathcal E(\SU(2))$ be the group of smooth special unitary gauge transformations which preserve each flag $\mathcal F_i$, i.e. $g|_{p_i}$ is upper triangular with respect to $0 \subset F_i \subset E_{p_i}$; let $\mathfrak G_\C=\mathcal E(\SL(2,\C))$ be the similar group of smooth complex gauge transformations.  We also use $\mathcal E(\mathfrak{su}(2))$ and $\mathcal E(\mathfrak{sl}(2,\C))$ for the analogous Lie algebras.  

The algebraic Dolbeault moduli space is the space of smooth pairs $(\delbar_E, \varphi)$ where $\delbar_E$ is a smooth holomorphic structure on $E$ which induces the fixed $\delbar_{\det E}$ on $\det E$, and $\varphi$ is a global section of $\Omega^{1,0}(C, \mathfrak{sl}(E)) \otimes \mathcal{O}(D)$ where (1) $\delbar_E \varphi = 0$ and (2) $\mathrm{res}_{p_i} \varphi$ is strictly upper triangular with respect to the flag $\mathcal F_i$, all considered up to equivalences
\begin{equation}
(\delbar_E, \varphi) \mapsto (g \circ \delbar_E \circ g^{-1}, g \varphi g^{-1})
\end{equation}
for $g$ in the smooth complex gauge group $\mathfrak{G}_{\C}$.  It is also possible to make the space of Higgs bundles and complex gauge group smaller by fixing a framing at $p_i \in D$ (see \cite{CFW24}) though we will not do this.

\bigskip
One can't prove good properties for this algebraic Dolbeault moduli space. Instead, one defines an analytic version of the Dolbeault moduli space where one considers an appropriate analytic completion of the space $\mathcal{D}$  of smooth pairs $(\delbar_E, \varphi)$ such that there is a bijective map $\iota: \mathcal{D}/\mathfrak{G}_{\C} \rightarrow \widehat{\mathcal{D}}/\widehat{\mathfrak{G}}_{\C}$ as sets. Define $\mu_{\C}(\delbar_E, \varphi) = \sqrt{-1}\delbar_E \varphi$. The space $\widehat{\mathcal{D}}$ is the kernel of $\mu_{\C}$ in some configuration space $\widehat{\mathcal{A}}_{J_1}$, where the subscript $J_1$ marks part of the hyperk\"ahler structure we will soon discuss. For example, in \cite{Kon02}, $\widehat{\mathcal{A}}$ is based on the $L^p$-theory in \cite{Biquard}, following \cite{LockhartMcOwen}; in \cite{CFW24}, $\widehat{\mathcal{A}}$ is weighted Sobolev spaces $L^{2,1}_{\vec \delta}$; one could also use $b$-weighted H\"older spaces and Sobolev spaces from geometric microlocal analysis, as in \cite{FMSW21}.  We discuss the weighted $b$-Sobolev spaces in \Cref{sec:analytic}, though we purposely leave the configuration space $\widehat{\mathcal{A}}$ unspecified here.  The goal is to show that the moduli space $\mathcal{M}^{\mathrm{Higgs}, \mathrm{an}, R}=\widehat{\mathcal{D}}/\widehat{\mathcal{G}}^{\C}$ is a smooth holomorphic symplectic manifold. In these types of arguments:
\begin{enumerate}
\item[(1)] The holomorphic symplectic structure is constructed via holomorphic symplectic quotient \cite[p. 256]{Kon02}.
Just as the hyperpolygon space $\mathcal{X}(\vec \beta) = X\twoslash_{(0, \vec \beta)} G_\C$, the holomorphic symplectic structure arises from constructing the Higgs bundle moduli space as an infinite-dimensional holomorphic symplectic quotient
    \[ \mathcal{M}^{\mathrm{Higgs}, \mathrm{an}, R}(\vec \alpha) = \mathcal{A}\!\!\twoslash\limits_{(0,\vec\alpha)}\!\!\mathfrak{G}_{\C}= \mu_\C^{-1}(0)^{\vec\alpha\text{-st}} /\mathfrak{G}_{\C}.\]
\item[(2)] The tangent space to $\mathcal{M}^{\mathrm{Higgs}, \mathrm{an}, R}$ is identified with the cohomology of the complex-analytic deformation complex
\[ 0 \to C^0 \overset{d^0}{\to} C^1  \overset{d^1}\to C^2 \to 0, \]
where $d^0$ and $d^1$ are given by the operator
\begin{equation}\label{eq:D''}\mathbb{D}''_R:=R^{-1/2} \delbar_E + R^{1/2} \varphi, \end{equation} $C^0$ is the space of infinitesimal complex gauge transformations, $C^1$ the space of infinitesimal Higgs bundle deformations, and $C^2$ represents the obstruction. Note that 
we've added the factor of $R>0$, but the holomorphic symplectic structure is independent of $R$.
\end{enumerate}
The manifold structure of $\mathcal{M}^{\mathrm{Higgs, an}, R}$ comes from the isomorphism with the Hitchin moduli space.

\bigskip

Similarly, one also constructs the Hitchin moduli space analytically in order to get the full hyperk\"ahler structure. 
There is a map from the deformation space $\widehat{\mathcal{A}}_{J_1} \to \mathcal{A}^{h_0}$ given by 
    \[(\delbar_E, \varphi) \mapsto (\nabla=\delbar_E + \del_E^{h_0}, \Psi = \varphi - \varphi^{\dagger_{h_0}}),\]
where $\nabla, \Psi$ are $h_0$-unitary.
The induced complex gauge group action is 
\begin{equation}
\nabla \to g \circ \delbar_E \circ g^{-1} + (g^{-1})^{\dagger_{h_0}} \circ \del_E^{h_0} \circ g^{\dagger_{h_0}} \qquad \Psi \mapsto g \varphi g^{-1} + (g^{-1})^{\dagger_{h_0}} \varphi^{\dagger_{h_0}}g^{\dagger_{h_0}}.
\end{equation}
When $g$ is not unitary, one can view this as a change of hermitian metric.
Inside the complex gauge group $\widehat{\mathcal{G}}_{\C}$ there is a real gauge group $\widehat{\mathcal{G}}$ of $\SU(E)$ gauge transformations. 

The main idea of the construction is to realize Hitchin's equations as the level sets of the moment maps of the unitary gauge action, namely $\mu_{\C}(\nabla , \Psi) =0$ and $\mu_{\R, R}(\nabla, \Psi) =0$ where
\begin{align}
\mu_{\C}(\nabla , \Psi) &:= i \nabla^{0,1} \Psi^{1,0} = i \delbar_E \varphi \nonumber\\
    \mu_{\R, R}(\nabla, \Psi) &:= -\frac{1}{2}\left(R^{-1}F^\perp_\nabla + R \Psi \wedge \Psi^{\dagger_{h_0}}\right)=-\frac{1}{2} \left(R^{-1} F^\perp_{\nabla(\delbar_E, h_0)} + R [\varphi, \varphi^{\dagger_{h_0}}]\right). \label{eq:Hitchinmomentmap}
\end{align}
The Hitchin moduli space is the space of solutions of Hitchin's equations up to unitary gauge transformations, and it has a rich structure.

In these types of constructions:
\begin{enumerate}
\item[(1)] The hyperk\"ahler structure is constructed via an infinite-dimensional hyperk\"ahler quotient\footnote{Note the analogy with the space of unitary hyperpolygons $\mathcal{X}(\vec \beta)=X\fourslash_{(0, \vec \beta)} G$, }
    \[\mathcal M_R(\vec\alpha)=\mathcal A^{h_0}\fourslash\limits_{(0,0)}\mathfrak G=(\mu_{\C,R}^{-1}(0)\cap\mu_{\R,R}^{-1}(0))/\mathfrak G.\]
The $R$-dependent hyperk\"ahler structure $(g_R, J_1, J_{2,R}, J_{3,R}, \omega_{J_1,R}, \omega_{J_2}, \omega_{J_3})$ on $\mathcal{A}^{h_0}$ is described in detail in \cite{FMSWforthcoming}\footnote{Our conventions match theirs with $\lambda_1 = \frac{R}{2}, \lambda_2=R^{-2}, \vartheta =0$.}, and $J_1$ and $\Omega_{J_1}=\omega_{J_2} + i \omega_{J_3}$ are independent of $R$\footnote{
    The space $\mathcal{A}^{h_0}$ is an affine space modeled on $L^{1,2}_{\vec \delta}(E(\mathfrak{su}(2)) \otimes T^*\CP^1)$ centered at some fixed $(\nabla_0, \Psi_0)$ such that, in an $h_0$-unitary frame near $p_i$ with $e_i \in F_i$,
        \[\delbar_{E,0} = \delbar  + \frac{1}{2} \begin{pmatrix} 1- \alpha_i & 0 \\ 0 & \alpha_i \end{pmatrix} \frac{\de \overline{z}}{\overline{z}-\overline{p}_i}.\]
    The hyperk\"ahler structure is 
    \begin{align}
        \omega_{J_1,R}((\dot \nabla_1,\dot\Psi_1),(\dot \nabla_2,\dot\Psi_2))
            &=\Im\left(\int_C R^{-1}\tr((\dot\nabla_1^{0,1})^{\dagger_{h_0}}\wedge\dot \nabla_2^{0,1})+R\tr((\dot\Psi_1^{1,0})^{\dagger_{h_0}}\wedge\dot\Psi_2^{1,0})\right),
                \nonumber\\
        \Omega_{J_1}((\dot \nabla_1,\dot\Psi_1),\dot \nabla_2,\dot\Psi_2)
            &=-\int_C\dot \nabla_1^{0,1}\wedge\dot\Psi_2^{1,0}-\dot \nabla_2^{0,1}\wedge\dot\Psi_1^{1,0}
                \nonumber\\
        J_{1,R}(\dot \nabla^{0,1},\dot\Psi^{1,0})&=(i\dot \nabla^{0,1},i\dot\Psi^{1,0})
                \nonumber\\
        J_{2,R}(\dot \nabla^{0,1},\dot\Psi^{1,0})&=(R(\dot\Psi^{1,0})^{\dagger_{h_0}},R^{-1}(\dot \nabla^{0,1})^{\dagger_{h_0}})
                \nonumber\\
        J_{3,R}(\dot \nabla^{0,1},\dot\Psi^{1,0})&=(iR(\dot\Psi^{1,0})^{\dagger_{h_0}},iR^{-1}(\dot \nabla^{0,1})^{\dagger_{h_0}}).
    \end{align}
    These analytic spaces are designed so that the integrands are integrable.  In the algebraic formulation, both Higgs field deformations have simple poles, which when paired could give a $\frac{1}{r}$ term in the integrand. The more restricted analytic deformation spaces prohibit this. 
}.
\item [(2)] The tangent space to the moduli space $\mathcal{M}$ is identified with the elliptic deformation complex
\[ 0 \to C^0 \overset{d^0}{\to} C^1  \overset{d^1}\to C^2 \to 0, \]
where
\begin{align*}
    d^0(\dot{\gamma})
        &=  \left(-R^{-1/2}d_\nabla \dot{\gamma}, -R^{1/2}[\Psi, \dot{\gamma}]\right)
\end{align*}
and
\begin{align*}
    d^1(\dot \nabla, \dot \Psi) 
        &=(d^{1,\R},\mathbb{D}''_R)(\dot \nabla, \dot \Psi)\\
        &= ( R^{-1/2}d_\nabla \dot{\nabla} + R^{1/2}[\Psi, \dot{\Psi}], R^{-1/2} \nabla^{0,1} \dot{\Psi}^{1,0}+ R^{1/2} [\dot{\nabla}^{0,1}, \Psi^{1,0}]);
\end{align*}
the components of $d^1$ are respectively the linearization of $\mu_\R$ and $\mu_\C$ at a solution of Hitchin's equations.
Here, $C^0$ is the space of infinitesimal \emph{unitary} gauge transformations, $C^1$ is infinitesimal deformations of pairs solving Hitchin's equations, and $C^2$ represents the obstruction. The tangent space is then identified with the kernel of $d^1 \oplus d_0^*$.
As Hitchin observes in \cite[p. 85]{Hit87} the kernel of $d^1 \oplus d_0^*$ is the kernel of 
\begin{equation}\label{eq:1st} \mathbb{D}'_{R,h_0} \oplus \mathbb{D}''_R\end{equation}
for $\mathbb{D}'_{R,h_0}$, the $h_0$-adjoint of $\mathbb{D}''_R$.    Elaborating a bit, the maps $d^{1,\R}$ and $d_0^*$ are real and imaginary parts of $\mathbb{D}'_{R,h_0}$\footnote{ In the case of the $4d$ anti-self-dual Yang-Mills equations (of which Hitchin's equations are a dimensional reduction), the linearization of $F_A^+=0$ and the Coulomb gauge condition can similarly be packaged into a single equation \cite[p. 55]{DonaldsonKronheimer}.}, as in \Cref{prop: symplectic quotient tangent space}.
(For more discussion see \cite[Proposition 2.2]{Fre19}.)
Lastly, 
\item [(3)] one uses a Kuranishi slice within the tangent space to give local coordinate charts on $\mathcal{M}$. Because $\mathcal{M}$ admits a smooth structure, $\mathcal{M}$ is a smooth hyperk\"ahler manifold.  Note that this is the only place smoothness is established for any of the moduli spaces! 
\end{enumerate}

\begin{remark}[Moment Maps are Distribution Valued] \label{rem:momentmaps} Having established the moment maps of the gauge group in \eqref{eq:Hitchinmomentmap}, we can now observe why we set $ \alpha_i(R) = \frac{1}{2}- \beta_i$.  Note that $\partial_{\overline{z}} \frac{1}{z} =\pi \delta_0$, the Dirac distribution supported at $0$. Consequently, 
in a coordinate $z$ centered at $p_i \in D$, for we have
\begin{align*}
    R^{-1}F_D^\perp + R[\varphi, \varphi^{\dagger_{h_R}}]
        &=R^{-1} \begin{pmatrix} \frac{1}{2}-\alpha_i(R)& \\ &  \alpha_i(R)- \frac{1}{2} \end{pmatrix}   \pi \delta_{p_i} d \overline{z} \wedge d z. 
\end{align*}
With our choice of $\alpha_i(R)= \frac{1}{2} - R \beta_i$, we are effectively holding the $R$-dependent distribution supported at $D$ constant.

While we restrict to the strongly parabolic case in this paper, note that in the weakly parabolic case and in a frame where 
    \[\delbar_E =\delbar, \varphi= \begin{pmatrix}m_{p_i} & * \\ 0 & -m_{p_i} \end{pmatrix} \frac{dz}{z} + \mathrm{hol},\]
one has 
    \[ \delbar_E \varphi =\begin{pmatrix}m_{p_i} & * \\ 0 & -m_{p_i} \end{pmatrix} \pi \delta_{p_i} d \overline{z} \wedge d z. \]
One can view our choice of $m_i=0$ as choosing the value $\mu_{\C, R}$ to be the zero distribution, rather than a more general distribution supported at $D$.
\end{remark}

\subsubsection{Tangent Space and Hyperk\"ahler Metric}\label[section]{sec:hitchinhk} We will unpack the above discussion in terms of triples $(\delbar_E, \varphi, h_R)$ solving the $R$-rescaled Hitchin's equations, choosing the background metric to be $h_R$ itself. There are two interesting operators:
\begin{align*}
\mathbb{D}''_R&=R^{-1/2} \delbar_E + R^{1/2}[\varphi, \cdot]\\
\mathbb{D}'_{R, h_R}&=R^{-1/2} \del_E^{h_R} + R^{1/2}[\varphi^{*_{h_R}}, \cdot].
\end{align*}
Note that $(\mathbb{D}''_R)^2=0$ since $\delbar_E \varphi=0$ and similarly $(\mathbb{D}'_{R, h_R})^2=0$.  Because $R^{-1} F_{\nabla(\delbar_E, h_R)}+ R[\varphi, \varphi^{*_{h_R}}]=0$, we have $F_{\mathbb{D}_R''+\mathbb{D}'_{R,h_R}}=0$, hence $\mathbb{D}'_{R, h_R} \mathbb{D}''_R=\mathbb{D}''_R \mathbb{D}'_{R, h_R}$.

We will call such a deformation in the kernel of the operator $\mathbb{D}''_R \oplus \mathbb{D}'_{R, h_R}$ appearing in \eqref{eq:1st} \emph{harmonic}, by analogy with the Hodge theorem.
Said another way, given a Higgs bundle $(\delbar_E, \varphi)$ for which $h_R$ is the $R$-harmonic metric, consider the family of deformations 
\begin{eqnarray} \label{eq:defhol}
 (\delbar_E)_\eps &=& \delbar_E + \eps \dot \eta\\ \nonumber
 \varphi_\eps &=& \varphi + \eps \dot \varphi
\end{eqnarray}
where $\dot \eta$ is $(0,1)$-valued in $\End E$ (giving an infinitesimal deformation of the holomorphic structure), $\dot{\varphi}$ is $(1,0)$-valued in $\End E$, and 
\begin{equation}
 \delbar_E \dot \varphi + [\dot \eta, \varphi]=0
\end{equation}
so that $(\dot{\eta}, \dot{\varphi})$ solves the infinitesimal version of
the Higgs bundle equation $\delbar_E \varphi=0$ (i.e. $R^{-1/2} \dot{\eta} + R^{1/2} \dot{\varphi}$ is in the kernel of $\mathbb{D}''_R$).
Notice that $[(\dot{\eta}_1, \dot{\varphi}_1)]=[(\dot{\eta}_2, \dot{\varphi}_2)]$ if there is an infinitesimal $\End E$-valued gauge transformation $\dot{\gamma}_R$  such that 
\begin{eqnarray} \label{eq:infinitesimal}
 \dot{\eta}_2 - \dot{\eta}_1 &=& -\delbar_E \dot{\gamma}_R\\ \nonumber
 \dot{\varphi}_2 - \dot{\varphi}_1 &=& -[\varphi, \dot{\gamma}_R],
\end{eqnarray}
i.e., the difference is in the image of $\mathbb{D}''_R$.  
To find the harmonic representative in the hypercohomology class $[(\dot{\eta}, \dot{\varphi})]$, denoted 
\begin{equation}
\label{eq:harmonic}
  \mathtt{H}_R:=R^{-1/2}\underbrace{\left(\dot{\eta} - \delbar_E \dot{\gamma}_R\right)}_{\mathtt{H}_R^{0,1}} + R^{1/2}\underbrace{\left(\dot{\varphi}-[\varphi, \dot{\gamma}_R]\right)}_{\mathtt{H}^{0,1}_R},
\end{equation} 
we must solve for $\dot{\gamma}_R$ in some appropriate analytic space such that  
\begin{equation}\label{eq:ingaugetriple}
 \mathbb{D}'_{R, h_R} \mathbb{D}''_{R} \dot{\gamma}_R = \mathbb{D}'_{R, h_R}(R^{-1/2} \dot{\eta} + R^{1/2} \dot{\varphi}),
\end{equation}
i.e. 
\[ R^{-1} \del_E^{h_R} \delbar_E \dot{\gamma}_R- R^{-1}\del_E^{h_R} \dot{\eta} + R\left[\varphi^{\dagger_{h_R}}, [\varphi, \dot{\gamma}_R] - \dot{\varphi}\right]=0.\]

The natural hyperk\"ahler metric is 
\begin{equation}
|(\dot{\eta}, \dot{\varphi})|^2_R = 2\int_{C} R^{-1} \langle \mathtt{H}_{R}^{0,1}, \mathtt{H}_{R}^{0,1}\rangle_{h_R} +  R \langle \mathtt{H}_{R}^{1,0}, \mathtt{H}_{R}^{1,0}\rangle_{h_R},
\end{equation}
where $\langle \alpha, \beta \rangle = \mathrm{Tr}(\alpha \wedge \star \beta^{\dagger_{h_R}})$ is a two-form.
As in \cite{Fre19}, by integrating by parts and using the harmonicity above, we can write the hyperk\"ahler metric as: \begin{equation}\label{eq: simplified hk expression}
|(\dot{\eta}, \dot{\varphi})|^2_R = 2\int_{C} R^{-1} \langle\mathtt{H}_{R}^{0,1}, \dot{\eta}\rangle_{h_R} +  R \langle\mathtt{H}_{R}^{1,0}, \dot{\varphi}\rangle_{h_R}.
\end{equation}
\begin{remark}\label[remark]{rem: gamma deforms metric}
Geometrically, one can view the $\mathfrak{sl}(E)$-valued section $\dot{\gamma}_R$ as giving a deformation of the hermitian metric 
\begin{equation}
 h_{R,\eps}(w_1,w_2) = h_R(\e^{\eps \dot{\gamma_R}} w_1, \e^{\eps \dot{\gamma_R}} w_2).
\end{equation}
This only depends on $h_R$-hermitian part of $\dot{\gamma}_R$ and not the part valued in $\mathfrak{su}(E, h_R)$.
\end{remark}

\subsection{Map from Hyperpolygons to Parabolic Higgs Bundles}
\label{subsec: Hyperpolygons to Higgs Map}
Let $\vec\alpha\in(0,\frac12)^n$ and $\vec\beta\in(0,\infty)^n$.  We define a map from  hyperpolygon space to the moduli space of strongly parabolic Higgs bundles on $(\C\P^1, D=\{p_1, p_2, \cdots, p_n\})$:
\begin{align} \mathcal{T}: \mathcal{X}(\vec \beta) &\to \mathcal{M}^{\mathrm{Higgs}}(\vec \alpha) \\ \nonumber 
(\mathbf x,\mathbf y) &\mapsto (\mathcal E_{(\bx, \by)}, \varphi_{(\bx, \by)})\end{align}
for $\beta_i = 1 - 2\alpha_i$.
Let $\mathcal E=\mathcal E_{(\bx,\by)}$ be the trivial rank-2 bundle on $C=\C\P^1$ with standard holomorphic structure $\delbar$ and parabolic structure given by $F_i=\langle x_i\rangle$ and $\alpha_i$ for each $i=1,\dots,n$, and let\footnote{The assumption $\mu_\C(\mathbf x,\mathbf y)=0$ ensures holomorphicity at $\infty$ and $\tr\varphi=0$.}
    $$\varphi=\varphi_{(\mathbf x,\mathbf y)}=\sum_{i=1}^n\frac{x_iy_i}{z-p_i}\de z.$$
The following theorem was proved in \cite{GM11}, though we provide a minor correction (see \Cref{rmk: chamber condition is necessary}).
\begin{theorem}\label[theorem]{thm: map from hyperpolygon to Higgs}
    Suppose $\vec\alpha$ is generic and satisfies\footnote{
        Note that this condition corresponds to 
        \begin{equation}\label{eq:Wbeta}
            W_{[n]}(\vec \beta) < 2.
        \end{equation}
    }
     \begin{equation}\label{eq:missingassumption}
    W_{[n]}(\vec\alpha):=\sum_{i=1}^n\alpha_i>\frac{n-2}{2}, \end{equation} and set $\beta_i=1-2\alpha_i$ for $i=1,\dots,n$.  Then the map $\mathcal T$ constructed above sends $\vec\beta$-stable hyperpolygons to $\vec\alpha$-stable Higgs bundles.
\end{theorem}
\begin{proof}   
    Suppose $(\mathbf x,\mathbf y)$ is a $\vec\beta$-stable hyperpolygon, as in \Cref{def: stable hyperpolygon}. Consider the Higgs bundle constructed above. Since $\det \mathcal{E}$ is $\mathcal{O}$ with parabolic weight $(1-\alpha_i)+\alpha_i =1$ at $p_i \in D$, $\text{pardeg}\,\mathcal E=n$, so $\text{slope}\,\mathcal E=\frac n2$.  To prove stability, we need only consider holomorphic line subbundles $\mathcal L\subset\mathcal E$ with the induced parabolic structure, since that maximizes slope.  Since $\mathcal E=\mathcal O_{\P^1}^2$, such a holomorphic subbundle is isomorphic to $\mathcal O(-d)$ for some $d\geq0$.  We consider the $d=0$ and $d>0$ cases separately.

    First we consider subbundles of the form $\mathcal L\cong\mathcal{O}$ preserved by $\varphi$. The inclusion $\mathcal L\cong\mathcal O\into\mathcal E$ is given by a $2\times1$ matrix $M$ of constant global sections.  Set
        $$I_{\mathcal L}=\{i\mid L_{p_i}=F_i\},$$
    which is necessarily a straight subset (see \Cref{def:straight}) of $[n]$ because $x_i\in F_i$ is parallel to $M$ as vectors. Furthermore, $y_i=0$ for any $i\notin I_{\mathcal L}$, as otherwise the residue $\varphi_i=x_iy_i$ for some $i\notin I_{\mathcal L}$ will have image $\im\varphi_i=F_i$ not contained in $L_{p_i}$.   Therefore, the $\vec\beta$-stability of $(\mathbf x,\mathbf y)$ implies $W_{I_{\mathcal L}}(\vec \beta)<0$ (See \Cref{def:straight}).  We have
        $$\text{slope}\,\mathcal L=\sum_{i\in I_{\mathcal L}}(1-\alpha_i)+\sum_{i\notin I_{\mathcal L}}\alpha_i=\sum_{i\in I_{\mathcal L}}\left(1-\frac12(1-\beta_i)\right)+\sum_{i\notin I_{\mathcal L}}\frac12(1-\beta_i)=\frac n2+\frac12W_{I_{\mathcal L}}(\vec\beta)<\frac n2$$
    so $\mathcal L$ does not destabilize $\mathcal E$.

    Now suppose $\mathcal L=\mathcal O(-d)\into\mathcal E$ for $d>0$.  The highest possible slope $\mathcal L$ can have occurs when 
    $I_{\mathcal L}=[n]$, meaning $\mathcal L$ meets all the 1-dimensional levels $F_i$ of the flags in $\mathcal E$ nontrivially.  In this case, 
        $$\text{slope}\,\mathcal L=-d+\sum_{i=1}^n(1-\alpha_i)\!=\!-d+\sum_{i=1}^n\left(1-\frac12(1-\beta_i)\right)\!=-d+\frac n2+\frac{1}{2}W_{[n]}(\vec\beta)<-d + \frac n2  +1<\frac n2$$
    so again $\mathcal L$ does not destabilize $\mathcal E$.   
    
    This proves $\mathcal T\from\mathcal X(\vec\beta)\to\mathcal M^\text{Higgs}(\vec\alpha)$ is well-defined.
\end{proof}\
\begin{remark}\label[remark]{rmk: chamber condition is necessary}
    In the case where $W_{\{i\}}(\vec\beta)<0$ for all $i$, there always exists a stable hyperpolygon of the form $(\bx,\mathbf 0)$ whose corresponding Higgs bundles admit subbundles of the form $\mathcal L\cong\mathcal O(-1)\into\mathcal E$ such that $I_{\mathcal L}=[n]$.\footnote{
        Let $\mathcal L=\mathcal O(-1)\into E$ be any holomorphic embedding of the tautological bundle, and pick $x_i\in \mathcal L_{p_i}$ for each $i$.  Then $(\bx,\mathbf 0)$ is stable because no subset with $|I|\geq2$ is straight, and it has the desired property.
    }  Thus, the assumption $W_{[n]}(\vec\alpha)>(n-2)/2$ is necessary for this construction to preserve stability.  In fact, the set of problematic hyperpolygons is precisely the $\C^\times$-upward flow of this $(\bx,\mathbf0)$ (see \Cref{subsec: The U(1)-Action and Moment Map}).
\end{remark}
\begin{theorem}\label[theorem]{thm: image of T}
    The map $\mathcal T\from\mathcal X(\vec\beta)\to\mathcal M^\text{Higgs}(\vec\alpha)$ defined above is a bijection onto the locus of Higgs bundles with underlying holomorphic bundle isomorphic to $\mathcal O_{\C\P^1}^2$.
\end{theorem}
\begin{proof} Define \[\mathcal{M}_0^{\mathrm{Higgs}}(\vec \alpha)=\{[(\mathcal E, \varphi)] \in \mathcal{M}^{\mathrm{Higgs}}(\vec \alpha): (E, \delbar_E) \simeq O_{\C\P^1}^2 \text{ as holomorphic bundles}\}. \]
    We prove surjectivity and refer the reader to \cite{GM11} for the proof of injectivity.  Suppose $(\mathcal E,\mathcal F,\varphi)\in\mathcal M_0^\text{Higgs}(\vec\alpha)$.  Since there is a trivialization in which $\delbar_E = \delbar$, the Higgs field can be written as $\varphi=\sum\frac{\varphi_i}{z-p_i}\de z$ where each $\varphi_i\in\End_0(\mathcal E)$.  For all $i$, pick $x_i\in\Hom(\C,\C^2)\cong\C^2$ which spans $F_i\subset E_{p_i}$ and $y_i\in\Hom(\C^2,\C)$ such that $\varphi_i=x_iy_i$ (this can be done since $\varphi_i$ is nilpotent, and the resulting pair will have $y_ix_i=\tr(x_iy_i)=0$).  Since $\varphi$ is holomorphic at $\infty$ we have $\sum\varphi_i=\sum x_iy_i=0$, so $\mu_\C(\mathbf x,\mathbf y)=0$ and $(\mathbf x,\mathbf y)$ is a hyperpolygon. It is clear from the construction that $\mathcal T(\mathbf x,\mathbf y)=(\mathcal E,\varphi)$.

    We chose $x_i\ne0$ for all $i$, so to prove stability we need to show that if $I$ is a straight subset and $y_i=0$ for $i\notin I$, then $W_I(\vec\beta)<0$.  In this case, we have $F_i=\langle x_i\rangle=\langle x_j\rangle=F_j$ for all $i,j\in I$ and $\varphi_i=0$ for $i\notin I$.  If $I=\emptyset$ then $W_I<0$ is automatic, so assume there is some $i_0\in I$.  The line bundle $L\subset E$ given by the constant section $z\mapsto x_{i_0}\in E_{p_{i_0}}$ has degree 0, intersects $F_i$ nontrivially for all $i\in I$, and is preserved by $\varphi$ as $\varphi=\sum_{i\in I}\frac{x_iy_i}{z-p_i}\de z$.  Give $\mathcal L=L$ the induced parabolic structure from $\mathcal E$.  Then
        $$\text{slope}\,\mathcal L=\text{pardeg}\,\mathcal L\geq\sum_{i\in I}(1-\alpha_i)+\sum_{i\notin I}\alpha_i=\frac n2+\frac12 W_I(\vec\beta).$$
    We assumed $(\mathcal E, \mathcal F,\varphi)$ is stable, so $\text{slope}\,\mathcal L<\text{slope}\,\mathcal E=\frac n2$ and thus $W_I(\vec\beta)<0$.  This proves $(\mathbf x,\mathbf y)$ is a $\vec\beta$-stable hyperpolygon.  
\end{proof}

The key result about this map is 
\begin{theorem}[\cite{BFGM15}]\label[theorem]{thm: T preserves holomorphic symplectic forms}
The holomorphic symplectic structure $(J_1, \Omega_{J_1})$ on the space $\mathcal{M}^{\mathrm{Higgs}}(\vec \alpha)$ and the 
holomorphic symplectic structure $(J^{HP}_1, \Omega^{HP}_{J_1})$ on $\mathcal{X}(\vec \beta)$ are related by
\[\mathcal{T}^*(J_1, \Omega_{J_1}) = (J^{HP}_1,2\pi\Omega^{HP}_{J_1}).\]
\end{theorem}

\begin{remark}\label[remark]{rem: our alpha convention}
    Varying $\vec\alpha$ and $\vec\beta$ within their chambers has no effect on the stability of Higgs bundles and hyperpolygons, so the map $\mathcal T$ remains well-defined.  It this case, we'll say $\vec\alpha$ and $\vec\beta$ are in \emph{corresponding chambers}.  In the above, we followed the conventions of \cite{GM11,BFGM15} and used $\alpha_i=\frac12(1-\beta_i)$ for all $i$.  Going forward, we shall instead be interested in setting $\alpha_i(R)=\frac12-R\beta_i$ for any generic $\vec\beta$ and $R>0$ small enough that $W_{[n]}(\vec\alpha(R))>(n-2)/2$.  In this case, the map $\mathcal T\from\mathcal X(\vec\beta)\to\mathcal M^\text{Higgs}(\vec\alpha(R))$ is still well-defined.
\end{remark}

\subsection{Induced Map on Tangent Spaces}\label[section]{sec:mapontan}

Let $\vec\beta \in (0, \infty)^n$ be generic and let $\vec\alpha$ be in the corresponding chamber as discussed in \Cref{rem: our alpha convention}.  The map $\mathcal T\from\mathcal X(\vec\beta)\to\mathcal M^\text{Higgs}(\vec\alpha)$ induces the map of tangent bundles 
\begin{equation}
\begin{array}{lrlrl}d\mathcal{T}&: &T\mathcal{X}(\vec \beta) &\to &T\mathcal{M}^{\mathrm{Higgs}}(\vec \alpha) \\ \nonumber 
d\mathcal{T}\Big|_{(\bx, \by)} &: &T_{(\bx, \by)} \mathcal{X}(\vec \beta) &\to &T_{\mathcal{T}(\bx, \by)}\mathcal{M}^{\mathrm{Higgs}}(\vec \alpha) \\ \nonumber 
& &(\dot\bx,\dot\by) &\mapsto &(\dot{\eta}, \dot{\mathcal{F}}, \dot{\varphi})
\end{array}
\end{equation}

\begin{lemma}
Given $(\dot \bx, \dot \by) \in T_{(\bx, \by)} \mathcal{X}(\vec \beta)$, the associated $(\dot{\eta}, \dot{\mathcal{F}}, \dot{\varphi})$ has $\dot{\eta}=0$, 
\[ \dot{\varphi} = \sum_{i=1}^n \frac{\dot{x}_i y_i + x_i \dot y_i}{z-p_i} \de z.\]
The flag $F_i=\langle x_i \rangle$ deforms by $\dot{F}_i = \langle\dot{x}_i \rangle,$ i.e. 
\begin{align*}
    \mathrm{res}_{p_i}(\varphi + \eps \dot{\varphi})(x_i + \eps \dot{x_i})= O(\eps^2).
\end{align*}
\end{lemma}
\begin{proof}
In this case, is is convenient to view 
$(\dot \bx, \dot \by) \in T_{(\bx, \by)} \mathcal{X}(\vec \beta)$ 
 as the family
\[ (\bx_\eps, \by_\eps) = (\bx + \eps \dot \bx, \by + \eps \dot \by) \]
solving the hyperpolygon equations to first order in $\eps$, and differentiate.
Then, $\dot{\varphi}$ is as claimed.
Note that in the analytic framework of \Cref{sec:hitchinhk}, the flags are fixed, and all deformations can be captured by the change of holomorphic structure $\dot{\eta}$.  In contrast, here, the underlying holomorphic bundle doesn't change, i.e. $\dot{\eta}=0$; the deformation of the flag is tracked by the following computation: \begin{align*}
    \mathrm{res}_{p_i}(\varphi + \eps \dot{\varphi})(x_i + \eps \dot{x_i})
        &=(x_i y_i + \eps \dot{x}_i y_i + \eps x_i \dot{y_i})(x_i + \eps \dot{x_i})\\
        &=\eps\left(x_iy_i\dot x_i+x_i\dot y_ix_i\right) + O(\eps^2)\\
        &=O(\eps^2),
\end{align*}
where the second equality uses $y_ix_i=0$ and the last equality uses $\dot y_ix_i+y_i\dot x_i=0$ from \eqref{eqn: linearized moment map}.
\end{proof}
\begin{proposition}\label[proposition]{prop:nuflags}
Define the infinitesimal smooth gauge transformation 
\begin{equation}\label{eq:nuflags}
    \dot{\nu}_{\mathrm{flags}}:=\sum_{i} \chi_i \dot{\nu}_{\mathrm{flags}, i} \qquad \dot{\nu}_{\mathrm{flags},i}=-\frac{(\dot{x}_i x_i^\dagger)^\perp}{|x_i|^2}- x_i^\dagger \dot{x}_i \frac{(x_i x_i^\dagger)^\perp}{|x_i|^4},
\end{equation}
where $\chi_i$ is a bump function that is $1$ on $B_{\delta}(p_i)$ and $0$ outside $B_{2\delta}(p_i)$, and $\delta$ is small enough so that $\{B_{2 \delta}(p_i)\}$ do not intersect.
Note that $\dot{\eta}- \delbar_E \dot{\nu}_{\mathrm{flags}}$ vanishes on $B_{\delta}(p_i)$; consequently, the residue of the Higgs field deformation is defined. With this $\dot{\nu}_{\mathrm{flags}, i}$, the deformation of the flags with respect to $\dot{\varphi} + [\dot{\nu}_{\mathrm{flags}}, \varphi]$ is trivial, i.e.
\[     \mathrm{res}_{p_i} \left( \varphi + \eps (\dot{\varphi} + [\dot{\nu}_{\mathrm{flags}}, \varphi]) \right)(x_i) =O(\eps^2). \]
\end{proposition}
\begin{proof}
  Notice that
\begin{align}
-\dot\nu_{\text{flags},i}x_i
    &=\left(\frac{(\dot{x}_i x_i^\dagger)^\perp}{|x_i|^2}+ x_i^\dagger \dot{x}_i \frac{(x_i x_i^\dagger)^\perp}{|x_i|^4}\right)x_i
        \notag\\
    &=\frac{\dot x_ix_i^\dagger x_i}{|x_i|^2}-\frac{\tr(\dot x_ix_i^\dagger)}{2|x_i|^2}x_i+x_i^\dagger\dot x_i\frac{x_ix_i^\dagger x_i}{|x_i|^4}-x_i^\dagger\dot x_i\frac{\tr(x_ix_i^\dagger)}{2|x_i|^4}x_i
        \notag\\
    &=\dot x_i+ \left(-\frac{x_i^\dagger\dot x_i}{2|x_i|^2}+\frac{x_i^\dagger\dot x_i}{|x_i|^2}-\frac{x_i^\dagger\dot x_i}{2|x_i|^2}\right)x_i
        \notag\\
    &=\dot x_i,
\end{align}
so $-\dot\nu_{\text{flags},i}$ takes $x_i$ to $\dot x_i$, and thus $F_i$ to $\dot F_i$.
 Consequently, it makes sense to talk about the residue of the Higgs field deformation: 
 \begin{align*}
   \left(\mathrm{res}_{p_i} (\dot{\varphi} + [\dot{\nu}_{\mathrm{flags}}, \varphi]\right)(x_i)
        &=\left(\dot{\varphi}_i + [\dot{\nu}_{\mathrm{flags}, i}, \varphi_i]\right)(x_i)\\
    &=\left(\dot x_iy_i+x_i\dot y_i+\dot\nu_{\text{flags},i}x_iy_i-x_iy_i\dot\nu_{\text{flags},i}\right)(x_i)\\
    &=\dot{x_i}y_ix_i + x_i\dot y_ix_i -\dot{x_i} y_i x_i + x_i y_i \dot{x_i}\\
    &=0.
\end{align*}
\end{proof}

\begin{remark}[Notation: $\dot{\gamma}$ versus $\dot{\nu}$]
We will only use $\dot{\gamma}$ for infinitesimal gauge transformations which are in the appropriate analytic space described in \Cref{sec:hitchinhk}. We will use $\dot{\gamma}$ for more generic infinitesimal gauge transformations, i.e. those that infinitesimally change the flag. 
\end{remark}

\begin{remark}\label{rem: nu flag simple form}
    Fix an $i$.  We drop that subscript $x:=x_i$, $y:=y_i$ in favor of writing the components as $x=(x_1,0)^\intercal$ and $y=(0,y_2)$ in an appropriate frame, and similarly write $\dot x=(\dot x_1,\dot x_2)^\intercal$ and $\dot y=(\dot y_1,\dot y_2)$.  Then \eqref{eq:nuflags} becomes
    \begin{equation}\label{eqn: nu flag simple form}
        \dot\nu_{\text{flags},i}=\begin{pmatrix}-\frac{\dot x_1}{2x_1}&0\\-\frac{\dot x_2}{x_1}&\frac{\dot x_1}{2x}\end{pmatrix}+\begin{pmatrix}-\frac{\dot x_1}{2x_1}&0\\0&\frac{\dot x_1}{2x_1}\end{pmatrix}=\begin{pmatrix}-\frac{\dot x_1}{x_1}&0\\-\frac{\dot x_2}{x_1}&\frac{\dot x_1}{x_1}\end{pmatrix}.
    \end{equation}
    Returning the subscript $i$, now it is clear to see $-\dot\nu_{\text{flags,i}}x_i=\dot x_i$.
\end{remark}

\bigskip

In the rest of this paper, we write $\vec \alpha$ as a function of $R$.
We define the family of parabolic weights 
\begin{equation}
\alpha_i(R)=\frac12-R\beta_i, 
\end{equation}
which as $R\to0$ approaches the vertex $(\frac12,\frac12,\cdots,\frac12)$.  Here $R$ is the same rescaling factor in Hitchin's equation \eqref{eq:Hitchin}.  For shorthand, we shall denote the distinguished real symplectic form $\omega_{J_1,R,\vec\alpha(R)}$ on $\mathcal M_R(\vec\alpha(R))$ by $\omega_R$.  
Let $R_{\max}>0$ be the solution to \begin{equation}W_{[n]}(\vec\alpha(R_{\max}))=(n-2)/2,\end{equation}
 so that $\mathcal T\from\mathcal X(\vec\beta)\to\mathcal M^\text{Higgs}(\vec\alpha(R))$ is well-defined for all $R\in(0,R_{\max})$.

\bigskip

Define $\mathcal{T}_R: \mathcal{X}(\vec \beta) \to \mathcal{M}_R(\vec \alpha(R))$ via the following composition:
\begin{equation}\label{eq:TR}
\begin{array}{l l c l c l c l}
\mathcal{T}_R : &\mathcal{X}(\vec \beta) &\overset{\mathcal{T}}{\rightarrow} 
&\mathcal{M}^{\mathrm{Higgs}}(\vec \alpha) &\overset{\mathrm{id}}{\rightarrow} 
&\mathcal{M}^{\mathrm{Higgs}}(\vec \alpha(R)) &\overset{\mathrm{NAHC}_R}{\rightarrow} 
&\mathcal{M}_R(\vec \alpha(R))\\
&(\bx, \by) &\mapsto &(\mathcal E_{(\bx, \by)}, \varphi_{(\bx, \by)}) &\mapsto &(\mathcal E_{(\bx, \by)}, \varphi_{(\bx, \by)})&\mapsto&(\mathcal E_{(\bx, \by)}, \varphi_{(\bx, \by)}, h_R)
\end{array}
\end{equation}
We note that since the family $\vec \alpha(R)$ for $R \in (0,R_{\max})$ stays within a chamber, $\vec \alpha$-stable Higgs bundles coincide with $\vec \alpha(R)$-stable Higgs bundles, thus the identity map appears above.

In order to prove \Cref{thm: HK degeneration}, we need to evaluate the pullback $\mathcal T_R^*g_R$ on a deformation $(\dot\bx,\dot\by)\in T_{(\bx,\by)}\mathcal X(\vec\beta)$.  This is done by \emph{pushing forward} the unitary deformation $(\dot\bx,\dot\by)$ to the harmonic deformation of $\mathcal T(\bx,\by)$
in $\mathcal M_R(\vec\alpha(R))$. 
Let $(\dot\bx, \dot \by)$ be a unitary deformation of $(\bx, \by)$. 
In light of  \Cref{prop:nuflags}, the 
harmonic representative with background metric $h_R$ is 
\begin{equation}\label{eq:harmtan} \mathtt{H}_R := R^{-1/2}(-\delbar_E \dot{\nu}_R) + R^{1/2}(\dot{\varphi} - [\varphi, \dot{\nu}_R]), \end{equation}
for 
\[ \dot{\nu}_R = \dot{\nu}_{\mathrm{flags}} + \dot{\gamma}_{\mathrm{cor},R}, \]
where the correction term $\dot{\gamma}_{\mathrm{cor}, R}$ in the space of analytic infinitesimal gauge transformations satisfies \eqref{eq:ingaugetriple}
\begin{equation}
 \mathbb{D}'_{R, h_R} \mathbb{D}''_{R} \dot{\gamma}_{\mathrm{cor},R} = \mathbb{D}'_{R, h_R}(R^{-1/2} (\dot{\eta}- \delbar_E \dot{\nu}_{\mathrm{flags}})+ R^{1/2} (\dot{\varphi} +[\dot{\nu}_{\mathrm{flags}}, \varphi])),
\end{equation}
We will often write that 
\begin{equation}
 \mathbb{D}'_{R, h_R} \mathbb{D}''_{R} \dot{\nu}_R= \mathbb{D}'_{R, h_R}(R^{-1/2} \dot{\eta}+ R^{1/2} \dot{\varphi}),
\end{equation}
though this is more informal since $\dot{\nu}_R$ does not lie in a fixed function space. Expanding this, and using that $\delbar_E=\delbar$ and $\dot{\eta}=0$, we have
\begin{equation}\label{eqn: Coulomb gauge equation}
    0=R^{-1}\del^{h_R}\delbar\dot\nu_R-R[\varphi^{\dagger_{h_R}},\dot\varphi+[\dot\nu_R,\varphi]].
\end{equation}
Because of $\dot{\nu}_{\mathrm{flags}}$, the expression for the hyperk\"ahler metric typically now has a boundary term:
\begin{align}\label{eq:hk expression}
    \|(\dot\varphi,\dot \eta, \dot\gamma_R)\|_{g_R}^2
   &=\int_{\P^1}R^{-1}|\dot{\eta} - \delbar\dot\nu_R|_{h_R}^2+R|\dot\varphi+[\dot\nu_R,\varphi]|_{h_R}^2 
            \notag\\
        &=\int_{\P^1}R^{-1}\langle \dot{\eta}, \dot{\eta} - \delbar\dot\nu_R \rangle_{h_R}^2+R\langle\dot{\varphi}, \dot\varphi+[\dot\nu_R,\varphi]\rangle_{h_R}^2 +R^{-1} d\langle -\dot{\nu}_R, \dot{\eta}-\delbar\dot{\nu}_R \rangle
            \notag\\
        &=\sum_{i=1}^n\left(\lim_{\delta'\to0}\int_{\partial B_{\delta'}(p_i)}R^{-1}\langle\dot\nu_R,\delbar\dot\nu_R\rangle_{h_R} \right)+\int_{\mathbb{P}^1}R\langle\dot\varphi,\dot\varphi+[\dot\nu_R,\varphi]\rangle_{h_R}.
\end{align}

\begin{remark}[Background Metric $h_0=\mathrm{Id}$]\label[remark]{rem:backgroundh0}
In much of the paper, it will be convenient to fix the background metric $h_0=\mathrm{Id}$ on the trivial complex vector bundle $E$ underlying $\mathcal{O} \oplus \mathcal{O}$. Let $h_R^{1/2}$ be an $h_0$-hermitian matrix such that in the above trivialization, $h_R( \cdot, \cdot) = h_0(h_R^{1/2}, h_R^{1/2})$.
The $h_0$-unitary pair solving the $R$-rescaled Hitchin's equations is $(\nabla, \Psi)$ with 
\[ \nabla^{0,1} =  h_R^{1/2}  \circ \delbar \circ h_R^{-1/2} ,  \qquad \Psi^{1,0}= h_R^{1/2} \varphi  h_R^{-1/2}.
 \]
The associated harmonic representative of the tangent space is similarly
\begin{equation} \label{eq:defun}
 \dot \nabla^{0,1}
=h_R^{1/2}  \left( \dot{\eta} - \delbar_E \dot{\gamma}_R \right) h_R^{-1/2} \qquad \dot{\Psi}^{1,0} = h_R^{1/2} \left(\dot{\varphi} + [\dot{\gamma}_R, \varphi] \right)  h_R^{-1/2}.
\end{equation}
\end{remark}

\subsection{Relevant Analytic Setup}\label[section]{sec:analytic}

We begin by defining the basic $b$-Sobolev spaces on $C$. Let $D$ be a finite set of points on $C$. For each $p \in D$, let $z_p$ be a holomorphic coordinate centered at $p \in D$ and introduce polar coordinates $z_p = r e^{i \theta}$.
We define the space of $b$-vector fields $\mathcal{V}_b$ so that near 
 each point $p \in D$, $\mathcal{V}_b$  is the span over  smooth functions on the blowup $C_P$ of $C$ at $P$ of the generating $b$-vector fields
$r\partial_r$ and $\partial_\theta$. 
The weighted $b$-Sobolev spaces with $\ell$ derivatives and weight $\delta$ are defined by 
\[
L_\delta^{\ell,2} = \{u = r^\delta v: 
V_1 \cdots V_j v \in L^2\ \text{ for all } j \leq \ell \mbox{ and } V_i \in \mathcal{V}_b\}.
\]
We refer the reader to \cite{MazzeoWeiss} for a highly readable extended discussion of the $b$-calculus in a similar context.\footnote{For the uninitiated, the 
first thing to observe is that the Laplacian $\nabla = \del_x^2 + \del_y^2$ on a punctured disk in $\R^2$ can be written in terms of these $b$-vector fields as $r^{-2}\left((r \del_r)^2 + \del_\theta^2\right)$.
}

This definition adapts to sections of vector bundles over $C$.  
In particular, let $\mathcal D_C$ be the sheaf of differential operators generated by the $b$-vector fields $\mathcal V_b$ and let $\mathscr S$ be a $\mathcal D_C$-module equipped with a norm $|\cdot|\from\mathscr S\to \mathcal C_{\geq0}^0$ taking sections of $\mathscr S$ to positive functions.  Then for $\delta>0$ and $U\subseteq C$, we use $L_\delta^{\ell,p}(U,\mathscr S)$ to denote the $b$-weighted Sobolev completion of $\mathscr S(U)$.
For the space of global sections we write $L_\delta^{k,p}(\mathscr S)=L_\delta^{k,p}(C,\mathscr S)$.

\begin{remark}The weighted Sobolev spaces in \cite{LockhartMcOwen} and used \cite{CFW24} are very similar to the standard $b$-calculus weighted Sobolev spaces defined above. In Lockhart--McOwen, one blows up $p_i \in D$, replacing the neighborhood with an infinite cylindrical end ($t \to \infty$) with coordinate $t$ related to $r$ by $-t=\log r$. Differentiation is done with respect to $t$. Note that 
 $\del_t = \frac{\del r}{\del_t} \del_r = -r \del_r$, so $\del_t$ is a $b$-vector field. The $b$-calculus framework is slightly more general that Lockhart--McOwen's set up. In Lockhart--McOwen's set up, one fixes $\nabla_0$ such that on the infinite cylinder, $\nabla_0$ is exactly
 \[ \nabla_0 = \de + \sqrt{-1}\begin{pmatrix} 1-\alpha_i & 0 \\ 0 & \alpha_i \end{pmatrix}d \theta.\]
Given an $R$-harmonic bundle $(\delbar_E, \varphi, h_R)$, one might like to take $\nabla_0$ to be the associated Chern connection in an $h_R$-unitary frame, however this Chern connection differs from $\nabla_0$ by non-zero subleading terms. The $b$-calculus permits such subleading terms.
\end{remark}

\begin{remark}
    In the construction of the moduli space discussed \Cref{sec:hitchinhk}, one looks at the indicial roots of $\nabla_0$  (or the Chern connection associated to an $R$-harmonic bundle), and picks $\delta$ to lie in the interval adjacent to $0$ (e.g. see \cite[Remark 3.2]{CFW24}). Here, that first positive indicial root is at $1-2 \alpha_i$. Consequently, the spaces one uses to construct the moduli space depends on the parabolic weight. Consequently, if $\alpha_i = \frac{1}{2} - R \beta_i$, the first positive indicial root converges to $0$ as $R \to 0$.
\end{remark}

\section{\texorpdfstring{Interlude: Moment Map at $\C^\times$-Fixed Points}{Interlude: Moment Map at C*-Fixed Points}}
\label{sec: Torelli Numbers}

The hyperk\"ahler spaces $\mathcal X(\vec \beta) = X\fourslash_{(0, \vec \beta)} G$ and $\mathcal M_R(\vec \alpha)$ inherit $U(1)$ actions, and the moment maps of these actions with respect to the relevant K\"ahler form ($\omega_{J_1}^{HP}$ or $\omega_{J_1, R}$, respectively) are Morse-Bott functions. We compute the values of the moment map at the $U(1)$-fixed points.  In \Cref{thm: Morse functions agree at fixed points}
we prove that for generic parameters $\vec\beta\in(0,\infty)^n$ and parabolic weights $\alpha_i(R) := \frac{1}{2} - R \beta_i$ for $R\in(0,R_{\max})$, we have at each $U(1)$-fixed point $(\bx,\by)$
\[ M_R(\mathcal{T}_R(\bx, \by)) = 2 \pi M(\bx, \by)\]
(recall $R_{\max}$ solves $W_{[n]}(\alpha(R_{\max}))=(n-2)/2$).  Then, we specialize to the case $n=4$.  In this case, both spaces are of real dimension four.
We describe the topology of these two spaces, and then use action-angle coordinates and the values of the moment map to integrate the respective symplectic forms over degree 2 homology generators on both spaces and compare.

\subsection{\texorpdfstring{The $\U(1)$-Action and Moment Map}{The U(1)-Action and Moment Map}}
\label{subsec: The U(1)-Action and Moment Map}
The group $\C^\times$ acts on $\mathcal X(\vec\beta)$ by $\lambda\cdot(\bx,\by)=(\bx,\lambda\by)$ and on $\mathcal M^{\mathrm{Higgs}}(\vec\alpha)$ by $\lambda\cdot(\mathcal E,\varphi)=(\mathcal E,\lambda\varphi)$.  It is easy to verify that these actions preserves $\vec\beta$- and $\vec\alpha$-stability respectively, and $\mathcal T\from\mathcal X(\vec\beta)\to\mathcal M^{\mathrm{Higgs}}(\vec\alpha)$ is $\C^\times$-equivariant with respect to this action.  The restrictions of these actions to $U(1)\subset\C^\times$ on the hyperkähler manifolds admit moment maps given by
    \begin{equation}\label{eqn: Morse-Bott functions}
        M_\text{HP}(\bx,\by)=\frac i2\sum|_iy_i|^2,\qquad\qquad M_R(\mathcal E,\varphi,h_R)=\frac i2\int R|\varphi|_{h_R}^2,
    \end{equation}
where the formulas require $(\bx,\by)$ be unitary and $h_R$ be the $R$-harmonic metric for $(\mathcal E,\varphi)$.  
 The moment maps $M_\text{HP}, M_R$ of the $\U(1)$-action with respect to the K\"ahler forms $\omega^{HP}_{J_1}$ and $\omega_{J_1, R}$ are Morse-Bott functions. 
 
 In this section, we prove in \Cref{thm: Morse functions agree at fixed points} that at corresponding $U(1)$-fixed points, the moment maps agree, up to a factor of $2 \pi$. 
 
\subsubsection{Moment Map on Hyperpolygon Space}
We first consider the $n$-sided hyperpolygon space.  Let $\mathcal S=\{I\subset\{1,\dots, n\}\mid |I|\geq 2 \text{ and }I \text{ is short}\}$.  
The fixed point set decomposes as
\begin{equation}\label{eqn: fixed point decomposition}
    \mathcal X(\vec\beta)^{\C^\times}=X_\emptyset\sqcup\bigsqcup_{I\in\mathcal S}X_I
\end{equation}
where $X_\emptyset$ consists of hyperpolygons of the form $(\bx,\mathbf 0)$, and $X_I$ consists of hyperpolygons of the form $(\bx,\by)$ with $\by\ne0$ such that $I$ and $I^c$ are both straight and $y_i=0$ for all $i\notin I$ (see \cite{Kon02,HP03}).  Furthermore, $X_{I}$ is diffeomorphic to $\mathbb{CP}^{|I|-2}$\cite{Kon02}.  Depending on the value of $\beta$, $X_\emptyset$ can be empty.  However, setting
\begin{equation}
    \mathcal{I} =\{I \in \mathcal{S} \cup \{ \emptyset\} : X_I \ne \emptyset \}
\end{equation}
we have that $|\mathcal I|$ is the same for all generic $\vec\beta$.\footnote{
    $X_\emptyset$ is empty if and only if $W_{\{i\}}(\vec\beta)>0$ for some $i\in\{1,\dots,n\}$, in which case we gain the short subset $I=\{1,\dots,n\}\sm\{i\}\in\mathcal S$.
}

\begin{proposition}\label[proposition]{prop: HP Morse function} 
    Consider a unitary hyperpolygon $(\bx,\by)\in\mathcal X(\vec\beta)^{\U(1)}$.
    \begin{itemize}
            \item If $(\bx,\by)\in X_\emptyset$, then $M_\text{HP}(\bx,\by)=0$.
            \item If $(\bx,\by)\in X_I$, then
        $$M_\text{HP}(\bx,\by)=\frac1{2i}W_I(\vec\beta)=\frac i2\left(\sum_{i\in I^c}\beta_i-\sum_{i\in I}\beta_i\right).$$
    \end{itemize}
\end{proposition}
\begin{proof} For the moment map equation, we assume $(\bx,\by)$ is a unitary representative of the $\C^\times$-fixed point, i.e. it solves the real moment map equation $\mu_\R(\bx,\by)=\vec\beta$. 
\begin{itemize}\item 
If $(\bx,\by)\in X_\emptyset$, then $\by=0$ so $M_\text{HP}(\bx,\by)=0$.  

\item 
Assume $(\bx,\by)\in X_I$ for some short subset with $|I|\geq2$, so $I$ and $I^c$ are both straight and $y_i=0$ for $i\in I^c$.  

Put $v_i=(x_ix_i^\dagger)^\perp$ and $w_i=(y_i^\dagger y_i)^\perp$, where $A^\perp$ denotes the trace free part of a matrix $A$.  The equations \eqref{eqn: v norm} and \eqref{eqn: w norm} imply
    $$|v_i|=\frac{|x_i|^2}{\sqrt2}\quad\text{and}\quad|w_i|=\frac{|y_i|^2}{\sqrt2}.$$
The condition $y_ix_i=0$ implies $v_i$ and $w_i$ point in opposite directions for $i\in I$.  Since $I$ is straight,
    $$\left|\sum_{i\in I}\left(v_i-w_i\right)\right|=\sum_{i\in I}|v_i|+|w_i|.$$
On the other hand, since $y_i=0$ and $|x_i|^2=2\beta_i$ for $i\in I^c$ we have
    $$\left|\sum_{i\in I^c}v_i\right|=\sum_{i\in I^c}\frac1{\sqrt2}|x_i|^2=\sum_{i\in I^c}\sqrt2\beta_i.$$
We can write the moment map condition as
    $$0=\mu_{\SU(2)}(\bx,\by)=\sum_{i=1}^n(x_ix_i^\dagger)^\perp-(y_i^\dagger y_i)^\perp=\sum_{i\in I}(v_i-w_i)+\sum_{i\in I^c}v_i.$$
Since $I$ and $I^c$ are straight and $v_i$, $w_i$ point in opposite directions, this yields
    $$0=\left|\sum_{i\in I}v_i-\sum_{i\in I}w_i\right|-\left|\sum_{i\in I^c}v_i\right|=\sum_{i\in I}\left|v_i\right|+\sum_{i\in I}\left|w_i\right|-\sqrt2\sum_{i\in I^c}\beta_i,$$
    i.e. 
    \[ \sum_{i\in I}\left|w_i\right| = \sqrt2\sum_{i\in I^c}\beta_i- \sum_{i\in I}\left|v_i\right|.\]
    However, we can rewrite $|v_i|$ using $2\beta_i=|x_i|^2-|y_i|^2=\sqrt2(|v_i|-|w_i|)$ to get
    $$\sum_{i\in I}|w_i|
    =\sqrt2\sum_{i\in I^c}\beta_i-\sum_{i\in I}(\sqrt2\beta_i+|w_i|),$$
and solve for $|w_i|$. Consequently, 
\begin{equation}
    M_\text{HP}(\bx,\by)=\frac i2\sum_{i\in I}|y_i|^2=\frac i{\sqrt2}\sum_{i\in I}|w_i|=\frac i2\left(\sum_{i\in I^c}\beta_i-\sum_{i\in I}\beta_i\right).
\end{equation}
\end{itemize}
\end{proof}

\subsubsection{Moment Maps on the Family of Hitchin Moduli Spaces}
If an $R$-harmonic bundle $(\mathcal{E}, \mathcal{\varphi}, h_R)$ is fixed by the $U(1)$ action, then there is decomposition
\begin{align}
\mathcal{E} &= \mathcal{E}_1 \oplus \cdots \oplus \mathcal{E}_{\ell} \nonumber \\
\varphi&=\mathcal{E}_i \to \mathcal{E}_{i-1} \otimes K(D) 
\end{align}
of $\mathcal{E}$ as a direct sum of parabolic bundles such that $\mathcal{E}_i$ are mutually orthogonal respect to $h_R$. In the rank $2$ case, there are only two options: 
\begin{itemize}\item[($\ell=1$)]$\varphi \equiv 0$ and $\mathcal{E}$ is a stable parabolic bundle, or \item[($\ell=2$)] $\mathcal{E} = \mathcal{L}_1 \oplus \mathcal{L}_2$ \[ \varphi =
\begin{pmatrix} 0 & \varphi_{12}\\
0 & 0 \end{pmatrix}\]
where $\varphi_{12}: \mathcal{L}_2 \to \mathcal{L}_1 \otimes K(D)$.
\end{itemize}

\begin{proposition}\label{prop:morsehitchin}
Consider an $R$-harmonic bundle $(\delbar_E, \varphi, h_R) \in \mathcal{M}_{R}(\vec \alpha)$, and recall the Morse-Bott function $M_R$ in \eqref{eqn: Morse-Bott functions}.
\begin{itemize}
    \item If $\varphi \equiv 0$, then $M_R(\mathcal{E}, \varphi, h_R)=0$.
    \item If $\varphi \not\equiv 0$, then setting $I=\{i : F_i \in \mathcal{L}_1 \} \subset \{1, 2, \cdots, n\}$, we have
    \begin{align*}
        M_R(\delbar_E, \varphi, h_R)
            &= - \pi i R^{-1} \left( \deg \mathcal{L}_1 + \sum_{i \in I} (\frac{1}{2}-\alpha_i) + \sum_{i \in I^c} (\alpha_i- \frac{1}{2}) \right).
    \end{align*}
\end{itemize}
\end{proposition}

\begin{proof}\hfill
     \begin{itemize}
        \item[($\ell=1$)] It is clear that $\mu$ vanishes when on stable parabolic bundles since $\varphi \equiv 0$. 
        \item[($\ell=2$)] 
The key observation is that because $\varphi$ is strictly upper triangular  triangular and $h_R$ respects the splitting, 
\[ \varphi^{\dagger_{h_R}} = \begin{pmatrix} 0 & 0 \\ (\varphi^{\dagger_{h_R}})_{21} & 0 \end{pmatrix}, \]
hence 
     \[
     [ \varphi, \varphi^{\dagger_{h_R}}] = \begin{pmatrix} \varphi_{12} \wedge (\varphi^{\dagger_{h_R}})_{21}	 & 0 \\ 
 0 & (\varphi^{\dagger_{h_R}})_{21}	 \wedge \varphi_{12} 
 \end{pmatrix}
     \]
     and Hitchin's equations break into two scalar-valued-$2$-form equations 
     \begin{align*}
\;\;F^\perp_{\nabla(\delbar_E, h_R)}|_{\mathcal{L}_1} + R^2 \varphi_{12} \wedge (\varphi^{\dagger_{h_R}})_{21} &=0,\\   \;\; F^\perp_{\nabla(\delbar_E, h_R)}|_{\mathcal{L}_2} + R^2 (\varphi^{\dagger_{h_R}})_{21}	 \wedge \varphi_{12} &=0.  
    \end{align*}
     Using
     \[
      \varphi \wedge \varphi^{\dagger_{h_R}} = \begin{pmatrix} \varphi_{12} \wedge (\varphi^{\dagger_{h_R}})_{21}	 & 0 \\ 
 0 & 0
 \end{pmatrix}
     \]
     we thus obtain
\begin{align*}
      M_R(\mathcal E, \varphi, h_R) 
      &= \frac{i}{2} \int_C  R |\varphi|^2_{h_R} \\
      &= \frac{1}{2} \int_C \mathrm{Tr}( \varphi \wedge \varphi^{\dagger_{h_R}})\\
      &= \frac{1}{2} \int_C R \varphi_{12} \wedge (\varphi^{\dagger_{h_R}})_{21}\\
      &= -\frac{1}{2} \int_C R^{-1} F^\perp_{\nabla(\delbar_E, h)}|_{\mathcal{L}_1} \\
      &= - \pi i  R^{-1} \left( \mathrm{pdeg} \mathcal{L}_1(\vec \alpha) -\mathrm{pdeg} \det \mathcal{E} (\vec \alpha)^{1/2} \right)\\
    &= - \pi i R^{-1} \left( \deg \mathcal{L}_1 + \sum_{i \in I} \left(\frac{1}{2}-\alpha_i\right) + \sum_{i \in I^c} \left(\alpha_i- \frac{1}{2}\right) \right).
    \end{align*}
    \end{itemize}

    \end{proof}
    
\subsubsection{Comparison of Moment Maps at Fixed Points}

\begin{theorem}\label[theorem]{thm: Morse functions agree at fixed points}  
    Let $\vec\beta \in (0, \infty)^n$ be generic, let $\alpha_i(R)= \frac{1}{2} - R \beta_i$ for $R>0$, and assume $R_\text{max}$ solves $W_{[n]}(\vec\alpha(R_\text{max}))=(n-2)/2$ as in \Cref{thm: map from hyperpolygon to Higgs}.  For $R\in(0,R_\text{max})$, let $M_\text{HP}$ and $M_R=M_{R,\vec\alpha(R)}$ be the Morse-Bott functions on $\mathcal{X}(\beta)$ and $\mathcal{M}_R(\vec \alpha(R))$ described in \eqref{eqn: Morse-Bott functions}.
    For each $\U(1)$-fixed point $(\bx,\by)$ and corresponding $\U(1)$-fixed harmonic bundle $\mathcal T_R(\bx,\by)=(\mathcal E,\varphi,h_R)$, and for all $0<R<R_\text{max}$, we have \begin{equation} M_R(\mathcal E,\varphi,h_R)=2\pi M_\text{HP}(\bx,\by).\end{equation}
\end{theorem}
\begin{proof} \hfill
\begin{itemize}
\item[($\ell=1$)]
    Suppose $(\bx,\by)\in X_\emptyset$, hence $M_\text{HP}(\bx,\by)=0$.  Then $\by=0$, so the corresponding Higgs bundle has $\varphi=0$.  Therefore, $M_R(\mathcal E,\varphi,h_R)=0$.
\item[($\ell=2$)]
Suppose $(\bx,\by)\in X_I$ for $I \neq \emptyset$. Pick a basis in which $x_i\in\langle e_1\rangle$ for $i\in I$ and $x_i\in\langle e_2\rangle$ for $i\in I^c$.  Since $y_i=0$ for $i\in I^c$, the Higgs bundle has the form
        $$\varphi=\begin{pmatrix}0&\varphi_{12}\\0&0\end{pmatrix}.$$
Since $\mathcal{E}= \mathcal{O} \oplus \mathcal{O}$ as holomorphic bundles, then $\mathcal{L}_1 \simeq \mathcal{O}$.
Moreover, the set $I$ appearing in the statement of \Cref{prop:morsehitchin} agrees with the set $I$ appearing in the statement of  \Cref{prop: HP Morse function}.
Then 
     \begin{align*}
        M_R(\mathcal E,\varphi,h_R)    &= -\pi i R^{-1} \left( \deg \mathcal{L}_1 + \sum_{i \in I} (\frac{1}{2}-\alpha_i(R)) + \sum_{i \in I^c} (\alpha_i(R)- \frac{1}{2}) \right)\\
                       &=- \pi i \left(\sum_{i\in I}\beta_i-\sum_{i\in I^c}\beta_i\right)\\\
                       &=  2\pi \cdot \frac{1}{2i} W_{I}(\vec \beta)\\
            &=2\pi M(\bx,\by).
    \end{align*}
    \end{itemize}
\end{proof}

\subsection{\texorpdfstring{Example: $n=4$ (ALG-$D_4$ to ALE-$D_4$ degeneration)}{Example: n=4 (ALG-D4 to ALE-D4 degeneration)}}
We now restrict to $n=4$.

\subsubsection{Chamber Structures and the Biswas Polytope} 
\label{subsec: Chamber Structures and the Biswas Polytope}

Following Nakajima in \cite{Nak94}, we define $\vec \alpha$ to be generic if $\vec \alpha$-semistability  implies $\vec \alpha$-stability.  The complement is a union of codimension one planes in $(0, \frac{1}{2})^4$ which we call \emph{walls}. 

Biswas \cite{Bis02} provides a criterion for the existence of stable parabolic vector bundles (with no Higgs field) with parabolic weights $\vec\alpha$, which we adapt to the case $n=4$ and parabolic degree 4. The function $
W_I$ defined in \eqref{eq:W} plays an important role.

\begin{theorem}[\cite{Bis02}]
    Let $D=\{p_1,\dots,p_4\}\subset\C\P^1$, $\vec\alpha\in(0,\frac12)^4$, and let $E\to\C\P^1$ be a trivial bundle of rank 2.  There exists holomorphic structure $\delbar_E$ and flags $\mathcal F=\{F_i\}_{i=1}^4$ such that $(E,\delbar_E,\mathcal F,\vec\alpha)$ is stable if and only of $0<W_I(\vec\alpha)<1$ for all subsets $I\subset\{1,2,3,4\}$ with $|I|=3$.
\end{theorem}

\begin{definition}
    The subset of $(0,\frac12)^4$ such that $0<W_I(\vec\alpha)<1$ for all $|I|=3$ is called the \textit{Biswas polytope} $\mathscr B$.  \end{definition}

These eight hyperplanes $W_I(\vec \alpha)=0, 1$ for $|I|=3$ are chamber walls.  The Biswas polytope is further divided into 16 chambers by four hyperplanes: Three of these walls are given by $W_{I}(\vec \alpha)=0$ for $|I|=2$ (note that $I$ and $I^c$ define the same walls) and the fourth is $W_{\{1, 2, 3, 4\}}(\vec \alpha)=1$. At these walls, $\mathcal M^\text{Higgs}(\vec\alpha)$ is singular.

Similarly, $\mathcal X(\vec\beta)$ is singular if and only if $W_I(\vec\beta)=0$ for some $I$ with $|I|=2$ or $|I|=3$; else we say $\vec\beta$ is generic.  For any $R>0$, the assignment $\alpha_i\leftrightarrow\frac12-R\beta_i$ maps the walls $W_I(\vec\alpha)=1$ to $W_I(\vec\beta)=0$ for all $|I|=3$, and the walls $W_I(\vec\alpha)=0$ to $W_I(\vec\beta)=0$ for $|I|=2$.  The additional walls for $\vec\alpha$, namely $W_I(\vec \alpha)=0$ for $|I|=3$ and $W_{\{1, 2, 3, 4\}}(\vec \alpha)=1$, have no analogue for the hyperpolygon parameters, so genericity of $\vec\alpha$ implies that of $\vec\beta$, but the converse does not hold.  

We continue to assume that $\vec\alpha$ and $\vec \beta$ are in corresponding chambers---in particular $W_{\{1, 2, 3, 4\}}(\vec \alpha)>2$ so that $\mathcal{T}$ preserves stability (see \Cref{rmk: chamber condition is necessary}).  To avoid dealing with multiple cases, for this section we will also assume that $\vec \alpha \in \mathscr B$.  The story outside the Biswas polytope is entirely analogous.

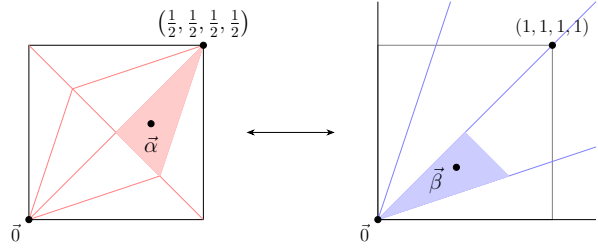
\begin{figure}
\begin{center}
\resizebox{8cm}{!}{
\begin{tikzpicture}
    \draw (0,0) rectangle +(4cm,4cm);
    \draw[red!50] (0,4) -- (4,0);
    \draw[red!50] (0,0) -- (4,4);
    \draw[red!50] (0,0) -- (3,1);
    \draw[red!50] (3,1) -- (4,4);
    \draw[red!50] (0,0) -- (1,3);
    \draw[red!50] (1,3) -- (4,4);
    \filldraw[red!20] (2,2) -- (3,1) -- (4,4) -- cycle;
    \filldraw (2.8,2.2) circle (2pt)
        node[below=1.5mm]{\Large$\vec\alpha$};
    \filldraw (0,0) circle (2pt);
    \filldraw (4,4) circle (2pt);
    \node[xshift = -3mm, yshift = -3mm] at (0,0) {\large$\vec0$};    
    \node[xshift = 0mm, yshift = 5mm] at (4,4) {\Large$\left(\frac12,\frac12,\frac12,\frac12\right)$};

    \draw[<->, >=Stealth] (5,2) -- (7,2);

    \draw[black!50] (8,0) rectangle +(4cm,4cm);
    \draw (8,0) -- (13,0);
    \draw (8,0) -- (8,5);
    \draw[blue!50] (8,0) -- (13,5);
    \draw[blue!50] (8,0) -- (13,5/3);
    \draw[blue!50] (8,0) -- (8+5/3,5);
    \filldraw[blue!20] (8,0) -- (10,2) -- (11,1) -- cycle;
    \filldraw (9.8,1.2) circle (2pt)
    node[below=3mm,left=1.8mm]{\Large$\vec\beta$};
    \filldraw (8,0) circle (2pt);
    \filldraw (12,4) circle (2pt);
    \node[xshift = -3mm, yshift = -3mm] at (8,0) {\large$\vec0$};    
    \node[xshift = 0mm, yshift = 4mm] at (12,4) {\large$(1,1,1,1)$};
\end{tikzpicture}
}
\end{center}
\caption{The chamber structures for $\vec\alpha$ (left) and $\vec\beta$ (right).}
\label{fig: chamber structures}
\end{figure}

\subsubsection{\texorpdfstring{Stratification by $\C^\times$-flow}{Stratification by C*-flow}}
\begin{definition}
    Given a hyperpolygon $(\bx,\by)$, the limit $[(\bx_0,\by_0)]=\lim_{\lambda\to0}\lambda\cdot[(\bx,\by)]$ exists and is a $\C^\times$-fixed point.  We say $(\bx,\by)$ \textit{flows down} to $(\bx_0,\by_0)$.  On the other hand, if $(\bx_0,\by_0)$ is a $\C^\times$-fixed point, the set of points $(\bx,\by)$ which flow down to $(\bx_0,\by_0)$ is called the \textit{upward flow} of $(\bx_0,\by_0)$.
\end{definition}
The moment map $M$ is a Morse-Bott function, and it gives rise to the Morse stratification
    $$\mathcal X(\vec\beta)=
    \bigsqcup_{I\in\mathcal I} Y_I$$
consisting of the upward flows of the various fixed point components.
However, $\mathcal M^\text{Higgs}(\vec\alpha)$ has one additional $\C^\times$-fixed point consisting of the unique Higgs bundle (up to equivalence) on $K_C^{1/2}\oplus K_C^{-1/2}=\mathcal O(-1)\oplus\mathcal O(1)$ with flags $F_i=\left\langle\left(\begin{smallmatrix}1\\0\end{smallmatrix}\right)\right\rangle$ for all $i$ and Higgs field
    $$\varphi=\begin{pmatrix}0&\varphi_{12}\\0&0\end{pmatrix},$$
    where $\varphi_{12}$ is a holomorphic section of $\mathcal{O}(-2) \otimes K_C(D)\cong\mathcal O$.
The Morse stratification (also called the Birula--Białynicki stratification) is
    $$\mathcal M^{\mathrm{Higgs}}(\vec\alpha)=
    \bigsqcup_{I\in\mathcal I}\mathcal Y_I\sqcup\mathcal Y_{\{1,2,3,4\}}.$$
    The strata $\mathcal Y_I$ for $|I|=2,4$ are embedded copies of $\C$ called the \textit{Hitchin sections}, and $\mathcal Y_{\{1,2,3,4\}}$ is called the \textit{distinguished Hitchin section}.  Outside $\mathcal Y_{\{1,2,3,4\}}$, the Higgs bundles all have underlying holomorphic structure $\mathcal O_{\C\P^1}^2$, so \Cref{thm: image of T} states $\mathcal T$ is a bijection onto $
    \bigsqcup_{I\in\mathcal I}\mathcal Y_I= \mathcal{M}^{\mathrm{Higgs}}(\vec \alpha) - \mathcal{Y}_{\{1, 2, 3, 4\}}$, which is open and dense in $\mathcal M_R(\vec\alpha)$. 

\subsubsection{The Nilpotent Cone}\label{subsec: The nilpotent cone} 
The moduli space $\mathcal M=\mathcal M^{\mathrm{Higgs}}(\vec\alpha)$ is equipped with the so-called Hitchin fibration \begin{equation}\mathcal H\from\mathcal M^{\mathrm{Higgs}}(\vec \alpha) \to\mathcal B\end{equation} which assigns to a Higgs bundle $(\mathcal E,\varphi)$ the spectral data of $\varphi$.  Since $\rk\varphi=2$ and $\tr\varphi=0$, the spectral data of $(\mathcal E, \varphi)$ is determined by $\det\varphi \in H^0(\C\P^1,K_{\C\P^1}^{\otimes2}(D))\cong\C$, hence $\mathcal B\cong\C$.  The fibers of $\mathcal H$ are holomorphic Lagrangian with respect to $\Omega_{J_1}$, so $\mathcal H$ gives $\mathcal M_R(\vec\alpha)$ the structure of an algebraic completely integrable system. 

Similarly, because $\mathcal{T}: \mathcal{X}(\vec \beta)=X\twoslash_{(0, \vec \beta)} G_{\C} \to \mathcal{M}^{\mathrm{Higgs}}(\vec \alpha)$ is a map of holomorphic symplectic manifolds,
\begin{equation}\mathcal H\circ\mathcal T: \mathcal{X}(\vec \beta) \to \mathcal{B}\end{equation}
is a holomorphic Lagrangian fibration on $\mathcal X(\vec\beta)$, but with noncompact fibers.  We emphasize that this depends on the divisor $D$ via the cross ratio of the four points. (Consequently, hyperpolygon spaces actually admits a whole family of holomorphic Lagrangian fibrations.)

The \textit{nilpotent cone} of $\mathcal M$ is the fiber $\mathcal H^{-1}(0)\subset\mathcal M$ consisting of Higgs bundles with nilpotent Higgs fields. 
We shall describe the components of the nilpotent cone of $\mathcal M^{\mathrm{Higgs}}(\vec \alpha)$.  
We are primarily interested in comparing the nilpotent cone in $\mathcal{M}^{\mathrm{Higgs}}(\vec \alpha)$ with its preimage under $\mathcal T$, which we will call the nilpotent cone of $\mathcal X(\vec \beta)$.
We refer the reader to \cite{Men22} for a full treatment of the former. 

\begin{lemma}\cite{Men22}
    Suppose $(\mathcal E,\varphi)$ is in the nilpotent cone of $\mathcal M$ and $\mathcal E=\mathcal O\oplus\mathcal O$.  Then $\ker\varphi$ is a holomorphic subbundle of $\mathcal E$, and one of the following holds:
    \begin{itemize}
        \item $\varphi=0$,
        \item $\ker\varphi\cong\mathcal O$,
        \item $\ker\varphi\cong\mathcal O(-1)$.
    \end{itemize}
\end{lemma}
\begin{proof}
    Suppose $\det\varphi=0$.  Since $\varphi$ is meromorphic, $\varphi=\sum\limits_i\frac{\varphi_i}{z-p_i}\de z$ for some constant matrices $\varphi_i\in\mathfrak{sl}(2,\C)$ with $\sum\limits_i\varphi_i=0$.
    Combining terms, we can write
        $$\varphi=\begin{pmatrix}a&b\\c&-a\end{pmatrix}\frac{\de z}{\prod z-p_i}$$
    for polynomials $a,b,c$ of degree at most 2 (the $z^3$ terms cancel because $\sum\varphi_i=0$), and $0=\prod|z-p_i|^2\det\varphi=(-a^2-bc)$.  If $\varphi=0$ we are done.  Suppose any of $a,b,c$ are identically 0.  Then $a^2+bc=0$ implies either $a,c=0$ or $a,b=0$, so $\ker\varphi=\langle e_1\rangle$ or $\ker\varphi=\langle e_2\rangle$.  In both cases $\ker\varphi\cong\mathcal O$.
    
    Now assume $a,b,c\ne0$.  Then $\left(\begin{smallmatrix}b\\-a\end{smallmatrix}\right)$ and $\left(\begin{smallmatrix}a\\c\end{smallmatrix}\right)$ generically generate $\ker\varphi$.  If $a$ is constant, so are $b$ and $c$, and again $\ker\varphi=\mathcal O$.  Otherwise, at least one of $b,c$ have a common root with $a$; say $b$ does.  Cancelling the common linear factor, the section $s(z)=\left(\begin{smallmatrix}b\\-a\end{smallmatrix}\right)$ is either constant or given by two (at most) linear polynomials, so either $\deg\ker\varphi=0$ or $\deg\ker\varphi=-1$.
\end{proof}

The nilpotent cone in $\mathcal M$ can be decomposed into five components whose closures generate $H_2(\mathcal M,\Z)$ based on the three possibilities for $\ker\varphi$ in the above lemma.  The component with $\varphi=0$ is the moduli space $\mathcal N(\vec\alpha)$ of parabolic vector bundles, and it is called the \textit{central sphere} of $\mathcal M$ because it intersects all the other components, which are called \textit{exterior spheres}.  The component of the nilpotent cone consisting of Higgs bundles with $\deg\ker\varphi=-1$ is called the \textit{distinguished exterior sphere}.  The remaining three components of the nilpotent cone consist of Higgs bundles with $\deg\ker\varphi=0$.
\begin{lemma}\label[lemma]{lemma: nilpotent cone components}
    Suppose $\vec\alpha$ is in the Biswas polytope, $(\mathcal E,\varphi)$ is an $\vec\alpha$-stable parabolic Higgs bundle in the nilpotent cone, and $I=I_\varphi:=\{i\in\{1,2,3,4\}\mid\varphi_i\ne0\}$.  Then:
    \begin{enumerate}[label=(\roman*)]
    \item If $\varphi=0$, $I=\emptyset$.
    \item If $\varphi\ne0$ and $\ker\varphi \cong \mathcal O$, then $|I|=2$.
    \item If $\varphi\ne0$ and $\ker\varphi\cong\mathcal O(-1)$, then $|I|=4$.
    \end{enumerate}
\end{lemma}
\begin{proof}
    (i) is obvious.  For (ii), suppose $\varphi\ne0$ and $\deg\ker\varphi=0$.  Note that the condition $\sum\varphi_i=0$ makes $|I|=1$ impossible.  The subbundle $L:=\ker\varphi\subset\mathcal E$ has induced parabolic weights
        $$\alpha_i(L)=
        \begin{cases}
            1-\alpha_i
                &\text{if }L_{p_i}=F_i\\
            \alpha_i
                &\text{if }L_{p_i}\cap F_i=\emptyset.
        \end{cases}$$
    Since $\deg L=0$,
        $$\text{slope}\,L=\text{pardeg}\,L=\sum_{i\in I}(1-\alpha_i)+\sum_{i\notin I}\alpha_i=|I|-W_I(\vec\alpha).$$
    Since $\text{slope}\,E=2$, stability implies $W_I(\vec\alpha)>|I|-2$.  But we assumed $0<W_I(\vec\alpha)<1$ for all $I$ with $|I|=3$, and $W_{\{1,2,3,4\}}(\vec\alpha)>2$ is impossible, so we must have $|I|=2$.

    If $\deg\ker\varphi=-1$, at least two of the residues $\varphi_i$ are nonproportional, say $\varphi_1$ and $\varphi_2$.  Since each $\varphi_i$ is nilpotent rank 1, $\varphi_1+\varphi_2$ has full rank,\footnote{Say $a\in\ker\varphi_1$ and $b\in\ker\varphi_2$.  Since $\varphi_1,\varphi_2$ are nonproportional, so are $a,b$.  Now $(\varphi_1+\varphi_2)a=c_1b$ and $(\varphi_1+\varphi_2)b=c_2a$ for some $c_1,c_2\ne0$.  This shows $\rk(\varphi_1+\varphi_2)=2$.} so $\varphi_3+\varphi_4=-\varphi_1-\varphi_2$ does too.  Therefore, $|I|=4$.
\end{proof}

This lemma allows us to label the components of the nilpotent cone according to the subsets $I_\varphi$: the central sphere is labelled $\mathcal M_\emptyset$, the distinguished exterior sphere is labelled $\mathcal M_{\{1,2,3,4\}}$, and the remaining three exterior spheres are labelled $\mathcal M_I$ for subsets $I$ with $|I|=2$ and $W_I(\vec\alpha)>0$ (see \Cref{fig: nilpotent cone}).  For Higgs bundles on these spheres only two residues of $\varphi$ are nonzero, so the condition $\sum\varphi_i=0$ implies they are proportional.  Note that the correspondence $\beta_i=1-2\alpha_i$ maps the region $W_I(\vec\alpha)>0$ to $W_I(\vec\beta)<0$.  Therefore the three exterior spheres are labelled by the short subsets (the term `short' will always be with respect to $\vec\beta$), and in fact the preimages $\mathcal X_I:=\mathcal T^{-1}(\mathcal M_I)$ for $|I|=2$ consists precisely of the hyperpolygons for which $I$ is straight.
\begin{remark}
    A similar lemma holds when $\vec\alpha$ is outside the Biswas polytope, but the case $\varphi=0$ becomes impossible.  Since we must assume that $W_{\{1, 2, 3, 4\}}(\vec \alpha)>1$ for $\mathcal{T}$ to preserve stability, we find that the central sphere is parameterized by Higgs bundles with $|I_\varphi|=3$.
\end{remark}
    
\begin{remark}[Hyperpolygon Core] In \cite{HP03}, Harada--Proudfoot discuss the ``hyperpolygon core.'' In this case, the hyperpolygon core is four spheres in $D_4$ configuration. The nilpotent cone contains the hyperpolygon core and an additional copy of $\C^\times$ that depends on the cross-ratio of the points of $D$. 
\end{remark}
\begin{figure}
\begin{center}
\resizebox{5cm}{!}{
\begin{tikzpicture}
        \draw (0,0) circle (2cm);
        \shade[ball color = red!80, opacity = 0.7] (0,0) circle (2cm);
        \node[font=\huge] at (2,-2) {$\mathcal M_\emptyset$};
        \draw (0,3) circle (1cm);
        \shade[ball color = black!60, opacity = 0.5] (0,3) circle (1cm);
        \filldraw[red] (0,4) circle (3pt);
        \node[font=\huge] at (0,5) {$\mathcal M_{\{1,2,3,4\}}$};
        \draw (-3,0) circle (1cm);
        \shade[ball color = black!60, opacity = 0.5] (-3,0) circle (1cm);
        \filldraw[red] (-4,0) circle (3pt);
        \node[font=\huge] at (5,0) {$\mathcal M_{I_3}$};
        \draw (0,-3) circle (1cm);
        \shade[ball color = black!60, opacity = 0.5] (0,-3) circle (1cm);
        \filldraw[red] (0,-4) circle (3pt);
        \node[font=\huge] at (0,-5) {$\mathcal M_{I_2}$};
        \draw (3,0) circle (1cm);
        \shade[ball color = black!60, opacity = 0.5] (3,0) circle (1cm);
        \filldraw[red] (4,0) circle (3pt);
        \node[font=\huge] at (-5,0) {$\mathcal M_{I_1}$};
\end{tikzpicture}
}
\end{center}
\caption{The nilpotent cone of $\mathcal M(\vec\alpha)$.  The three exterior spheres are labelled by the three short subsets $I_1,I_2,I_3\in\mathcal I$.  The $\C^\times$-fixed points are shown in red.}
\label{fig: nilpotent cone}
\end{figure}
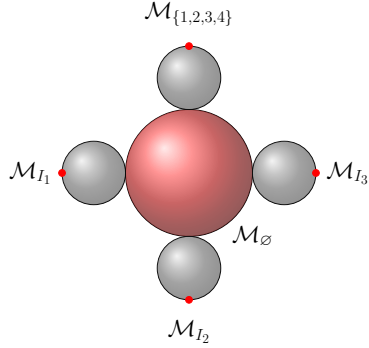

\subsubsection{Computation of Torelli Numbers}
The four components $\mathcal X_I$ of the nilpotent cone ($I=\emptyset$ and the three short $|I|=2$) in hyperpolygon space generate $H_2(\mathcal X(\vec\beta),\Z)$.  Meanwhile, the nilpotent cone of the Hitchin system $\mathcal M_R(\vec\alpha(R))$ consists of five components $\mathcal M_I$ ($I=\emptyset$, three $|I|=2$, and one $|I|=4$) which generate $H_2(\mathcal M_R(\vec\alpha),\Z)$.
\begin{definition}
    Let $(X,g,J_1,J_2,J_3,\omega_{J_1},\omega_{J_2},\omega_{J_3})$ be an ALE (resp.\ ALG) gravitational instanton, and suppose $S_1,\dots, S_R$ are generators for $h_2(X,\Z)$.  The \emph{period matrix} of $X$ is $(\langle[\omega_{J_j}],S_i\rangle)_{i,j}$ where $\langle[\omega_{J_j}],S_i\rangle=\int_{S_i}\omega_{J_j}$.  The entries of the period matrix are called the \emph{Torelli numbers} of the instanton.
\end{definition}
A Torelli-type theorem proved by Kronheimer \cite{Kro89} states that these values on an ALE space determines the hyperkähler structure, hence our interest in comparing these numbers between $\mathcal X(\vec\beta)$ and $\mathcal M_R(\vec\alpha(R))$. 

\begin{theorem}\label[theorem]{thm: Torelli numbers agree}
    The Torelli numbers $\tau_{I,j}^\text{HP}:=\int_{\mathcal X_I}\omega_{J_j}^\text{HP}$ and $\tau_{I,j,R}:=\int_{\mathcal M_I}\omega_{J_j,R}$ satisfy $\tau_{I,j,R}=2\pi\tau_{I,j}^\text{HP}$ for all short subsets $|I|=2$, $j=1,2,3$, and $R\in(0,1]$.
\end{theorem}
\begin{proof}
    Recall that the nilpotent cones on both spaces are holomorphic Lagrangian with respect to $\Omega_{J_1}^\text{HP}=\omega_{J_2}^\text{HP}+i\omega_{J_3}^\text{HP}$ and $\Omega_{J_1}^\text{Higgs}=\omega_{J_2}^\text{Higgs}+i\omega_{J_3}^\text{Higgs}$, so the Torelli numbers associated to the symplectic forms $\omega_{J_2}^\text{HP},\omega_{J_3}^\text{HP}$ and $\omega_{J_2},\omega_{J_3}$ all vanish.  By definition of moment maps, the comparison of $\omega_{J_1}^\text{HP}$ and $\omega_{J_1,R}$ follows immediately from \Cref{thm: Morse functions agree at fixed points} by integrating with respect to action-angle coordinates $\de\theta\wedge\de M$ and $\de\theta\wedge\de M_R$.
\end{proof}
\begin{remark}
    The analogous statement holds for the integrals over the central spheres $\mathcal M_\emptyset$ and $X_\emptyset$, though we omit the proof as it is somewhat technical and we do not need the result (see \cite{Kho05} and \cite{HellerHeller}).  However, a consequence is that $\mathcal T_R^*([\omega_{J_j, R}])=2\pi[\omega^{\text{HP}}_{J_j}]$ for all $j=1,2,3$ and $R \in (0, R_{\max})$.
\end{remark}

This computation motivates \Cref{fig: metric degeneration}.  Ultimately, we will prove the stronger result $\lim_{R \to 0}\mathcal{T}_R^* g_R \to 2 \pi g_{\mathrm{HP}}$.

\tikzset{
  pics/torus/.style n args={3}{
    code = {
      \providecolor{pgffillcolor}{rgb}{1,1,1}
      \begin{scope}[
          yscale=cos(#3),
          outer torus/.style = {draw,line width/.expanded={\the\dimexpr6\pgflinewidth+#2*2},line join=round},
          inner torus/.style = {draw=pgffillcolor,line width={#2*2}}
        ]
        \draw[outer torus] circle(#1);\draw[inner torus] circle(#1);
        \draw[outer torus] (180:#1) arc (180:360:#1);\draw[inner torus,line cap=round] (180:#1) arc (180:360:#1);
      \end{scope}
    }
  }
}
\tikzset{
  pics/opentorus/.style n args={3}{
    code = {
      \providecolor{pgffillcolor}{rgb}{1,1,1}
      \begin{scope}[
          xscale=cos(#3)
        ]
        \begin{axis}[axis x line = none, axis y line = none, width = 24cm, height = 16cm, at = {(-5.2cm,-5.3cm)}]
            \addplot[black, line width = 3pt]{x^2};
        \end{axis}
        \draw[line width = 1pt] (6.8,-1.5) arc (270:90:2);
        \draw[line width = 1pt] (5.7,-1.1) arc (-90:90:1.6);
    \end{scope}
    }
  }
}
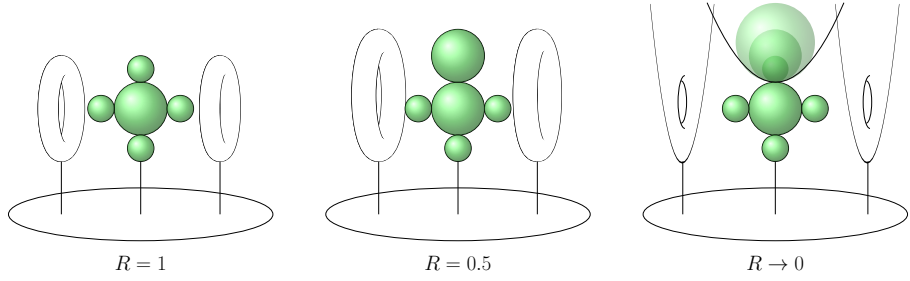
\begin{figure}
\begin{center}
\resizebox{12cm}{!}{
\begin{tikzpicture}
            \draw (0,-8) ellipse (10cm and 2cm)
            node[below=3cm,font=\Huge]{\scalebox{1.8}{$R=1$}};
            \pic[rotate=-90] at (-6,0) {torus={3.25cm}{7.5mm}{72}};
            \draw (-6,-4) -- (-6,-8);
            \pic[rotate=-90] at (6,0) {torus={3.25cm}{7.5mm}{76}};
            \draw (6,-4) -- (6,-8);
            \draw (0,0) circle (2cm);
            \shade[ball color = green!60, opacity = 0.7] (0,0) circle (2cm);
            \draw (0,3) circle (1cm);
            \shade[ball color = green!60, opacity = 0.7] (0,3) circle (1cm);
            \draw (-3,0) circle (1cm);
            \shade[ball color = green!60, opacity = 0.7] (-3,0) circle (1cm);
            \draw (0,-3) circle (1cm);
            \shade[ball color = green!60, opacity = 0.7] (0,-3) circle (1cm);
            \draw (3,0) circle (1cm);
            \shade[ball color = green!60, opacity = 0.7] (3,0) circle (1cm);
            \draw (0,-4) -- (0,-8);

            \draw (24,-8) ellipse (10cm and 2cm)
            node[below=3cm,font=\Huge]{\scalebox{1.8}{$R=0.5$}};
            \pic[rotate=-90] at (18,1) {torus={4.1cm}{9mm}{74}};
            \draw (18,-4) -- (18,-8);
            \draw (24,0) circle (2cm);
            \shade[ball color = green!60, opacity = 0.7] (24,0) circle (2cm);
            \draw (24,4) circle (2cm);
            \shade[ball color = green!60, opacity = 0.7] (24,4) circle (2cm);
            \draw (21,0) circle (1cm);
            \shade[ball color = green!60, opacity = 0.7] (21,0) circle (1cm);
            \draw (24,-3) circle (1cm);
            \shade[ball color = green!60, opacity = 0.7] (24,-3) circle (1cm);
            \draw (27,0) circle (1cm);
            \shade[ball color = green!60, opacity = 0.7] (27,0) circle (1cm);
            \draw (24,-4) -- (24,-8);
                \pic[rotate=-90] at (30,1) {torus={4.1cm}{9mm}{77}};
                \draw (30,-4) -- (30,-8);

            \draw (48,-8) ellipse (10cm and 2cm)
            node[below=3cm,font=\Huge]{\scalebox{1.8}{$R\to0$}};
        \pic[rotate=0] at (39.4,0) {opentorus={3.25cm}{7.5mm}{75}};
            \draw (41,-4) -- (41,-8);
        \pic[rotate=0] at (53.4,0) {opentorus={3.25cm}{7.5mm}{75}};
            \draw (55,-4) -- (55,-8);
            \draw (48,0) circle (2cm);
            \shade[ball color = green!60, opacity = 0.7] (48,0) circle (2cm);
            \shade[ball color = green!60, opacity = 0.5] (48,3) circle (1cm);
            \shade[ball color = green!60, opacity = 0.4] (48,4) circle (2cm);
            \shade[ball color = green!60, opacity = 0.4] (48,5) circle (3cm);
            \begin{axis}[axis x line = none, axis y line = none, width = 14cm, height = 8.8cm, at = {(41.805cm,1.4cm)}]
                \addplot[black, line width = 1.5pt]{x^2};
            \end{axis}
            \draw (45,0) circle (1cm);
            \shade[ball color = green!60, opacity = 0.7] (45,0) circle (1cm);
            \draw (48,-3) circle (1cm);
            \shade[ball color = green!60, opacity = 0.7] (48,-3) circle (1cm);
            \draw (51,0) circle (1cm);
            \shade[ball color = green!60, opacity = 0.7] (51,0) circle (1cm);
            \draw (48,-4) -- (48,-8);
\end{tikzpicture}
}
\end{center}
\caption{ALG Degenerate Limit as $R\to0$}
\label{fig: metric degeneration}
\end{figure}

\section{Local Model}\label{sec: Local Model}

Fix $\beta \in \R^+$. In this section, we construct a family of solutions of Hitchin's equations over the punctured unit disk parameterized by $R$ with parabolic weights $\alpha(R) = \frac{1}{2} - R \beta$  and $1-\alpha(R)$ that are extremely well-behaved as $R \to 0$ (see \Cref{def: Local Harmonic Model Metric}).  Given a deformation of the Higgs bundle, we then construct a family of endomorphisms $\dot\nu_R$ which solve the complex Coulomb gauge equation \eqref{eqn: Coulomb gauge equation}.  Finally, we relate the norms of the Higgs bundle deformations to deformations of a single side of a hyperpolygon.

\subsection{Remarks on Joint Limit}

First, we make a few more detailed remarks on this joint limit $\alpha \to \frac{1}{2}$ and $R \to 0$. 

Let $(\mathcal E, \mathcal F, \varphi)=\mathcal T(\bx,\by)$. A metric $h_R$ solves the $R$-rescaled Hitchin equations if and only if it solves the original Hitchin equation for the bundle $(\mathcal E,\mathcal F,R\varphi)$. We will emphasize the dependence of $h_R$ on the parabolic weights by writing $h_{R, \vec \alpha}$ in this subsection.

We consider the case $n=4$ for simplicity. As a first case, suppose the $\C^\times$-limit $\lim_{R\to0}(\mathcal E, \mathcal F, R\varphi)$ is an $\vec \alpha$-stable parabolic bundle $(\mathcal E,\mathcal F,0)$.  As $R\to0$, the solution $h_{R, \vec \alpha}$ limits to the Kähler-Einstein metric $h_{\natural, \vec \alpha}$ for the parabolic bundle $(\mathcal E,\mathcal F,0)$.  As a second case, suppose $(\mathcal E, \mathcal F, \varphi)$ is in one of the Hitchin sections such that $\mathcal E \simeq \mathcal{O} \oplus \mathcal{O}$. 
 Using the family of gauge transformations $g_R=\text{diag}(R^{1/2},R^{-1/2})$, we find that the $\C^\times$-limit is an $\vec \alpha$-stable Higgs bundle
    $$(\mathcal E,\mathcal F, \varphi_0):=\lim_{R\to0}g_R\cdot(\mathcal E, \mathcal F, \varphi).$$
The solution $h_{R, \vec \alpha, 0}$ associated to $(\mathcal E, \mathcal F, \varphi_0)$ is given in terms of $g_R$ and the uniformization metric $g_{\natural, \vec \alpha}$ on $\C\P^1\sm D$ with conical singularities  depending on $\alpha_i$ at the punctures, and in certain cases it provides a sufficiently good approximation of $h_{R, \vec \alpha}$\footnote{
    For $R\ne0$, the true solution $h_{R,\vec\alpha}$ arises from the uniformization metric $g_{\natural, \vec \alpha, \varphi}$ with conical singularities depending on $\alpha_i$ in a \emph{different conformal class} on $(\CP^1, D)$ with cross-ratio of the four points depending on $\varphi$.
} as $R \to 0$.  
This is the approach used in \cite{CFW24} to calculate the \textit{conformal limit} of a Higgs bundle.

The difficulty arises from the fact that we wish to simultaneously vary the parabolic weights $\vec\alpha(R)$ with $R$.  This does not affect the notion of $\vec\alpha$-stability since we stay within a single chamber when $\vec \alpha$ is sufficiently close to $(\frac{1}{2}, \cdots, \frac{1}{2})$, but it does change the asymptotic behavior of the solutions $h_{R, \vec \alpha(R)}$ near the punctures.  A na\"ive approach might be to decouple the rescaling factor $R$ from the parameterization of the parabolic weights $\vec\alpha(R')$, first taking the limit $R \to 0$ above and then taking the $R' \to 0$ limit. 
We will focus on the second case where $(\mathcal E, \mathcal F, \varphi)$ is in one of the three Hitchin sections. Since the solution of Hitchin's equations $h_{R, \vec \alpha(R'), 0}$ can be written in terms of $g_R$ and $g_{\natural, \vec \alpha(R')}$,  we could then take the limit $R'\to0$ to degenerate the conical singularities to cuspidal singularities as described in \cite{KW18}, extending \cite{Jud98}.  However, it is precisely the relative rates at which $R\to0$ and $\alpha_i\to\frac12$ that encodes the data of the hyperpolygon weights $\beta_i$ (see \Cref{fig: relative rates of R and R'}),
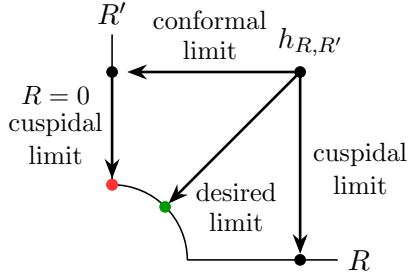
\begin{figure}
    \begin{center}
    \begin{tikzpicture}
        \draw[line width = .6pt] (1,0) -- (3,0);
        \node[right=3cm]{$R$};
        \draw[line width = .6pt] (0,1) -- (0,3);
        \node[above=3cm]{$R'$};
        \draw[line width = .6pt] (1,0) arc (0:90:1);
        \draw[line width=1pt, ->, >=Stealth] (2.5,2.5) -- (0.2,2.5);
        \node[font=\small] at (1.3,3.2) {conformal};
        \node[font=\small] at (1.3,2.8) {limit};
        \filldraw[black] (0,2.5) circle (2pt);
        \draw[black, line width=1pt, ->, >=Stealth] (2.5,2.5) -- (2.5,0.08);
        \node[font=\small] at (3.3,1.4) {cuspidal};
        \node[font=\small] at (3.3,1.0) {limit};
        \filldraw[red!80] (0,1) circle (2pt);
        \filldraw[black] (2.5,0) circle (2pt);
        \draw[black, line width=1pt, ->, >=Stealth] (0,2.5) -- (0,1.08);
        \node[font=\small] at (-.75,2.2) {$R=0$};
        \node[font=\small] at (-.75,1.8) {cuspidal};
        \node[font=\small] at (-.75,1.4) {limit};
        \filldraw[red!80] (0,1) circle (2pt);
        \draw[black, line width=1pt, ->, >=Stealth] (2.5,2.5) -- (1.5/2,1.5/2);
        \filldraw[green!60!black] (1.41/2,1.41/2) circle (2pt);
        \node[font=\small] at (1.65,0.9) {desired};
        \node[font=\small] at (1.65,0.5) {limit};
        \filldraw[black] (2.5,2.5) circle (2pt)
        node[right=4pt, above=2pt]{$h_{R,R'}$};
    \end{tikzpicture}
    \caption{This figure illustrates the dependence of the limit on the relative rates of $R\to0$ and $\alpha_i\to\frac12$.}
    \label{fig: relative rates of R and R'}
    \end{center}
\end{figure}
so this limiting metric is not a sufficiently good approximation for our purposes.  The upshot is that while the conformal limit story was very dependent of the stratification of the Higgs bundle moduli space by $\C^\times$-fixed points, our approximation strategy doesn't make use of the $\C^\times$-action. Rather than needing different approximations depending on the limiting $\C^\times$-fixed Higgs bundle, we only have a single case.

\bigskip

Nevertheless, the na\"ive approach described above inspires our approximation strategy away from the punctures.  Near the punctures, we require more sophisticated local models, which we now present.

\subsection{Background: Kim--Wilkin's harmonic metric}
Let $\mathcal E\to \hat B$ be a rank 2 holomorphically trivial bundle on the punctured unit disk $\hat B$ with parabolic structure
\begin{equation}\label{eq:weightmodel}\begin{array}{ccccc}
    0   &\subset    &\langle e_1\rangle &\subset    &\mathcal E_0\\
      &         &\theta             &>          &-\theta
\end{array}\end{equation}
for $\theta \in (0, \frac{1}{2})$ and
  \begin{equation} \label{eq:phimodel} \varphi=\begin{pmatrix}0&\phi\\0&0\end{pmatrix}\frac{\de z}z.\end{equation}
  Note the sum of the parabolic weights $-\theta, \theta$ is zero, and consequently the induced parabolic structure on $\det \mathcal{E}$ is trivial.
Kim and Wilkin \cite{KW18}\footnote{
    We briefly summarize the results in \cite{KW18}. Kim and Wilkin fix $R=1$. For a given stable parabolic Higgs bundle, they consider how the harmonic metric depends on the parabolic weights. Half of their main result is that the harmonic metric depends analytically on the choice of weights, and the local models play a key role in their proof. Notably $\lim_{\theta \to 0} h_\theta = \begin{pmatrix} -(\log r)^{-1} & 0\\ 0 & - \log r \end{pmatrix}$ and the family depends analytically on $\theta$. The other half of their main result is that the harmonic metric depends analytically on the choice of stable Higgs bundle.
} provide a harmonic metric for $(\delbar,\varphi)$ with $\phi=1$:
    $$h_\theta=\begin{pmatrix}\frac{\theta r^\theta}{1-r^{2\theta}}&0\\0&\frac{1-r^{2\theta}}{\theta r^\theta}\end{pmatrix}$$
where $r=|z|$.

We shall adapt this model metric to work for general $\phi$ and to solve the $R$-rescaled Hitchin equations.  Crucially, we discover an extra free parameter $c$ in the space of solutions which we shall require for our gluing procedure later.
 Notice that in the rest of the paper the parabolic weights are in $(0, 1)$, as we use the convention $\alpha_{\det}=\alpha^{(1)}+\alpha^{(2)}=1$.  Here, we match Kim--Wilkin's convention $\theta_{\det}=0$ so the parabolic weights $-\theta, \theta$ are in $(-\frac{1}{2}, \frac{1}{2})$.
Our harmonic metric will only differ by multiplication by the factor
$\sqrt{h_{\det}}$, where $h_{\det}$ is a Kähler-Einstein connection of $(\det E,\delbar)$ with parabolic weight 1 at the puncture.  

\subsection{Local Model and Properties}
\begin{proposition}\label{prop: delbar del log lambda}
    Let $(\mathcal E,\varphi)$ be the parabolic Higgs bundle on $\hat{B}$ described above in \eqref{eq:weightmodel} and \eqref{eq:phimodel}.  Then for any $c>0$, the metric $h_R=\diag(\lambda_R,\lambda_R^{-1})$ with
    \begin{equation}\label{eqn: KW local model general form}
        \lambda_R=\frac{R^{-1}\theta cr^{\theta}}{1-|\phi|^2c^2r^{2\theta}}
    \end{equation}
    solves the $R$-rescaled Hitchin equations at points in $\hat B$ with $r^\theta<(|\phi|c)^{-1}$.
\end{proposition}
\begin{remark}
    The $R$-dependence of $\lambda_{R}$ is standard for Higgs bundles that are strictly upper triangular: we can write $R \varphi =g_R \varphi g_R^{-1}$ for the family of gauge transformations $g_R=\text{diag}(R^{1/2},R^{-1/2})$.  Then $(\mathcal E,\varphi,h_R)$ is $R$-harmonic if and only if $(\mathcal E,R\varphi,h_1)$ is $1$-harmonic, so $h_R = g_R^{-1} h_1 (g_R^{-1})^{\dagger}$. 
\end{remark}

\begin{proof}
    Since $h_R=\text{diag}(\lambda_R,\lambda_R^{-1})$, we calculate
        $$[\varphi^{\dagger_{h_R}},\varphi]=[h_R^{-1}\varphi^\dagger h_R,\varphi]=-\lambda_R^2|\phi|^2\begin{pmatrix}1&0\\0&-1\end{pmatrix}\,\de\bar z\wedge\de z$$
    and 
        $$F_{\nabla(\delbar,h_R)}=\delbar(h_R^{-1}\del h_R)=\delbar\del\log\lambda_R\begin{pmatrix}1&0\\0&-1\end{pmatrix}.$$
    Hence, Hitchin's equation \eqref{eq:Hitchin} reduces to $\,\delbar\del\log\lambda_R=R^2\lambda_R^2|\phi|^2r^{-2}\,\de\bar z\wedge\de z$.  We will now show this via a direct computation, making use of the fact that $\delbar\del\log(r^\theta)=0$:
    \begin{align}
        \delbar\del\log\lambda_R
            &=-\delbar\del\log(1-|\phi|^2c^2r^{2\theta})
                \label{eqn: delbar del log lambda calculation}\\[2pt]
            &=\delbar\left(\frac{|\phi|^2c^2\theta r^{2\theta}}{1-|\phi|^2c^2r^{2\theta}}\,\frac{\de z}z\right)\notag\\[2pt]
            &=\frac{|\phi|^2c^2\theta^2 r^{2\theta}(1-|\phi|^2c^2r^{2\theta})-|\phi|^2c^2\theta r^{2\theta}(-|\phi|^2c^2\theta r^{2\theta})}{(1-|\phi|^2c^2r^{2\theta})^2}\,\frac{\de\bar z}{\bar z}\wedge\frac{\de z}z\notag\\[2pt]
            &=\frac{|\phi|^2c^2\theta^2r^{2\theta}}{(1-|\phi|^2c^2r^{2\theta})^2}r^{-2}\,\de\bar z\wedge\de z\notag\\[2pt]
            &=R^2\lambda_R^2|\phi|^2r^{-2}\,\de\bar z\wedge\de z.\notag\qedhere
    \end{align}

\end{proof}

We now apply this formula to the setting of a Higgs bundle obtained from a single ``branch'' $(x,y)=(x_i,y_i)$ of a hyperpolygon.  The parameter $c$ must be chosen carefully for the upcoming gluing procedure, and in this case we set $c=|x|^{-2}$. We write this generally without assuming $x\in\langle e_1\rangle$, as we assumed above.
\begin{definition}[Local Model Metric]\label{def: Local Harmonic Model Metric}
Fix $x\in\Hom(\C,\C^2)$, $y\in\Hom(\C^2,\C)$, and $R,\beta>0$ satisfying $yx=0$ and $|x|^2-|y|^2=2\beta$.  Let $\mathcal E$ be a trivial $\SL(2,\C)$-bundle over   with parabolic structure
\begin{equation}\label{eq:weightmodelnew}\begin{array}{ccccccc}
    0   &\subset&\langle x\rangle   &\subset    &\mathcal E_0       &&\\
    1   &>      &\frac12+R\beta_i   &>          &\frac12-R\beta_i   &>&0.
\end{array}\end{equation}
Let $\varphi = xy \frac{\de z}{z}$ and let $h_{\det}$ be a fixed adapted Hermitian--Einstein metric on $\det E$.  Set
\begin{equation}\label{eqn: N formula}
N = 2\frac{(x x^\dagger)^\perp}{|x|^2}.
\end{equation}
The \textit{local model metric} for $(x,y)$ is 
\begin{equation} \label{eq:localmetric} h_{\text{loc},R}=\sqrt{h_{\det}}\exp( N \log \lambda_{\text{loc},R})\end{equation} where
\begin{equation}\label{eqn: KW local model}
    \lambda_\text{loc,R}=\frac{2\beta|x|^{-2}r^{2R\beta}}{1-|y|^2|x|^{-2}r^{4R\beta}}=\frac{2\beta r^{2R\beta}}{|x|^2-|y|^2r^{4R\beta}}.
\end{equation}
\end{definition}

\begin{example}\label{ex:xine1}If we assume that $x\in\langle e_1\rangle$, then $N= \begin{pmatrix} 1 & 0 \\ 0 & -1 \end{pmatrix}$, hence \begin{equation}h_{\text{loc},R}=\sqrt{h_{\det}}\begin{pmatrix}\lambda_{\text{loc},R} & 0 \\ 0 & \lambda_{\text{loc},R}^{-1} \end{pmatrix}.\end{equation}
\end{example}

\begin{lemma}\label[lemma]{lemma:localproperties}\hfill
\begin{enumerate}
    \item The hermitian metric $h_{\mathrm{loc}, R}$ is an $R$-harmonic metric for $\delbar_E=\delbar$, $\varphi=xy\,\frac{\de z}z$ and is adapted to the parabolic structure. 
    \item The local model metric $\lambda_{\mathrm{loc}, R}$ is defined on the whole punctured unit disk.
    \item The pointwise limit is $\lim_{R \to 0} \lambda_{\mathrm{loc},R} = 1$. 
\end{enumerate}
\end{lemma}
\begin{proof} Results (a), (c) are immediate from \Cref{prop: delbar del log lambda} with $c=|x|^{-2}$. For (b), note that 
the local model metric is defined on the disk $r^{4R\beta}<\frac{|x|^4}{|xy|^2}=\frac{|x|^2}{|y|^2}=\frac{2\beta+|y|^2}{|y|^2}$.  
\end{proof}

We now establish a couple of useful formulas involving $\lambda_{\mathrm{loc}, R}$.
\begin{lemma}\label[lemma]{lemma: integral of dd log lambda}
    Let $0<\delta<1$ and $B_\delta=B_\delta(0)$.  Then
        $$\lim_{R\to0}\int_{B_\delta}R^{-1}\delbar\del\log\lambda_{\text{loc},R}=2\pi i|y|^2$$
\end{lemma}
\begin{proof}
    Observe that
    \begin{equation}\label{eqn: del log lambda}
        \del\log\lambda_{\text{loc},R}
            =\left(R\beta-\frac{-2R\beta|y|^2r^{4R\beta}}{|x|^2-|y|^2 r^{4R\beta}}\right)\frac{\de z}z
                =R\beta\frac{|x|^2+|y|^2r^{4R\beta}}{|x|^2-|y|^2r^{4R\beta}}\frac{\de z}z.
    \end{equation}
    Applying Stokes' theorem for this punctured setting,
    \begin{align}
        \lim_{R\to0}\int_{B_\delta}R^{-1}\delbar\del\log\lambda_{\text{loc},R}
            &=\lim_{R\to0}\left(\oint_{\partial B_\delta}R^{-1}\del\log\lambda_{\text{loc},R}-\lim_{\delta'\to0}\oint_{\partial B_{\delta'}}R^{-1}\del\log\lambda_{\text{loc},R}\right)
                \notag\\
            &=\lim_{R\to0}\left(2\pi i \beta\frac{|x|^2+|y|^2\delta^{4R\beta}}{|x|^2-|y|^2\delta^{4R\beta}}-2\pi i\lim_{\delta'\to0}\left(\beta\frac{|x|^2+|y|^2(\delta')^{4R\beta}}{|x|^2-|y|^2(\delta')^{4R\beta}}\right) \right)
                \label{eqn: Stokes on a disc}\\
            &=2\pi i\left((\beta+|y|^2)-\beta\right)\\
            &=2\pi i|y|^2.
                \notag\qedhere
    \end{align}
\end{proof}

For the upcoming gluing construction, we use as building blocks the functions $\lambda_{\natural,R}=\log r^{2R\tilde\beta}$ with modified cone angles $\tilde\beta=\frac12(|x|^2+|y|^2)$.  These can be used to produce flat (K\"ahler-Einstein metrics) $h_{\natural,R}=\exp(N\log\lambda_{\natural,R})$ with conical singularities.  To glue these  with our local models, we require the next result which shows that on an annulus, $\lambda_\text{loc,R}$ closely matches $\lambda_{\natural,R}=\log r^{2R\tilde\beta}$.

\begin{lemma}[Gluing Lemma]\label[lemma]{lemma: annulus gluing}
Define $\tilde \beta=\frac12(|x|^2+|y|^2)$ and fix  $r_{\min},r_{\max}$ satisfying $0<r_{\min}<r_{\max}<1$.  On the annulus $r_{\min}\leq r\leq r_{\max}$, the differences
        $$\log\lambda_{\text{loc},R}-\log r^{2R\tilde\beta},\quad
        \del(\log\lambda_{\text{loc},R}-\log r^{2R\tilde\beta}),\quad
        \delbar(\log\lambda_{\text{loc},R}-\log r^{2R\tilde\beta})$$
    all decay uniformly like $R^2$.
\end{lemma}
\begin{proof}
    We calculate
    \begin{align}
        \log\left(\frac{2\beta r^{2R\beta}}{|x|^2-|y|^2r^{4R\beta}}\right)-\log r^{2R\tilde\beta}
            &=\log(2\beta)+2R(\beta-\tilde\beta)\log r-\log\left(|x|^2-|y|^2r^{4R\beta}\right)
                \notag\\
            &=\log(2\beta)-2R|y|^2\log r-\log\left(2\beta-|y|^2(4R\beta\log r+O(R^2))\right)
                \notag\\
            &=-2R|y|^2\log r+\frac{|y|^2\cdot 4R\beta\log r}{2\beta}+O(R^2)
                \notag\\
            &\in O(R^2).
    \end{align}
    In the second equality we use the Taylor expansion of $r^{4R\beta}$ in $R$, and in the third equality we use $\log(a+Rb)=\log a+R\frac ba+O(R^2)$.  Next, we have
    \begin{align}
        \del(\log\lambda_{\text{loc},R}-\log r^{2R\tilde\beta})
            &=R\beta\frac{|x|^2+|y|^2r^{4R\beta}}{|x|^2-|y|^2r^{4R\beta}}\frac{\de z}z-R\tilde\beta\,\frac{\de z}z
                \notag\\
            &=R\left(\beta\frac{|x|^2+|y^2|}{|x|^2-|y|^2}+O(R)\right)\frac{\de z}z-\frac R2(|x|^2+|y|^2)\frac{\de z}z
                \notag\\
            &\in O(R^2)
    \end{align}
The proof for $\delbar(\log\lambda_{\text{loc},R}-\log r^{2R\tilde\beta})\in O(R^2)$ is identical.
\end{proof}

\subsection{Local First Variations and Properties}\label{subsec: Local First Variations}

We are interested in understanding how the model metric $h_{\mathrm{loc}, R}$ in \Cref{def: Local Harmonic Model Metric}
changes as the pair $(x,y)$ changes.  To this end, let $(x,y)$ be as in  \Cref{def: Local Harmonic Model Metric} and assume $(\dot x,\dot y)$ satisfies $\dot yx+y\dot x=0$ and $x^\dagger\dot x+\dot y y^\dagger=0$ as motivated by the unitary hyperpolygon tangent space equations \eqref{eqn: linearized moment map} and \eqref{eqn: gauge orbit orthogonality} in \Cref{prop: hyperpolygon space tangent space}.  In this section, we are interested in finding a harmonic representative of the hypercohomology class $[(\dot{\eta}, \dot{\mathcal{F}}, \dot{\varphi})]=d \mathcal{T}_{(x,y)}(\dot{x}, \dot{y})$. Since we are working locally, by harmonic representative we just mean $\dot{\nu}_R$ solving \eqref{eqn: Coulomb gauge equation}.

\subsubsection{Harmonic Representative}
\begin{definition}[Local first variation of $h_{\mathrm{loc}, R}$]\label{def: local first variation}
Pick a basis in which we can write $x=(x_1,0)^\intercal$, $y=(0,y_2)$, and $\dot x=(\dot x_1,\dot x_2)^\intercal$, $\dot y=(\dot y_1,\dot y_2)$.  Then $N=\mathrm{diag}(1, -1)$. 
The \emph{local first variation} is \begin{equation}\dot\nu_{\mathrm{loc}, R}=\begin{pmatrix}\dot\nu_{11,R}&0\\\dot\nu_{21,R}&-\dot\nu_{11,R}\end{pmatrix}\end{equation} where
    \begin{equation}\label{eqn: nu dot component formulas}
        \dot\nu_{11,R}=\frac{\dot y_2\bar y_2}{|x|^2-|y|^2r^{4R\beta}}(r^{4R\beta}-1),
            \qquad\dot\nu_{21,R}=\left(\lambda_{\text{loc},R}^2-1\right)\frac{\dot x_2}{x_1}.
    \end{equation}
\end{definition}
The origin of this expression will be explained in \Cref{lemma: nu is infinitesimal gauge transformation}, but first we prove the most important property in  \Cref{prop: local Coulomb gauge condition}: that it gives the harmonic representative of the Higgs bundle deformation. After that, we calculate the norm of the harmonic representative.

We encourage the reader to skip the technical proofs in this section upon a first reading; only \Cref{prop: local metric pairing} will be needed for our main result \Cref{thm: HK degeneration}.

\begin{lemma}\label[lemma]{prop: local Coulomb gauge condition}
  The $\dot\nu_{\mathrm{loc},R}$ defined in \Cref{def: local first variation} satisfies the complex Coulomb gauge equation \eqref{eq:ingaugetriple}, i.e. 
    \begin{equation}\label{eqn: Coulomb gauge equation local version}
        0=
        -R^{-1}\del^{h_R}\delbar\dot\nu_{\mathrm{loc}, R}+R[\varphi^{\dagger_{h_R}},\dot\varphi+[\dot\nu_{\mathrm{loc}, R},\varphi]].
    \end{equation}
\end{lemma}
\begin{proof}
    We will drop some of the $\text{loc}$ and $R$ subscripts for notational simplicity.  Working in the basis described above, we calculate
    \begin{align}
        \delbar\dot\nu_{11}
            &=\dot y_2\bar y_2\frac{2R\beta\,r^{4R\beta}(|x|^2-|y|^2r^{4R\beta})-(r^{4R\beta}-1)(-|y|^2 2R\beta r^{4R\beta})}{(|x|^2-|y|^2r^{4R\beta})^2}\frac{\de\bar z}{\bar z}
                \notag\\
            &=\dot y_2\bar y_2\frac{2R\beta\left(|x|^2-|y|^2\right)r^{4R\beta}}{(|x|^2-|y|^2r^{4R\beta})^2}\frac{\de\bar z}{\bar z}
                =R\dot y_2\bar y_2\lambda^2\frac{\de\bar z}{\bar z}
                    \label{eqn: delbar nu11 formula}
    \end{align}
    and
    \begin{equation}\label{eqn: delbar nu12 formula}
        \delbar\dot\nu_{21}=\delbar(\lambda^2-1)\frac{\dot x_2}{x_1}=(2\lambda\delbar\lambda)\frac{\dot x_2}{x_1}=2\frac{\dot x_2}{x_1}\lambda^2\delbar\log\lambda.
    \end{equation}
    Recall $\varphi=xy\frac{\de z}z$ and $\dot\varphi=(\dot xy+x\dot y)\frac{\de z}z$.  Define
    \begin{equation}
        \dot\Phi=\dot\varphi+[\dot\nu,\varphi].
    \end{equation}
    In our chosen basis
        $$\dot\varphi=\begin{pmatrix}\dot\varphi_{11}&\dot\varphi_{12}\\0&-\dot\varphi_{11}\end{pmatrix}\frac{\de z}z=\begin{pmatrix}x_1\dot y_1&\dot x_1y_2+x_1\dot y_2\\0&\dot x_2y_2\end{pmatrix}\frac{\de z}z.$$
    Decompose $\dot\Phi=\dot\Phi_{(0)}+\dot\Phi_{(1)}$ into semisimple and nilpotent parts, and similarly decompose $\dot\nu=\dot\nu_{(-1)}+\dot\nu_{(0)}$ and $\dot\varphi=\dot\varphi_{(0)}+\dot\varphi_{(1)}$.  Using the assumption $x_1\dot y_1+\dot x_2y_2=0$,
    \begin{align}
        \dot\Phi_{(0)}
            =\dot\varphi_{(0)}+[\dot\nu_{(-1)},\varphi]
                &=\left(x_1\dot y_1-\dot\nu_{21}x_1y_2\right)\begin{pmatrix}1&0\\0&-1\end{pmatrix}\frac{\de z}z
                    \notag\\
                &=\left(x_1\dot y_1-(\lambda^2-1)\dot x_2y_2\right)\begin{pmatrix}1&0\\0&-1\end{pmatrix}\frac{\de z}z
                    \notag\\
                &=\lambda^2x_2\dot y_2\begin{pmatrix}1&0\\0&-1\end{pmatrix}\frac{\de z}z
                    \label{eqn: Phi_(0) formula}
    \end{align}
    and
    \begin{align}
        \dot\Phi_{(1)}&=\dot\varphi_{(1)}+[\dot\nu_{(0)},\varphi]=\left(\dot x_1y_2+x_1\dot y_2+2\dot\nu_{11}x_1y_2\right)\begin{pmatrix}0&1\\0&0\end{pmatrix}\frac{\de z}z.
            \label{eqn: Phi_(1) formula}
    \end{align}
    Now it is easy to see that the upper right entry of \eqref{eqn: Coulomb gauge equation local version} is zero.  Since the output of the PDE is trace-free, it suffices to prove equality for the upper left and lower left entries.  Looking at the upper left entry (and simplifying the tricky Lie brackets) gives the equation
    \begin{equation}\label{eqn: upperleft entry of Coulomb gauge eqn}
        -R^{-1}(\del\delbar\dot\nu_{11})=R\lambda^2\bar{x_1y_2}(\dot\varphi_{12}+2\dot\nu_{11}x_1y_2)r^{-2}\,\de\bar z\wedge\de z.
    \end{equation}
    Using \eqref{eqn: del log lambda} and \eqref{eqn: delbar nu11 formula}, the left-hand side is
    \begin{align}
        \text{LHS}
            =-R^{-1}\del(R\dot y_2\bar y_2\lambda^2)\frac{\de\bar z}{\bar z}
                &=-\dot y_2\bar y_2(2\lambda\del\lambda)\frac{\de\bar z}{\bar z}
                    =\dot y_2\bar y_2(2\lambda^2\del\log\lambda)\frac{\de\bar z}{\bar z}
                        \notag\\
            &=2\dot y_2\bar y_2\lambda^2\left(R\beta\frac{|x|^2+|y|^2r^{4R\beta}}{|x|^2-|y|^2r^{4R\beta}}\right)r^{-2}\,\de\bar z\wedge\de z
    \end{align}
    On the other hand, the right-hand side is
    \begin{align}
        \text{RHS}
            &=R\lambda^2\bar{x_1y_2}\left(\dot\varphi_{12}+\frac{2\dot y_2\bar y_2(r^{4R\beta}-1)}{|x|^2-|y|^2r^{4R\beta}}x_1y_2\right)r^{-2}\,\de\bar z\wedge\de z
                \notag\\
            &=R\lambda^2\left(\bar{x_1y_2}\dot x_1y_2+\bar{x_1y_2}x_1\dot y_2+\frac{2|x|^2|y|^2\dot y_2\bar y_2(r^{4R\beta}-1)}{|x|^2-|y|^2r^{4R\beta}}\right)r^{-2}\,\de\bar z\wedge\de z
                \notag\\
            &=R\dot y_2\bar y_2\lambda^2\left(|y|^2+|x|^2+\frac{2|x|^2|y|^2(r^{4R\beta}-1)}{|x|^2-|y|^2r^{4R\beta}}\right)r^{-2}\,\de\bar z\wedge\de z
                \notag\\
            &=R\dot y_2\bar y_2\lambda^2\!\left(\frac{|x|^2|y|^2+|x|^4-(|y|^2+|x|^2)|y|^2r^{4R\beta}+2|x|^2|y|^2r^{4R\beta}-2|x|^2|y|^2}{|x|^2-|y|^2r^{4R\beta}}\right)\!\frac{\de\bar z\wedge\de z}{r^2}
                \notag\\
            &=R\dot y_2\bar y_2\lambda^2\left(\frac{|x|^4-|x|^2|y|^2+(|x|^2|y|^2-|y|^4)r^{4R\beta}}{|x|^2-|y|^2r^{4R\beta}}\right)r^{-2}\,\de\bar z\wedge\de z
                \notag\\
            &=R\dot y_2\bar y_2\lambda^2\left(2\beta\frac{|x|^2+|y|^2r^{4R\beta}}{|x|^2-|y|^2r^{4R\beta}}\right)r^{-2}\,\de\bar z\wedge\de z.
    \end{align}
    The third equality used $\bar x_1\dot x_1=\dot y_2\bar y_2$ and the last used $2\beta=|x|^2-|y|^2$.  This proves \eqref{eqn: upperleft entry of Coulomb gauge eqn}.
    
    Using \eqref{eqn: Phi_(0) formula}, the lower left entries of \eqref{eqn: Coulomb gauge equation local version} give the equation
    \begin{equation}
        R^{-1}(\del\delbar\dot\nu_{21}-2(\del\log\lambda)\wedge\delbar\dot\nu_{21})=2R\lambda^2\,\bar{x_1y_2}(\lambda^2x_2\dot y_2)r^{-2}\,\de\bar z\wedge\de z
    \end{equation}
    By \eqref{eqn: delbar nu12 formula}, the left-hand side is
    \begin{align}
        \text{LHS}
            &=-R^{-1}(\delbar\del\dot\nu_{21}-2\delbar\dot\nu_{21}\wedge\del\log\lambda)
                \notag\\
            &=-R^{-1}(2\delbar(\lambda\del\lambda)-4\lambda\delbar\lambda\wedge\del\log\lambda)\frac{\dot x_2}{x_1}
                \notag\\
            &=-R^{-1}(2\delbar(\lambda^2\del\log\lambda)-4\lambda^2\delbar\log\lambda\wedge\del\log\lambda)\frac{\dot x_2}{x_1}
                \notag\\
            &=-R^{-1}\left(2\lambda^2\delbar\del\log\lambda\right)\frac{\dot x_2}{x_1}
                \notag\\
            &=-R\left(2|x|^2|y|^2\lambda^4\right)\frac{\dot x_2}{x_1}r^{-2}\,\de\bar z\wedge\de z
                \notag\\
            &=-2R\,\bar x_1\dot x_2y_2\,\bar y_2\lambda^4r^{-2}\,\de\bar z\wedge\de z
                \notag\\
            &=2R\,\bar x_1x_1\dot y_2\,\bar y_2\lambda^4r^{-2}\,\de\bar z\wedge\de z
    \end{align}
    which equals the right-hand side.  In the last equality we used $\dot x_2y_2=-x_1\dot y_1$.  This proves the lemma.
\end{proof}

\subsubsection{Infinitesimal deformation of local model metric} 

The next proposition explains the origin of the formula for $\dot{\nu}_{\mathrm{loc}, R}$ in \Cref{def: local first variation}.
Let $(x(t),y(t))=(x+t\dot x,y+t\dot y)$, let $N(t)=2\frac{x(t)x(t)^\dagger}{|x(t)|^2}$, and consider $h_{\mathrm{loc}, R}(t)$ the local model metric for $(x(t), y(t))$ given in \Cref{def: Local Harmonic Model Metric}. In \Cref{rem: gamma deforms metric}, we wrote that finding the infinitesimal gauge transformation to produce the harmonic representative can be thought of as finding infinitesimal change of hermitian metric. We now verify:

\begin{lemma}\label[lemma]{lemma: nu is infinitesimal gauge transformation} The local first variation
    $\dot\nu_{\mathrm{loc}, R}$ defined in \Cref{def: local first variation} 
    satisfies \[e^{t\dot\nu_{\mathrm{loc},R}^{\dagger}}h_{\mathrm{loc}, R}(0)e^{t\dot\nu_{\mathrm{loc},R}}=h_{\mathrm{loc},R}(t) + O(t^2).\]
\end{lemma}

\begin{remark}
Note $\dot{\nu}_R \to 0$ pointwise as $R\to0$, which is compatible with the fact that $\lambda_R \to 1$.
\end{remark}
\begin{proof} We again drop the subscripts $\mathrm{loc}$ and $R$ for notational simplicity. From \Cref{def: Local Harmonic Model Metric},
 $$h_R(t)=\sqrt{h_{\det}}\begin{pmatrix}\lambda_R(t)&0\\0&\lambda_R(t)^{-1}\end{pmatrix}=\sqrt{h_{\det}}\exp(\underbrace{f_R(t)N(t)}_{K_R(t)})$$
where $f_R(t):=\log\lambda_R(t)$ for
    $$\lambda_R(t)=\frac{2\beta r^{2R\beta}}{|x(t)|^2-|y(t)|^2r^{4R\beta}}.$$
One should think of $N(t)$ as a normalization of $(x(t)x(t)^\dagger)^\perp\in\mathfrak{su}(2)$.  We first compute the first variation $\dot N$ of $N(t)$.
    Using the quotient rule,
    \begin{align}
        \dot N
            &=2\frac{(\dot xx^\dagger+x\dot x^\dagger)^\perp|x|^2-(xx^\dagger)^\perp(x^\dagger\dot x+\dot x^\dagger x)}{|x|^4}
                \notag\\
            &=\frac2{|x|^2}
            \begin{pmatrix}
                \dot x_1\bar x_1+x_1\bar{\dot x}_1  &x_1\bar{\dot x}_2\\
                \dot x_2\bar x_1                    &0
            \end{pmatrix}^\perp
            -\frac{\bar x_1\dot x_1+\bar{\dot x}_1x_1}{|x|^4}
            \begin{pmatrix}
                |x|^2   &0\\
                0       &-|x|^2
            \end{pmatrix}
                \notag\\
            &=\frac2{|x|^2}
            \begin{pmatrix}
                0                   &x_1\bar{\dot x}_2\\
                \bar x_1\dot x_2    &0
            \end{pmatrix}
                \notag\\
            & =2\begin{pmatrix}0&\bar{\left(\frac{\dot x_2}{x_1}\right)}\\\frac{\dot x_2}{x_1}&0\end{pmatrix}.
                \label{eqn: N dot formula}
    \end{align}

Let $K_R(t)=f_R(t)N(t)$. Then the first variation $\dot{K}_R$ (from $K_R(t)=K_R+t\dot K_R+O(t^2)$) can be written in terms of $N, \dot{N}, f_R \dot{f}_R$, as follows.
        \begin{align*}
        \dot K
            &=\dot f_RN+f_R\dot N\\
            &=-\frac{(x^\dagger\dot x+\dot x^\dagger x)-(\dot y y^\dagger+y^\dagger\dot y)r^{4R\beta}}{|x|^2-|y|^2r^{4R\beta}}N+f_R\dot N\\
            &=\frac{2\text{Re}(\dot y_2\bar y_2)}{|x|^2-|y|^2r^{4R\beta}}(r^{4R\beta}-1)N+f_R\dot N.
    \end{align*}
    Using the formulas for $N$ and $\dot{N}$ (\eqref{eqn: N formula} and \eqref{eqn: N dot formula} respectively), we can write
        $$\dot K_R=\begin{pmatrix}a_R&\bar b_R\\b_R&-a_R\end{pmatrix}$$
    where
    \begin{equation}\label{eqn: a and b formulas}
        a_R=\frac{2\text{Re}(\dot y_2\bar y_2)}{|x|^2-|y|^2r^{4R\beta}}(r^{4R\beta}-1),
            \qquad
                b_R=2f_R\frac{\dot x_2}{x_1}.
    \end{equation}
    
    Observe that $K=\begin{pmatrix}f_R&0\\0&-f_R\end{pmatrix}$.  We once again drop the subscript $R$ and write $\dot\nu=\begin{pmatrix}n&0\\m&-n\end{pmatrix}$ for $n$ and $m$ complex-valued functions, respectively.  A close look at the Taylor expansions reveals
        $$\frac{\de}{\de t}\Big|_{t=0}\exp(K+t\dot K)=\begin{pmatrix}ae^f&\frac{\bar b\sinh f}f\\\frac{b\sinh f}f&b(\sinh f-\cosh f)\end{pmatrix}$$
        $$\frac{\de}{\de t}\Big|_{t=0}\left(\exp\left(t\begin{pmatrix}n&0\\m&-n\end{pmatrix}^\dagger\right)\cdot h_R\cdot\exp\left(t\begin{pmatrix}n&0\\m&-n\end{pmatrix}\right)\right)=\begin{pmatrix}2\Re(n)e^f&e^{-f}\bar m\\e^{-f}m&-2\Re(n)e^{-f}\end{pmatrix}.$$
    Setting the two expressions equal and solving for $\Re(n)$ and $m$ gives
        $$\begin{pmatrix}\Re(n)&0\\m&-\Re(n)\end{pmatrix}=\begin{pmatrix}a/2&0\\(b e^f\sinh f)/f&-a/2\end{pmatrix}=\begin{pmatrix}a/2&0\\b(e^{2f}-1)/(2f)&-a/2\end{pmatrix}.$$
    Substituting in the values for $a_R$ and $b_R$ from \eqref{eqn: a and b formulas} proves the lemma.
\end{proof}

\subsubsection{Properties}

Recall from \eqref{eq:ingaugetriple} that the hyperk\"ahler metric $g_R$ on $\mathcal{M}^{R}(\vec \alpha)$ is an $L^2$ norm of the harmonic representative
\begin{equation}
  \mathtt{H}_R:=R^{-1/2}\underbrace{\left(\dot{\eta} - \delbar_E \dot{\nu}\right)}_{\mathtt{H}_R^{0,1}} + R^{1/2}\underbrace{\left(\dot{\varphi}-[\varphi, \dot{\nu}]\right)}_{\mathtt{H}^{0,1}_R},
\end{equation} 
Applying Stokes' theorem, we obtained in \eqref{eq:hk expression}
\begin{equation}
|(\dot{\eta}, \dot{\mathcal{F}}, \dot{\varphi})|^2_{g_R} = \sum_{i=1}^n\left(\lim_{\delta'\to0}\oint_{\partial B_{\delta'}(p_i)} R^{-1} \langle\dot\nu_R,\delbar\dot\nu_R\rangle_{h_R}\right) + \int_{\C\P^1} R\langle\dot\varphi,\dot{\varphi}-[\varphi, \dot{\nu}]\rangle_{h_R}.
\end{equation}
We now evaluate this integral for our local models, restricting the integral to the domain with $B_{\delta}=B_{\delta}(p_i)$.  The following proposition is the main result of this section, and is a key ingredient in the proof of our main theorem, \Cref{thm: HK degeneration}.
\begin{restatable}{proposition}{localmetricpairing}\label{prop: local metric pairing}
    Let $0<\delta'<\delta<1$ and $B_\delta=B_\delta(0)$.  Then 
        $$\lim_{R\to0}\left(\lim_{\delta'\to0}\int_{\partial B_{\delta'}}R^{-1}\langle\dot\nu_R,\delbar\dot\nu_R\rangle_{h_R}+\int_{B_\delta}R\langle\dot\varphi,\dot\varphi+[\dot\nu_R,\varphi]\rangle_{h_R}\right)=2\pi i(|\dot x|^2+|\dot y|^2).$$
\end{restatable}
The proof follows from \Cref{lem: pairing of A dot in local model} and \Cref{lemma: pairing of Phi dot in local model}.  As before, we will drop the subscripts $\text{loc}$ and $R$, but we will keep in mind that many of these objects still depend on $R$.
While the proposition is written without reference to a basis, our expression for $\dot{\nu}_{\mathrm{loc}, R}$ assumes that $x=(x_1,0)^\intercal$, $y=(0,y_2)$, and $\dot x=(\dot x_1,\dot x_2)^\intercal$, $\dot y=(\dot y_1,\dot y_2)$, so we will continue to assume that.  

\begin{lemma}\label[lemma]{lem: pairing of A dot in local model}
    Let $B_\delta=B_\delta(0)$ be the ball of radius $\delta$ centered at 0.  Then
        $$\lim_{R\to0}\lim_{\delta'\to0}\int_{\partial B_{\delta'}}R^{-1}\langle\dot\nu,\delbar\dot\nu\rangle_{h_R}=2\pi i(|\dot x_2|^2-|\dot y_1|^2).$$
\end{lemma}
\begin{proof}
    We expand the trace pairing using \eqref{eqn: delbar nu11 formula} and \eqref{eqn: delbar nu12 formula}
    \begin{align*}
        \langle\dot\nu,\delbar\dot\nu\rangle_{h_R}
            &=\tr(h_R^{-1}\dot\nu^\dagger h_R(\delbar\dot\nu))=\bar{\dot\nu}_{11}\delbar\dot\nu_{11}+(\lambda^{-2}(\lambda^2-1)\cdot2\lambda\delbar\lambda)\frac{|\dot x_2|^2}{|x_1|^2}\\
            &=|\dot\nu_{11}\delbar\dot\nu_{11}|+(2\lambda\delbar\lambda-2\delbar\log\lambda)\frac{|\dot x_2|^2}{|x_1|^2}.
    \end{align*}
    Observe that the above is a radially symmetric function times $\frac{\de\bar z}{\bar z}$, so integration along $\partial B_{\delta'}$ simply replaces $\frac{\de\bar z}{\bar z}$ with $-2\pi i$ and $r$ with $\delta'$ in the above formulas.  Taking the limit as $\delta'\to0$ only leaves the last term, which, following the proof of \Cref{lemma: integral of dd log lambda}, integrates to
    \begin{align*}
        \lim_{R\to0}\lim_{\delta'\to0}\int_{\partial B_{\delta'}}R^{-1}\langle\dot\nu,\delbar\dot\nu\rangle_{h_R}
            &=-2\frac{|\dot x_2|^2}{|x_1|^2}\lim_{R\to0}\lim_{\delta'\to0}\int_{\partial B_{\delta'}}R^{-1}\delbar\log\lambda\\
            &=2\pi i\frac{|\dot x_2|^2}{|x|^2}(2\beta)\\
                &=2\pi i\frac{|\dot x_2|^2}{|x|^2}(|x|^2-|y|^2)\\
            &=2\pi i\left(|\dot x_2|^2-\frac{|\dot x_2|^2|y|^2}{|x|^2}\right)\\
            &=2\pi i\left(|\dot x_2|^2-|\dot y_1|^2\right)
    \end{align*}
    where the last equality uses $0=\dot yx+y\dot x=\dot y_1x_1+y_2\dot x_2$.
\end{proof}

\begin{lemma}\label[lemma]{lemma: pairing of Phi dot in local model}
    Let $0<\delta\leq1$ and let $B_\delta=B_\delta(0)$.  Then
        $$\lim_{R\to0}\int_{B_\delta}R\langle\dot\varphi,\dot\varphi+[\dot\nu,\varphi]\rangle_{h_R}\to2\pi i\left(2|\dot y_1|^2+|\dot x_1|^2+|\dot y_2|^2\right).$$
\end{lemma}
\begin{proof}
    The pairing decomposes as
    \begin{align*}
        \langle\dot\varphi,\dot\varphi+[\dot\nu,\varphi]\rangle_{h_R}
            &=\tr\left(\left(\dot\varphi_{(0)}^\dagger+\lambda^2\dot\varphi_{(1)}^{\dagger}\right)\dot\Phi\right)\\
            &=\tr\left(\dot\varphi_{(0)}^\dagger\dot\Phi_{(0)}+\lambda^2\dot\varphi_{(1)}^\dagger\dot\Phi_{(1)}\right)\frac{\de\bar z}{\bar z}\wedge\frac{\de z}z.
    \end{align*}
    Recalling $\delbar\del\log\lambda=R^2\lambda^2|x|^2|y|^2r^{-2}\,\de\bar z\wedge\de z$ and \Cref{lemma: integral of dd log lambda}, and substituting the formula \eqref{eqn: Phi_(0) formula} for $\dot\Phi_{(0)}$, the first term integrates to
    \begin{align}
        \lim_{R\to0}\int_{B_\delta}R\tr\left(\dot\varphi_{(0)}^\dagger\dot\Phi_{(0)}\right)
            &=\lim_{R\to0}\int_{B_\delta}2R|x_1|^2|\dot y_1|^2\lambda^2r^{-2}\,\de z\wedge\de z
                \notag\\
            &=2\frac{|\dot y_1|^2}{|y|^2}\lim_{R\to0}\int_{B_\delta}R^{-1}\delbar\del\log\lambda
            =4\pi i|\dot y_1|^2.
                \label{eqn: phi Phi semisimple parts integral}
    \end{align}
    Applying the formula \eqref{eqn: Phi_(1) formula} for $\dot\Phi_{(1)}$, the next term can be broken up into two parts $A$ and $B$:
    \begin{align*}
        \tr\left(\lambda^2\dot\varphi_{(1)}^\dagger\dot\Phi_{(1)}\right)r^{-2}\,\de\bar z\wedge\de z
            &=\left(\lambda^2\bar{\dot\varphi}_{12}(\dot\varphi_{12}+2\dot\nu_{11}x_1y_2)\right)r^{-2}\,\de\bar z\wedge\de z\\
            &=\underbrace{|\dot\varphi_{12}|^2\lambda^2r^{-2}\,\de\bar z\wedge\de z}_A+\underbrace{2\lambda^2\bar{\dot\varphi}_{12}x_1y_2\dot\nu_{11}r^{-2}\,\de\bar z\wedge\de z}_B.
        \end{align*}
    The integral of $RA$ is carried out the same way as above:
    \begin{align}
        \lim_{R\to0}\int_{B_\delta}RA
            &=\frac{|\dot\varphi_{12}|^2}{|x|^2|y|^2}\lim_{R\to0}\int_{B_\delta}\delbar\del\log\lambda
                =2\pi i\frac{|\dot\varphi_{12}|^2}{|x|^2}
                \notag\\
            &=2\pi i\frac1{|x|^2}\bar{(\dot x_1y_2+x_1\dot y_2)}(\dot x_1y_2+x_1\dot y_2)
                \notag\\
            &=2\pi i\frac1{|x|^2}\left(|y|^2|\dot x_1|^2+\bar{\dot x}_1\bar y_2x_1\dot y_2+\bar x_1\bar{\dot y}_2\dot x_1y_2+|x|^2|\dot y_2|^2\right)
                \notag\\
            &=2\pi i\frac1{|x|^2}\left(|y|^2|\dot x_1|^2+2|x|^2|\dot x_1|^2+|x|^2|\dot y_2|^2\right)
                \notag\\
            &=2\pi i\left(\frac{|y|^2|\dot x_1|^2}{|x|^2}+2|\dot x_1|^2+|\dot y_2|^2\right).
                \label{eqn: RA integral}
    \end{align}
    Next we integrate $RB$.  Recall from \eqref{eqn: nu dot component formulas} and \eqref{eqn: KW local model} the formulas
        $$\dot\nu_{11}=\frac{\dot y_2\bar y_2}{|x|^2-|y|^2r^{4R\beta}}(r^{4R\beta}-1),
            \qquad \lambda=\frac{2\beta r^{2R\beta}}{|x|^2-|y|^2r^{4R\beta}}.$$
    A routine integral calculation and repeated use of the relations $2\beta=|x|^2-|y|^2$ and $\bar x_1\dot x_1=\dot y_2\bar y_2$ yields
    \begin{align}
        \lim_{R\to0}\int_{B_\delta}RB
            &=(2i)2\pi\lim_{R\to0}\int_0^\delta2R\frac{(2\beta)^2r^{4R\beta}}{(|x|^2-|y|^2r^{4R\beta})^2}\bar{\dot\varphi}_{12}x_1y_2\frac{\dot y_2\bar y_2(r^{4R\beta}-1)r^{-1}}{|x|^2-|y|^2r^{4R\beta}}\,\de r
                \notag\\
            &=8\pi i(2\beta)^2|y|^2\bar{\dot\varphi}_{12}x_1\dot y_2\lim_{R\to0}\int_0^\delta R\frac{(r^{4R\beta}-1)r^{4R\beta-1}}{(|x|^2-|y|^2r^{4R\beta})^3}\,\de r
                \notag\\
            &=8\pi i(2\beta)^2|y|^2\bar{\dot\varphi}_{12}x_1\dot y_2\lim_{R\to0}\frac{(r^{4R\beta}-1)^2}{8\beta(|x|^2-|y|^2)(|x|^2-|y|^2r^{4R\beta})^2}\Big|_{r=0}^\delta
                \notag\\
            &=-2\pi i|y|^2\bar{\dot\varphi}_{12}x_1\dot y_2\frac{1}{|x|^4}
                \notag\\
            &=-2\pi i\frac{|y|^2}{|x|^4}\bar{(\dot x_1y_2+x_1\dot y_2)}x_1\dot y_2
                \notag\\
            &=-2\pi i\frac{|y|^2}{|x|^4}\left(|x|^2|\dot x_1|^2+|x|^2|\dot y_2|^2\right)
                \notag\\
            &=-2\pi i\frac{|y|^2}{|x|^2}\left(|\dot x_1|^2+|\dot y_2|^2\right)
                \notag\\
            &=-2\pi i\left(\frac{|y|^2|\dot x_1|^2}{|x|^2}+|\dot x_1|^2\right)
                \label{eqn: RB integral}
    \end{align}
    Combining \eqref{eqn: phi Phi semisimple parts integral}, \eqref{eqn: RA integral}, and \eqref{eqn: RB integral},
    \begin{equation}
        2\pi i\left(2|\dot y_1|^2+\frac{|y|^2|\dot x_1|^2}{|x|^2}+2|\dot x_1|^2+|\dot y_2|^2-\frac{|y|^2|\dot x_1|^2}{|x|^2}-|\dot x_1|^2\right)=2\pi i\left(2|\dot y_1|+|\dot x_1|^2+|\dot y_2|^2\right)
    \end{equation}
    as desired.
\end{proof}

\section{The Hyperkähler Metrics}\label{sec: The Hyperkahler Metrics}
Consider the map $\mathcal{T}_R:\mathcal X(\vec\beta)\into\mathcal M_R(\vec\alpha(R))$, which was defined for all $R\in(0,R_{\max})$.
Our main result in \Cref{thm: HK degeneration} states that in the limit $R\to0$, 
$\mathcal{T}_R^*(g_R)$ converges pointwise to $2\pi\cdot g_\text{HP}$ as $R \to 0$.

Together with \Cref{thm: T preserves holomorphic symplectic forms} ($\mathcal T^*\Omega_{J_1}=2\pi\,\Omega_{\text{HP},J_1}$), this proves that the full hyperk\"ahler structures converge.

In \Cref{subsec: analytic expansion of the harmonic metric}, we use the  local models in \Cref{sec: Local Model} (see \Cref{def: Local Harmonic Model Metric}) to produce an approximate harmonic metric $h_{\text{app},R}$, and use an implicit function theorem argument to prove its accuracy for small $R\ll1$.
Finally, in \Cref{sec:firstvar} we calculate the first variation of the approximation $h_{\text{app},R}$ as we perturb a Higgs bundle $(\mathcal E,\varphi)$, and use it to prove \Cref{thm: HK degeneration} in \Cref{sec:mainthm}.

Lastly, we prove that the Morse functions converge pointwise
    \begin{equation}\lim_{R \to 0} M_R= 2 \pi M_\text{HP},\end{equation}
in \Cref{thm: moment maps agree} of \Cref{sec:Morse} since the proof cleanly highlights some essential properties of the approximate harmonic metric $h_{\text{app}, R}$.  While this is a corollary of \Cref{thm: HK degeneration},  the convergence of the hyperk\"ahler metrics themselves requires a bit more notational complication and analytic set up; the complication can obscure just how beautiful the approximate harmonic metrics is!

\subsection{Analytic Expansion of the Harmonic Metric}
\label{subsec: analytic expansion of the harmonic metric}
In this section, we produce an approximation $h_{\text{app},R}$ for the harmonic metric associated to the Higgs bundle $(\mathcal E_{(\bx,\by)},\varphi_{(\bx,\by)})$ constructed in \Cref{subsec: Hyperpolygons to Higgs Map}.  First, we establish some notation.  Let $(\bx,\by) \in \mathcal{X}(\vec \beta)$ be a unitary hyperpolygon, and let
\begin{equation}N_i=\frac{2v_i}{|x_i|^2}=\frac{2(x_ix_i^\dagger)^\perp}{|x_i|^2} \qquad  i=1,\dots,n, \end{equation}
as in \eqref{eqn: N formula}.
\begin{lemma}[Interpretation of $\mu_{\SU(2)}(\bx, \by)=0$ condition]
    The condition $\mu_{\SU(2)}(\bx,\by)=0$ implies
    \begin{equation}
        \sum_i(2\beta_i+2|y_i|^2)N_i=\sum_i(|x_i|^2+|y_i|^2)N_i=0.
    \end{equation}
\end{lemma}
\begin{proof}
    Since $y_ix_i=0$ for all $i$, $w_i:=(y_i^\dagger y_i)^\perp=-|y_i|^2N_i$.  Thus,
    \begin{equation}\label{eqn: v - w formula}
        v_i-w_i=(|x_i|^2+|y_i|^2)N_i.
    \end{equation}
    The moment map \eqref{eqn: HP real moment map} implies that $2\beta_iN_i=(|x_i|^2-|y_i|^2)N_i$, so
    \begin{align*}
        \sum_i(2\beta_i+2|y_i|^2)N_i
            &=\sum_i(|x_i|^2+|y_i|^2)N_i=\sum_i v_i-w_i=\mu_{\SU(2)}(\bx,\by)=0.\qedhere
    \end{align*}
\end{proof}

\begin{definition}[Approximate harmonic metric $h_{\mathrm{app}, R}$]

   Let $(\bx,\by)$
    be a unitary hyperpolygon, fix $0<\delta\ll1$ so that $B_{2\delta}(p_i) \subset \C \subset \CP^1$ do not intersect,
    and consider the modified cone angles $\tilde\beta_i=\frac12(|x_i|^2+|y_i|^2)$.

  We define
    \begin{equation}\label{eqn: approximate metric}
        h_{\text{app}, R}=\sqrt{h_{\det}}\exp\left(\chi_\infty\Lambda_{\natural,R}+\sum_{i=1}^n\chi_i f_{i,R} N_i\right),
    \end{equation}
where:
\begin{itemize}
    \item Away from the punctures, we use the approximately flat\footnote{
        The curvature of $\sqrt{h_{\det}}\exp(\Lambda_{\natural,R})$ differs from $\delbar\del\Lambda_{\natural,R}=0$ by terms coming from the Baker-Campbell-Hausdorff formula (or more precisely the Zassenhaus formula), which all pointwise vanish to order $R^2$.
    } metric $\sqrt{h_{\det}}\exp(\Lambda_{\natural,R})$ where
        $$\Lambda_{\natural,R}=\sum_{i=1}^n\log|z-p_i|^{2R\tilde\beta_i}N_i=\sum_{i=1}^n\log|z-p_i|^R(|x_i|^2+|y_i|^2)N_i.$$
\item Near the punctures, we take our local models from \eqref{eqn: KW local model}:
    \begin{equation}\label{eqn: lambda_i definition}
        \lambda_{\text{loc},i,R}=\frac{2\beta_i |z-p_i|^{2R\beta_i}}{|x_i|^2-|y_i|^2|z-p_i|^{4R\beta_i}},
    \end{equation}
    and let $f_{i,R}=\log\lambda_{\text{loc},i,R}$.  
    \item Take $\chi_i$ be a bump function equal to 1 on $B_\delta(p_i)$ and equal to 0 outside $B_{2\delta}(p_i)$, and let $\chi_\infty=1-\sum\chi_i$.  
\end{itemize}
\end{definition}

\begin{proposition}[Construction of approximate hermitian metric $h_{\mathrm{app}, R}$]\label[proposition]{prop: h app construction} Shrinking $R_{\max}$ if necessary, $h_{\text{app},R}$ is a globally-defined hermitian metric adapted to the parabolic structure with weights $\vec \alpha(R)$ for all $R\in(0,R_{\max})$.
\end{proposition}

\begin{proof}
   There are three things to show: (1) that $h_{\mathrm{app}, R}$ is smooth over infinity, (2) that $h_{\mathrm{app}, R}$ is adapted to the parabolic structure with weights $\vec \alpha(R)$, and (3) for $R$ sufficiently small $h_{\mathrm{app}, R}$ is positive definite.  For (1), note that $\Lambda_{\natural,R}$ extends over $\infty$ because $\sum 2\tilde\beta_iN_i=\sum v_i-w_i=0$.  For (2), since we assumed $\{B_{2 \delta}(p_i)\}_{i=1}^n$ do not intersect, on $B_{\delta}(p_i)$, $h_{\text{app}, R}= \sqrt{h_{\det}} \exp( f_{i, R} N_i)$ agrees with the local model metric \eqref{eq:localmetric} which by \Cref{lemma:localproperties} is adapted to the parabolic structure $(F_i,\alpha_i(R))$.  For (3), we need to verify positive-definiteness on the gluing annulus $B_{2 \delta}(p_i)\backslash B_{\delta}(p_i)$ where, since $\chi_j=0$ for $j \neq i$, we have 
   \[         h_{\text{app}, R}=\sqrt{h_{\det}}\exp\left((\chi_\infty 
\log|z-p_i|^{2R\tilde\beta_i} + \chi_i f_{i,R}) N_i + R
\sum_{j \neq i} 2\tilde\beta_j \log|z-p_j|N_j \right). \]
This exponent decays uniformly on the annulus as $R\to0$, so the eigenvalues of $h_{\app,R}$ all approach 1 uniformly. Thus, possibly after shrinking $R_{\max}$, it is positive-definite.
\end{proof}

\begin{remark}[Pointwise Convergence] 
Note that the metric $h_{\text{app},R}$ converges pointwise to the standard metric $h_0=\text{Id}$, which will simplify many calculations. If this seems surprising that something so canonical would appear in the $R \to 0$ limit, note that the map from hyperpolygons to Higgs bundles is picking a very special frame in which $\mathcal{E}= \mathcal{O} \oplus \mathcal{O}$. As a sanity check note that the $SU(2)$-action on the space of unitary hyperpolygons corresponds to a constant unitary gauge transformation on $\mathcal{O} \oplus \mathcal{O}$, which does not change the metric $h_0$.
\end{remark}
\bigskip

The next lemma utilizes the implicit function theorem to perturb our $h_{\text{app},R}$ to an actual solution for all $R$ in an interval $(0,R_{\max})$ (possibly shrinking the $R_{\max}$ introduced in \Cref{sec:mapontan}).  Let $h_0=\lim_{R\to0}h_{\text{app},R}=\text{Id}$, and let $i\End_0(\mathcal E,h_0)$ be the sheaf of trace-free $h_0$-hermitian endomorphisms equipped with the norm $|\cdot|_{h_0}$.  We use the weighted $b$-Sobolev spaces from \Cref{sec:analytic}.

\begin{proposition} \label[proposition]{prop: approximate h accuracy}
    Fix $\eps\in(0,\frac12)$.  There exists an $R_\text{max}>0$ and a family $\gamma_R\in L_{2-\eps}^{2,2}(i\End_0(\mathcal E,h_0))$ for $0\leq R<R_\text{max}$ which is analytic in $R$ such that $\gamma_0= 0$
    and $h_{\text{app},R}\cdot e^{\gamma_R}:=e^{\gamma_R^\dagger}h_{\text{app},R}e^{\gamma_R}$ is $R$-harmonic with respect to $(\mathcal E,\varphi)$ for all $0<R<R_\text{max}$.
\end{proposition}
The rough idea of the proof is as follows: For a generic hermitian metric $h$ adapted to the parabolic structure, the error in Hitchin's equation $F_{\nabla(\delbar_E,h)}^\perp+R^2[\varphi^{\dagger_{h}},\varphi]$ lives in the Sobolev space $L_{-1-\eps_R}^{0,2}(\Omega_C^{1,2}\otimes i\End_0(\mathcal E,h))$ for $0<\eps_R<4R\beta_i$, as the leading order term as $r\to0$ is asymptotically $r^{-2}\,\de\bar z\wedge\de z$ times an upper-triangular matrix.  Thus, the correction $\gamma_R$ to $h$ will need to lie in $L_{1-\eps_R}^{2,2}(i\End_0(\mathcal E,h))$, and its limiting behavior is poorly behaved because we lose regularity as $R\to0$.  However, we shall show that our construction $h_{\text{app},R}$ produces an error term in the \emph{subspace} $L_{-\eps}^{0,2}\subset L_{-1+\eps_R}^{0,2}$, where crucially $\eps\in(0,\frac12)$ is \emph{fixed} with respect to $R$.\footnote{
    Any $\eps\in(0,1)$ suffices, but we use $\eps\in(0,\frac12)$ for future convenience.
}
Thus, we may carry out a more refined implicit function theorem argument using a mapping $L_{2-\eps}^{2,2}\to L_{-\eps}^{0,2}$ to obtain our result.

To simplify our calculations, we work in unitary gauge with respect to the standard metric $h_0$.  Let $h_{\app,R}^{1/2}$ be the unique $h_0$-hermitian endomorphism solving $h_{\app,R}=(h_{\app,R}^{1/2})^\dagger h_{\app,R}$.  The Higgs bundle in unitary gauge is $(\nabla^{0,1}_{\app,R},\Psi_{\app,R})$ where
\begin{equation}
    \nabla _{\app,R}^{0,1}=\delbar + \delbar({h_{\app,R}^{1/2}})h_{\app,R}^{-1/2},
        \qquad \Psi_{\app,R}^{1,0}=h_{\app,R}^{1/2}\varphi\,h_{\app,R}^{-1/2}.
\end{equation}
Then $\nabla_{\app,R}=\nabla_{\app,R}^{0,1}+\nabla_{\app,R}^{1,0}$ where $\nabla_{\app,R}^{1,0}-\del=-(\nabla_{\app,R}^{0,1}-\delbar)^\dagger$.

\begin{proof}[Proof of \Cref{prop: approximate h accuracy}]
    For $\gamma\in L_{2-\eps}^{2,2}(i\End_0(\mathcal E,h_0))$, $e^\gamma$ acts like 
    \begin{align}
        \nabla_{\app,R}^{0,1}
            &\mapsto \nabla_\gamma^{0,1}:=e^\gamma\,\delbar(e^{-\gamma})+e^\gamma \nabla_{\app,R}^{0,1}e^{-\gamma}  \notag\\
        \Psi_{\app,R}^{1,0}
            &\mapsto\Psi_\gamma^{1,0}:=e^\gamma\Psi_{\app,R}^{1,0}e^{-\gamma}.
    \end{align}
    We again denote $\nabla_\gamma=\nabla_\gamma^{0,1}+\nabla_\gamma^{1,0}$ where $\nabla_\gamma^{1,0}-\del=-(\nabla_\gamma^{0,1}-\delbar)^\dagger$.  Note that we have dropped the subscripts $R$, but $\nabla_\gamma$ and $\Psi_\gamma^{1,0}$ still depend on $R$.  Consider the operator
        $$\mathcal N_R\from L_{2-\eps}^{2,2}(i\End_0(\mathcal E,h_0))\to L_{-\eps}^{0,2}(\Omega_C^{1,1}\otimes i\End_0(\mathcal E,h_0))$$
    given by
    \begin{equation}\label{eqn: def of N}
        \mathcal N_R(\gamma)=F_{\nabla_{\gamma}}^\perp+R^2\left[(\Psi_\gamma^{1,0})^\dagger,\Psi_\gamma^{1,0}\right].
    \end{equation}
    Observe $\mathcal N_R(\gamma)=0$ if and only if $(\nabla_\gamma^{0,1},\Psi_{\gamma}^{1,0})$ solves the $R$-rescaled Hitchin equation, i.e.  if and only if $e^{\gamma^\dagger}h_{\text{app},R}\,e^\gamma$ is the $R$-harmonic for $(\mathcal E,\varphi)$.

    \begin{claim}\label{claim1}
        $\mathcal N_R$ is well-defined.
    \end{claim}
    \begin{proof}[Proof of Claim \ref{claim1}]
        Given $\gamma\in L_{2-\eps}^{2,2}(i\End_0(\mathcal E,h_0))$, we must show that $\mathcal N_R(\gamma)$ lies in $L_{-\eps}^{0,2}(\Omega^{1,1}\otimes i\End_0(\mathcal E,h_0))$.  This amounts to showing $\int_{B_\delta(p_i)}|r^\eps\mathcal N_R(\gamma)|^2<\infty$ for $i=1,\dots,n$.
        \noindent 
        Fix $i\in\{1,\dots,n\}$ and pick a trivialization of $\mathcal E$ such that $x_1\in\langle e_1\rangle$.  We expand this around $\gamma=0$, i.e. we consider the error from $h_{\text{app}, R}$ on $B_{\delta}(p_i)$:
        \begin{equation}\label{eqn: HE expanded}
            \mathcal N_R(0)
                =h_{\text{app},R}^{1/2}\left(F_{\nabla_{\text{app},R}}^\perp+R^2\sum_{j,k}\frac{[\varphi_j^{\dagger_{h_{\text{app},R}}},\varphi_k]}{(\bar{z-p_j})(z-p_k)}\de\bar z\wedge\de z\right)h_{\text{app},R}^{-1/2}.
        \end{equation}
        The construction of $h_{\text{app},R}$ ensures that the summand $j=k=i$ cancels with $F_{\nabla_{\text{app},R}}^\perp$ on $B_\delta(p_i)$.  The set of remaining indices can be decomposed as $J\sqcup\bar J\sqcup J_0$ depending on our fixed choice of $i$: 
        \begin{align*}
        \bar J&=\{(j,i)\mid j\ne i\}, \\
        J&=\{(i,k)\mid k\ne i\}, \\
        J_0&=\{(j,k)\mid j,k\ne i\}.
        \end{align*}
        Furthermore, we can decompose the residues $\varphi_j=\varphi_{j,(-1)}+\varphi_{j,(0)}+\varphi_{j,(1)}$ where
            $$\varphi_{j,(-1)}=\begin{pmatrix}0&0\\ *&0\end{pmatrix},\qquad \varphi_{j,(0)}=\begin{pmatrix}*&0\\0&*\end{pmatrix},\quad\varphi_{j,(1)}=\begin{pmatrix}0&*\\0&0\end{pmatrix}.$$
            Because of our choice of trivialization, 
        $\varphi_i=\varphi_{i,(1)}$ 
        $[\varphi_{j,(-1)}^\dagger,\varphi_i]=0$, and 
        \begin{equation}
        h_{\text{app},R}|_{B_{\delta}(p_i)}=\sqrt{h_{\det}}\begin{pmatrix} \lambda_{\text{loc},i,R}& 0 \\ 0 & \lambda_{\text{loc},i,R}^{-1}\end{pmatrix}.
        \end{equation}
        Consequently, 
            $$\varphi_{j,(-1)}^{\dagger_{h_{\text{app},R}}}=\lambda_{\text{loc},i,R}^{-2}\varphi_{j,(-1)}^\dagger,\qquad
            \varphi_{j,(0)}^{\dagger_{h_{\text{app},R}}}=\varphi_{j,(0)}^\dagger,\qquad
            \varphi_{j,(1)}^{\dagger_{h_{\text{app},R}}}=\lambda_{\text{loc},i,R}^2\varphi_{j,(1)}^\dagger.$$
        We calculate the contribution of the terms in \eqref{eqn: HE expanded} coming from $J,\bar J,J_0$:
        \begin{align*}
            &T_{\bar J}=
                R^2\sum_k\frac{[\varphi_i^{\dagger_{h_{\text{app},R}}},\varphi_k]}{(\bar{z-p_i})(z-p_k)}
                =R^2\sum_k\frac{\lambda_{i,R}^2[\varphi_i^\dagger,\varphi_k]}{(\bar{z-p_i})(z-p_k)}
            \\
            &T_J=
                R^2\sum_j\frac{[\varphi_j^{\dagger_{h_{\text{app},R}}},\varphi_i]}{(\bar{z-p_j})(z-p_i)}
                =R^2\sum_j\frac{[\varphi_{j,(0)},\varphi_i]+\lambda_{i,R}^2[\varphi_{j,(1)}^\dagger,\varphi_i]}{(\bar{z-p_j})(z-p_i)}\\
            &T_{J_0}=
                R^2\sum_{j,k\ne i}\frac{[\varphi_j^{\dagger_{h_{\text{app},R}}},\varphi_k]}{(\bar{z-p_j})(z-p_k)}
                =R^2\sum_{j,k\ne i}\frac{\lambda_{i,R}^{-2}[\varphi_{j,(-1)},\varphi_k]+[\varphi_{j,(0)},\varphi_k]+\lambda_{i,R}^2[\varphi_{j,(1)}^\dagger,\varphi_k]}{(\bar{z-p_j})(z-p_k)}
        \end{align*}
        Let $T=T_{\bar J}+T_J+T_{J_0}$, and note then that the equation in \eqref{eqn: HE expanded} is             \begin{align*}
            \mathcal N_R(0)
                &=h_{\text{app},R}^{1/2}\left(F_{\nabla_{\text{app},R}}^\perp+R^2\sum_{j,k}\frac{[\varphi_j^{\dagger_{h_{\text{app},R}}},\varphi_k]}{(\bar{z-p_j})(z-p_k)}\de\bar z\wedge\de z\right)h_{\text{app},R}^{-1/2}\\
                &=h_{\text{app},R}^{1/2}\left(F_{\nabla_{\text{app},R}}^\perp+ R^2  \frac{[\varphi_i^{\dagger_{h_{\text{app},R}}},\varphi_i]}{(\bar{z-p_i})(z-p_i)}\de\bar z\wedge\de z + R^2 T\de\bar z\wedge\de z\right)h_{\text{app},R}^{-1/2}\\
                            &=h_{\text{app},R}^{1/2}\left(R^2 T\de\bar z\wedge\de z\right)h_{\text{app},R}^{-1/2}\\
        \end{align*}
        Now we want to show that $\mathcal{N}_R(0) \in L^{0,2}_{-\eps}(\Omega^{1,1} \otimes i \End_0(\mathcal{E}, h_0))$. Let $r=|z-p_i|$.  Note from \eqref{eqn: KW local model}\footnote{
            In particular, the numerator is $2\beta r^{2 r \beta}$ and the denominator is $|x|^2-|y|^2r^{4R\beta}>2\beta$.
        }
        that $\lambda_{\text{loc},i,R}\leq r^{2R\beta_i}$ on $B_\delta(p_i)$, which leads to the key observation that $|T^{\dagger_{h_{\text{app},R}}}|\in O(r^{-1+4R\beta_i})$ and $|T|\in O(r^{-1})$, which is easily verified by checking one component $T_{\bar J},T_J,T_{J_0}$ at a time.  These two facts and the boundedness of $h_{\app,R}$ imply that there is some $C>0$ independent of $R$ such that $\tr(T^{\dagger_{h_{\app,R}}}T)<Cr^{-2+4R\beta_i}$.  Then
        \begin{align} \int_{B_\delta(p_i)}||z-p_i|^\eps T|_{h_{\app, R}}^2
            &=
            \frac1{2i}\int_{B_\delta(p_i)}r^{2\eps}\tr(T^{\dagger_{h_{\text{app},R}}}T)\,\de\bar z\wedge\de z \nonumber \\
                &<\frac1{2i}C\int_{B_\delta(p_i)}R^4r^{2\eps-2+4R\beta_i}\de\bar z\wedge\de z \nonumber \\
                    &=2\pi C\frac{R^3}{2\eps+4\beta_i}\delta^{2\eps+4R\beta_i}M<\infty.
                    \label{eqn: L2 norm of T}
        \end{align}

        We have reduced the claim to showing that  $\mathcal N_R(\gamma)-\mathcal N_R(0)\in L_{-\eps}^{0,2}(i\End_0(\mathcal E,h_0))$.  For this, we adapt the method of proof used in \cite[Proposition 6.1]{FMSW21}.  The difference \[F_{\nabla_\gamma}^\perp - F_{\nabla_{\app, R}}^\perp\] is schematically 
        \[ B_1(\gamma) \nabla_{\app, R}^{0,1} \nabla_{\app, R}^{1,0} \gamma+ B_2(\gamma)\nabla_{\app, R}^{0,1} \gamma+ B_3(\gamma)\nabla_{\app, R}^{1,0} \gamma\]
        with coefficients which are smooth functions of $\gamma$. 
        Our goal is to show that each of these summands lies in $L^{0,2}_{-\eps}$, by individually considering $\nabla_{\app, R}^{0,1} \nabla_{\app, R}^{1,0} \gamma$, $\nabla_{\app, R}^{0,1} \gamma$, and $\nabla_{\app, R}^{1,0} \gamma$.
        We use the fact  $|\nabla_{\app,R}-\de|<C_1r^{-1}$ 
        on $B_\delta(p_i)$ for some $C_1>0$. Now
        \begin{align*}
            \del_{\overline{z}} 
                &= r^{-1} \left( \frac{e^{i\theta}}{2} r \del_r + \frac{i}{2} \del_\theta \right), \\
            \del_{z}
                &= r^{-1} \left( \frac{e^{-i\theta}}{2} r \del_r - \frac{i}{2} \del_\theta \right), 
        \end{align*}
        are both $r^{-1}$ times a $b$-vector field, so $\nabla_{\app, R}^{0,1} \nabla_{\app, R}^{1,0} \gamma \in L^{0,2}_{-\eps}$, 
        $\nabla_{\app, R}^{0,1} \gamma \in L^{1,2}_{1-\eps}$, 
        $\nabla_{\app, R}^{1,0} \gamma \in L^{1,2}_{1-\eps}$.
        Since the coefficients $B_i(\gamma)$ are bounded in $L^\infty$, each of these summands lie in $L^{2}_{-\eps}$.

We expand the difference:
\begin{align} &\left[(\Psi_\gamma^{1,0})^\dagger,\Psi_\gamma^{1,0}\right] - \left[(\Psi_{\app, R}^{1,0})^\dagger,\Psi_{\app, R}^{1,0}\right]  \nonumber \\
&= \e^{-\gamma} \Psi_{\app, R}^{1,0}\e^{2\gamma} \wedge \Psi_{\app, R}^{0,1 }\e^{-\gamma} - \e^{\gamma} \Psi_{\app, R}^{0,1 }\wedge \e^{-2\gamma}\Psi_{\app, R}^{1,0}\e^{\gamma} - \Psi_{\app, R}^{1,0}\wedge \Psi_{\app, R}^{0,1 }\nonumber \\
&=  (\e^{-\gamma}-1) \Psi_{\app, R}^{1,0}\e^{2\gamma} \wedge \Psi_{\app, R}^{0,1 }\e^{-\gamma} - (\e^{\gamma}-1) \Psi_{\app, R}^{0,1 }\wedge \e^{-2\gamma}\Psi_{\app, R}^{1,0}\e^{\gamma} \nonumber  \\
& \qquad +  \Psi_{\app, R}^{1,0}(\e^{2\gamma}-1) \wedge \Psi_{\app, R}^{0,1 }\e^{-\gamma} -  \Psi_{\app, R}^{0,1 }\wedge (\e^{-2\gamma}-1)\Psi_{\app, R}^{1,0}\e^{\gamma} \nonumber \\
& \qquad +  \Psi_{\app, R}^{1,0} \wedge \Psi_{\app, R}^{0,1 }(\e^{-\gamma}-1)-  \Psi_{\app, R}^{0,1 }\wedge \Psi_{\app, R}^{1,0}(\e^{\gamma}-1)
\end{align}
We can show that each of these six summands are in $L^{0,2}_{-\eps}$. 
Each term has a term $\e^{c \gamma} -1$ for some constant $c$.
 We use the embedding $L^{2,2}_{2-\eps} \hookrightarrow L^\infty$ since $2-\eps>0$ to get a bound $\e^{c \gamma} -1 \in \mathcal{O}(r^{2-\eps})$.  We use the $L^\infty$ norm on the other $\e^{c' \gamma}$ terms appearing.
 Then since  $|\Psi^{1,0}_{\app,R}|<C_2r^{-1+4R\beta_i}$  for some $C_2$, in total is order $r^{(-2 + 8R\beta_i)+(2-\eps)}=r^{8R \beta_i - \eps}$.
 This embeds in $L^{0,2}_{-\eps}$.
   \end{proof}
    
    The construction of $h_{\text{app},R}$, \Cref{lemma: annulus gluing}, and the fact that $h_{\text{app},R}\to h_0$ uniformly on $K_{2\delta}=C\sm\bigcup B_{2\delta}(p_i)$ as $R\to0$ implies $\lim_{R\to0}\mathcal N_R(0)\to0$ in the $L^2$ sense on $K_\delta$.  The Fr\'echet derivative at $\gamma=0$ is
    \begin{align}
        \de\mathcal N_R\big|_0(\dot\gamma)
            &=\,\nabla^{0,1}\nabla^{1,0}\dot\gamma-\nabla^{1,0}\nabla^{0,1}\dot\gamma
                +R^2\left(\left[[\dot\gamma,\Psi^{1,0}]^\dagger,\Psi^{1,0}\right]+\left[(\Psi^{1,0})^\dagger,[\dot\gamma,\Psi^{1,0}]\right]\right)
                    \notag\\
            &=2\delbar\del\dot\gamma-[F_{\nabla_{\app,R}}^\perp,\dot\gamma]
                +R^2\left(\left[[\dot\gamma,\Psi^{1,0}]^\dagger,\Psi^{1,0}\right]+\left[(\Psi^{1,0})^\dagger,[\dot\gamma,\Psi^{1,0}]\right]\right)
    \end{align}
    At $R=0$ we recover the Laplacian
        $$d\mathcal N_0\big|_0=2 \delbar \del \from L_{2-\eps}^{2,2}(i\End_0(\mathcal E,h_0))\to L_{-\eps}^{0,2}(\Omega_C^{1,1}\otimes i\End_0(\mathcal E,h_r)).$$
    This operator is Fredholm with index 0, and its kernel is 0 because of the decay constraints on $\dot\gamma\in L_{2-\eps}^{2,2}(i\End_0(\mathcal E,h_0))$ at $D$ and the fact that $C=\C\P^1$ has genus 0.  Thus, the analytic implicit function theorem \cite{KrantzPark} provides the desired family $\gamma_R$ for $R\in(0,R_{\max})$. \end{proof}

\subsection{First Variation of Higgs Bundle Data}\label{sec:firstvar}
In order to prove \Cref{thm: HK degeneration}, we need to evaluate the pullback $\mathcal T_R^*g_R$ on a deformation $(\dot\bx,\dot\by)\in T_{(\bx,\by)}\mathcal X(\vec\beta)$.  This is done by \emph{pushing forward} $(\dot\bx,\dot\by)$ to a deformation in $\mathcal M_R(\vec\alpha(R))$ of $\mathcal T(\bx,\by)$.

Given a unitary hyperpolygon $(\bx, \by)$ let $(\mathcal{E}, \varphi, h_R)$ be the associated $R$-harmonic bundle. Given a unitary deformation $(\dot \bx, \dot \by) \in T_{(\bx, \by)}\mathcal{X}(\vec \beta)$, let $\mathtt{H}_R$ be the harmonic representative of the tangent space described in \eqref{eq:harmtan}. 
Then, as described in \Cref{sec:mapontan}, and particularly in \eqref{eq:hk expression},
\begin{align}
    \|\de \mathcal{T}_R(\dot \bx, \dot \by)\|_{g_R}^2
        &=\int_{\P^1}R^{-1}|\delbar\dot\nu_R|_{h_R}^2+R|\dot\varphi+[\dot\nu_R,\varphi]|_{h_R}^2
            \notag\\
        &=\frac1{2i}\sum_{i=1}^n\left(-\lim_{\delta'\to0}\int_{\partial B_{\delta'}(p_i)}\langle\dot\nu_R,\delbar\dot\nu_R\rangle_{h_R}\right)+\int_{\C\P^1}\langle\dot\varphi,\dot\varphi+[\dot\nu_R,\varphi]\rangle_{h_R}.
\end{align}

Equipped with $h_{\text{app},R}$ from \Cref{subsec: analytic expansion of the harmonic metric}, we now produce an approximation of the first variation $\dot\nu_{\text{app},R}$ by gluing our local models $\dot\nu_{\text{loc},i,R}$ from \Cref{def: local first variation} on the disks $B_{2\delta}(p_i)$ with the zero section everywhere else:
    $$\dot\nu_{\text{app},R}=\sum\chi_i\dot\nu_{\text{loc},i,R}$$
where the $\chi_i$ are bump functions equal to 1 on $B_\delta(p_i)$ and 0 outside $B_{2\delta}(p_i)$.
\begin{lemma}\label[lemma]{lemma: approximate nu dot accuracy}
    There is an analytic family (in $R$) of correction terms $\rho_R\in L_{\eps-2}^{2,2}(\End(\mathcal E))$ such that $\dot\nu_{\text{app},R}+\rho_R$ solves the complex Coulomb gauge equation for all $R$ in some interval $(0,R_{\max})$ and $\lim\limits_{R\to0}\rho_R=0$.
\end{lemma}
\begin{proof}
In \Cref{prop:nuflags}, we introduced $\dot{\nu}_{\text{flags}}$ and the decomposition
\[ \dot\nu_{\text{app},R} = \dot{\nu}_{\text{flags}} + \dot{\gamma}_{\text{cor}, R}. \]
Note that the expression for $\dot{\nu}_{\text{app}, R}$ which is written near $p_i$ in a special frame is related to the expression for $\dot{\nu}_{\text{flags}}$ by simply evaluating $\dot\nu_{\app,R}(p_i)$.
We'll borrow the notation from \Cref{def: local first variation} and drop the subscript $i$, instead writing $x=(x_1,0)^\intercal$ and $y=(0,y_2)$ in the appropriate frame.  Then using the formulas \eqref{eqn: nu dot component formulas} and \eqref{eqn: nu flag simple form}, the components of $\dot\gamma_{\text{cor},R}=\dot\nu_{\app,R}-\dot\nu_{\text{flags}}$ are
\begin{align*}
    (\dot\gamma_{\text{cor},R})_{11}
        &=\frac{\bar x_1\dot x_1}{|x|^2-|y|^2r^{4R\beta}}(r^{4R\beta}-1)+\frac{\bar x_1\dot x_1}{|x|^2}\\
        &=\frac{|x|^2\bar x_1\dot x_1(r^{4R\beta}-1)}{|x|^4-|x|^2|y|^2r^{4R\beta}}+\frac{(|x|^2-|y|^2r^{4R\beta})\bar x_1\dot x_1}{|x|^4-|x|^2|y|^2r^{4R\beta}}\\
        &=\frac{(|x|^2+|y|^2)\bar x_1\dot x_1r^{4R\beta}}{|x|^4-|x|^2|y|^2r^{4R\beta}}\\
        &\in O(r^{4R\beta})\\
    (\dot\gamma_{\text{cor},R})_{21}
        &=(\lambda_{\text{loc},R}^2-1)\frac{\dot x_2}{x_1}+\frac{\dot x_2}{x_1}\in O(r^{4R\beta})\\
            (\dot\gamma_{\text{cor},R})_{12}&=0
\end{align*}
It is a straightforward adaptation of \cite[Proposition 3.8]{CFW24} that the usual residue map at $p_i$ extends to 
\[\dot{\varphi} -[\dot{\nu}_{\text{app}, R}, \varphi] \]
and that $[\dot{\gamma}_{\mathrm{cor}}, \varphi]$ does not contribute to the residue map;
consequently, all contributions to the residue map come from $\dot{\nu}_{\text{flags}}$, so following the discussion in \Cref{prop:nuflags}, we conclude
    \[\mathrm{res}_{p_i} \left(\dot\varphi+[\dot\nu_{\text{app},R},\varphi]\right) \in\End_0(E_{p_i},F_i).\]

    \bigskip
  
    Consider the operator $\mathcal N_R^\C\from L_{2-\eps}^{2,2}(\End(\mathcal E))\to L_{-\eps}^{0,2}(\End(\mathcal E))$ given by\footnote{
        Along the lines of \Cref{prop: symplectic quotient tangent space}, this map is a kind of complexification of the map $\mathcal N_R$ from \Cref{prop: approximate h accuracy}, presented in complex gauge.
    }
    \begin{equation}
        \mathcal N_R^\C(\rho)=\del^{h_R}\delbar(\dot\nu_{\text{app},R}+\rho)+R^2[\varphi^{\dagger_{h_R}},\dot\varphi+[(\dot\nu_{\text{app},R}+\rho),\varphi]]
    \end{equation}
    Using the definition of $h_R=e^{\gamma_R^\dagger}h_{\text{app},R}e^{\gamma_R}$, the construction of $\dot\nu_{\text{app},R}$, and \Cref{prop: local Coulomb gauge condition}, it is straightforward to verify this map is well-defined; the argument is very similar to the proof of \Cref{prop: approximate h accuracy}.  The linearization at $\rho=0$ is
    \begin{equation}
        \de\mathcal N_R^\C\big|_0(\dot\rho)=\del^{h_R}\delbar\dot\rho+R^2[\varphi^{h_R},[\dot\rho,\varphi]].
    \end{equation}
    At $R=0$ we again recover the Laplacian $\de\mathcal N_0^\C\big|_0(\dot\rho)=\del\delbar\dot\rho$, which has Fredholm index 0 as a mapping between these weighted $b$-Sobolev spaces.  Hence, we obtain the desired analytic family $\rho_R$ by the analytic implicit function theorem \cite{KrantzPark}.
\end{proof}

\subsection{Proof of Main Theorem}\label{sec:mainthm}
We now prove the main theorem:
\begin{restatable}{theorem}{HKdegeneration}\label[theorem]{thm: HK degeneration}
    Suppose $\vec\beta\in(0,\infty)^n$ is generic and define $\vec\alpha(R)\in(0,\frac12)^n$ by $\alpha_i(R)=\frac12-R\beta_i$.  For $R>0$ small enough that $W_{[n]}(\vec\alpha(R))>(n-2)/2$, let $\mathcal T_R\from\mathcal X(\vec\beta)\into\mathcal M_R(\vec\alpha)$ be the family of embeddings described above.  Then the family of hyperkähler metrics $\mathcal{T}_R^*(g_R)$ converges pointwise to $2\pi\cdot g_\text{HP}$ as $R \to 0$.
\end{restatable}
\begin{proof}     Let $\delta>0$.
    Let $(\bx,\by)$ be a unitary hyperpolygon and $(\dot\bx,\dot\by)$ a unitary deformation of $(\bx,\by)$.  The image $(\mathcal E,\varphi)=\mathcal T_R(\bx,\by)$ is a Higgs bundle with holomorphic structure $\delbar_{\mathcal E}=\delbar$ independent of $(\bx, \by)$, while the parabolic flag structure $\mathcal F$ depends on $(\bx, \by)$.  
        It suffices to prove that for every unitary deformations $(\dot \bx, \dot \by) \in T_{(\bx, \by)} \mathcal{X}(\vec \beta)$ and associated harmonic deformation 
        $\mathtt{H}_R \in T_{\mathcal{T}_R(\bx, \by)} \mathcal{M}_R(\vec(\alpha))$,
    \begin{equation}\label{eqn: main goal}
        \lim_{R\to0}\| \mathtt{H}_R\|_{g_R}^2=2\pi\sum_i\left(|\dot x_i|^2+|\dot y_i|^2\right).
    \end{equation}
    Recall that the deformation $(\dot\bx,\dot\by)$ corresponds to the Higgs bundle deformation  $(\dot{\eta}, \dot{\mathcal F},\dot\varphi)$. 
    It will be convenient to introduce the notation $\dot{\Phi} = \dot \Psi^{1,0}$, 
    so that 
    the associated deformation in the $h_R$-unitary formulation is $(\dot \nabla_R^{0,1},\dot\Phi_R)=(-\delbar\dot\nu_R,\dot\varphi+[\dot\nu_R,\varphi])$ as discussed in \Cref{sec:hitchinhk}.
Using \eqref{eq:hk expression} and the regularity of $\gamma_R$,
    \begin{align}
        \|(\dot \nabla^{0,1}_R,\dot\Phi_R)\|^2
                  &=\sum_{i=1}^n \left( \lim_{\delta'\to0}\oint_{\partial B_{\delta'}(p_i)}R^{-1}\langle\dot\nu_R,\delbar\dot\nu_R\rangle_{h_R}\right) +\int_{\P^1}R\langle\dot\varphi,\dot\Phi_R\rangle_{h_R}
                \notag\\
            &=\sum_{i=1}^n \left(  \lim_{\delta'\to0}\oint_{\partial B_{\delta'}(p_i)}R^{-1}\langle\dot\nu_R,\delbar\dot\nu_R\rangle_{h_{\app,R}} \right)+\int_{\P^1}R\langle\dot\varphi,\dot\Phi_{\app,R}\rangle_{h_{\app,R}}+O(R^2)
                \notag\\
            &=\sum_{i=1}^n \left( \lim_{\delta'\to0}\oint_{\partial B_{\delta'}(p_i)}R^{-1}\langle\dot\nu_{\text{app},R},\delbar\dot\nu_{\text{app},R}\rangle_{h_{\text{app},R}} \right)\\
                &\quad+\int_{\P^1}R\langle\dot\varphi,\dot\Phi_{\text{app},R}\rangle_{h_{\text{app},R}}+O(R^2)
                \notag\\
                &\quad+\sum_{i=1}^n \left( \lim_{\delta'\to0}\oint_{\partial B_{\delta'}(p_i)}\underbrace{R^{-1}\langle\rho_R,\delbar\dot\nu_{\text{app},R}\rangle_{h_{\text{app},R}}+R^{-1}\langle\dot\nu_{\text{app},R},\delbar\rho_R\rangle_{h_{\text{app},R}}}_{f_\text{cor}(\bar{z-p_i})^{-1}\de\bar z} \right).
           \end{align}
    The last integrand can be written as $f_\text{cor}(\bar{z-p_i})^{-1}\de\bar z$.  Since $\rho_R\in L_{{2-\eps}}^{2,2}(\End(\mathcal E))$, we have $\rho_R\in o(r^{1-\eps})$ and $\delbar\rho_R\in o(r^{-\eps})$.  Thus, $|f_\text{cor}|$ has leading order terms of the form $r^{1-\eps-4R\beta_i}$, so for sufficiently small $R$ we can bound it above by $Cr^{1/2}$ for some constant $C$;
    this constant is uniform as $R\to0$ by analyticity of $\rho_R$.  Now $Cr^{1/2}(\bar{z-p_i})^{-1}\de\bar z$ vanishes under integration along $\partial B_{\delta'}(p_i)$ in the limit $\delta'\to0$, so the last integral is zero.  Finally, the uniform boundedness of $\langle\dot\varphi,\dot\Phi_{\text{app},R}\rangle_{h_R}$ on $[0,R_{\max})\times K_\delta$ and \Cref{prop: local metric pairing} implies \eqref{eqn: main goal}.
\end{proof}

\subsection{\texorpdfstring{Comparison of $\U(1)$-Moment Maps}{Comparison of U(1)-Moment Maps}}\label[section]{sec:Morse}
The result in this section is a corollary of \Cref{thm: HK degeneration}, but we provide a direct proof here as an enlightening demonstration not requiring the technical details from \Cref{subsec: Local First Variations}.

Recall the Morse-Bott functions in \eqref{eqn: Morse-Bott functions} arising as moment maps with respect to the real symplectic form for the $\U(1)$-actions on the hyperk\"ahler spaces $\mathcal X(\vec\beta)=X\fourslash_{(0, \vec \beta)}$ and $\mathcal M_R(\vec\alpha(R))$:
    $$M_\text{HP}(\bx,\by)=\frac i2\sum_i|y_i|^2,
    \quad M_R(\mathcal E,\varphi,h_R)=\frac i2\int_{\C\P^1}R|\varphi|_{h_R}^2$$
    In \Cref{thm: Morse functions agree at fixed points}, we proved that $M_R(\mathcal{T}_R(\bx, \by))=2 \pi M_\text{HP}(\bx, \by)$ at corresponding $U(1)$-fixed points. Now, we prove that the family of Morse functions converges: 
\begin{theorem}
    \label[theorem]{thm: moment maps agree}
    In the setting of \Cref{thm: HK degeneration}, for any unitary hyperpolygon $(\bx,\by)$ and corresponding $(0,R_{\max})$-family of harmonic bundles $\mathcal{T}_R(\bx, \by)$, we have \begin{equation} \lim_{R\to0}M_R(\mathcal{T}_R(\bx,\by))=2\pi M_\text{HP}(\bx,\by).\end{equation}
\end{theorem}
\begin{proof}
    Let $(\bx,\by)$ be a unitary hyperpolygon, and let $\gamma_R$, $R_\text{max}$, and $\eps\in(0,\frac12)$ be as in \Cref{prop: approximate h accuracy}.  Our approximation $h_{\text{app},R}$ is uniformly bounded on $K_\delta=\CP^1 \sm \bigcup_{i} B_\delta(p_i)$ as $R\to0$, and therefore, so is $h_R=e^{\gamma_R}\cdot h_{\text{app},R}$ (recall $\lim_{R\to}\gamma_R=0$ and $\gamma_R$ is real analytic in $R$). 
    Thus, $R|\varphi|_{h_R}^2\to0$ on $K_\delta$ as $R\to0$.  
    
    On the ball $B_\delta(p_i)$ and in coordinates with $x_i \in \langle e_1\rangle$, we have
        \begin{align}\tr(\varphi^{\dagger_{h_R}}\varphi)&=\tr((e^{\gamma_R}\varphi e^{-\gamma_R})^{\dagger_{h_{\text{app},R}}}(e^{\gamma_R}\varphi e^{-\gamma_R}))\nonumber  \\ 
        &=\tr(\varphi^{\dagger_{h_\text{app,R}}}\varphi)+R\cdot O(r^{-1-\eps})\,\de\bar z\wedge\de z 
        \end{align}
    where the last equality follows from \Cref{prop: approximate h accuracy}. 
Now we further expand the term $\tr(\varphi^{\dagger_{h_\text{app,R}}}\varphi)$ on $B_\delta(p_i)$:
    \begin{align}
        \tr(\varphi^{\dagger_{h_{\text{app},R}}}\varphi)
            &=-i\sum_{j,k}\frac{\tr(h_{\text{app},R}^{-1}\varphi_j^\dagger h_{\text{app},R}\varphi_k)}{(\bar{z-p_j})(z-p_k)}\,\de\bar z\wedge\de z
                \notag\\
            &=-i\left(\frac{\lambda_{\text{loc},i,R}^2|\varphi_i|^2}{|z-p_i|^2}+\sum_{(j,k)\ne(i,i)}\frac{\tr(h_{\text{app},R}^{-1}\varphi_j^\dagger h_{\text{app},R}\varphi_k)}{(\bar{z-p_j})(z-p_k)}\right)\de\bar z\wedge\de z.
    \end{align}
We now compute the following integral:
    \begin{align}
    &\lim_{R \to 0} \int_{B_{\delta}(p_i)} R |\varphi|^2_{h_R}   \nonumber \\
    & \quad =\lim_{R \to 0} \int_{B_{\delta}(p_i)}\!\! -iR\left(\frac{\lambda_{\text{loc},i,R}^2|\varphi_i|^2}{|z-p_i|^2}+\!\!\sum_{(j,k)\ne(i,i)}\!\!\frac{\tr(h_{\text{app},R}^{-1}\varphi_j^\dagger h_{\text{app},R}\varphi_k)}{(\bar{z-p_j})(z-p_k)} + R\cdot O(r^{-1 - \eps})\right)\de\bar z\wedge\de z
        \label{eqn: norm of phi expansion}
    \end{align}
    The terms in the sum with $j\neq i $ and $k \neq i$ have no singularity at $p_i$. In the remaining cases where either $j=i$ or $k=i$, we have $\tr(h_{\app,R}^{-1}\varphi_j^\dagger h_{\app,R}\varphi_k)=\lambda_{\text{loc},i,R}^2\tr(\varphi_j^\dagger\varphi_k)\in O(r^{4R\beta_i})$. Thus, the $(j,k) \neq (i,i)$ summands in  \eqref{eqn: norm of phi expansion} have singularities at worst $O(r^{-1+4R\beta_i})$, hence $O(r^{-1-\eps})$.
      Since 
        $$\lim_{R\to0}\int_{B_\delta(p_i)}Rr^{-1-\eps}\de\bar z\wedge\de z
            =2\pi i\lim_{R\to0}\frac{R\delta^{1-\eps}}{1-\eps}
                =0,$$
    only the first in \eqref{eqn: norm of phi expansion} survives the limit $R \to 0$ after integration.
    By construction, 
    $\lambda_{\text{loc},i,R}$ satisfies $R^{-1}\delbar\del\log\lambda_{\text{loc},i,R}=R\lambda_{\text{loc},i,R}^2|\varphi_i|^2|z-p_i|^{-2}\,\de\bar z\wedge\de z$ (see \Cref{prop: delbar del log lambda}), so the above implies
    \begin{align}
        \lim_{R\to0}M_R(\mathcal E,\varphi,h_R)
            &=\lim_{R\to0}\frac i2\sum_i\int_{B_\delta(p_i)}R|\varphi|_{h_R}^2
                \notag\\
            &=\lim_{R\to0}\frac12\sum_i\int_{B_\delta(p_i)}R\frac{\lambda_{\text{loc},i,R}^2|\varphi_i|^2}{|z-p_i|^2}\,\de\bar z\wedge\de z
                \notag\\
            &=\lim_{R\to0}\frac12\sum_i\int_{B_\delta(p_i)}R^{-1}\delbar\del\log\lambda_{\text{loc},i,R}
                \notag\\
            &=\pi i\sum_i|y_i|^2
                \notag\\
            &=2\pi M_\text{HP}(\bx,\by)
    \end{align}
    where the second last equality uses \Cref{lemma: integral of dd log lambda}.

\end{proof}

\subsection{Further Directions}

\begin{enumerate}
\item When one defines $\mathcal{X}(\vec \beta)$, one typically sets the values of the complex moment map $\mu_{\C^\times, i} =0$. This ensures that $\mathcal{X}(\vec \beta)$ has a $\C^\times$-action. However, it is not necessary.  Similarly, one can generalize from strongly parabolic Higgs bundles (with nilpotent residues) to weakly parabolic Higgs bundles (with residues with possible non-zero eigenvalues $\pm m_i$). As we remarked in \Cref{rem:momentmaps}, one can view these non-zero eigenvalues of the Higgs field as choice of value of the complex map.  We conjecture that our result extends to this case. One would first have to prove that the map $\mathcal{T}$ preserves stability and holomorphic structures, since this has not been established. 
This is work in progress of AY.
\item We additionally conjecture that the same result holds for hyperpolygons with more general double star-shaped quivers \cite[Figure 2b]{RS} and strongly parabolic Higgs bundles on $\CP^1$ with matching general flag types.   
One would again first have to prove that the map $\mathcal{T}$ preserves stability and the holomorphic symplectic structure, since this has not been established. The proof that $\mathcal{T}$ preserves stability for appropriate weights and $\mathcal{T}$ preserves holomorphic structures is early work in progress of AY. 

\item One would like to extend the above result to weakly parabolic Higgs bundles, but there is some subtlety about the definition of these moduli spaces in the not full flag type. We have not yet thought deeply about this.

\item We wonder to what extent this type of result extends to comet-shaped quivers \cite[Figure 4]{RS} and parabolic Higgs bundles on higher genus Riemann surfaces. 
\end{enumerate}

\bibliography{Main}
\bibliographystyle{math}

\end{document}